\documentclass[12pt,a4paper]{article}

\usepackage[spanish]{babel}
\usepackage[utf8]{inputenc}
\usepackage[T1]{fontenc}

\usepackage{graphicx}
\usepackage{epstopdf}

\usepackage{amsmath}
\usepackage{amssymb}
\usepackage{amsthm}
\usepackage{MnSymbol}
\usepackage{tensor}
\usepackage{bm}
\usepackage{tikz-cd}
\usepackage[makeroom]{cancel}
\usepackage[shortlabels]{enumitem}

\def\undertilde#1{\mathord{\vtop{\ialign{##\crcr
$\hfil\displaystyle{#1}\hfil$\crcr\noalign{\kern1.5pt\nointerlineskip}
$\hfil\tilde{}\hfil$\crcr\noalign{\kern1.5pt}}}}}

\usepackage{hyperref}

\usepackage{float}
\floatstyle{boxed}
\restylefloat{figure}

\theoremstyle{plain}
\newtheorem{teo}{Teorema}
\newtheorem{lema}{Lema}
\newtheorem{prop}{Proposición}
\newtheorem{corol}{Corolario}

\theoremstyle{definition}
\newtheorem{defn}{Definición} 
\newtheorem{ejmp}{Ejemplo}
\newtheorem{obs}{Observación}

\usepackage{vmargin}
\usepackage{verbatim} 

\setpapersize{A4}
\setmargins{3.17cm}       
{2.5cm}                        
{15cm}                      
{22.00cm}                    
{10pt}                           
{1cm}                           
{0.5cm}                             
{3cm}                           

\newcommand{\HRule}{\rule{\linewidth}{0.5mm}}

\numberwithin{equation}{subsection}

\begin{document}

\begin{titlepage}
\begin{center}


\HRule \\[0.4cm]
{ \huge \bfseries Álgebras y grupos de Clifford, espinores algebraicos y aplicaciones a la física \\[0.4cm] }

\HRule \\[1.5cm]

\noindent
\large
Marcos R. A. Arcodía\footnote{marcodia@mdp.edu.ar}\\
~\\
{{\large\textbf{RESUMEN}}}
\end{center}
En el presente trabajo se introduce el álgebra de Clifford asociada a un espacio cuadrático, utilizando técnicas de álgebra universal y teoría algebraica de formas cuadráticas. También se definen los grupos de Clifford y los grupos Pin y Spin asociados a estas álgebras y se estudia la relación existente entre estos grupos y las isometrías del espacio cuadrático de base. Por último se introducen los espacios de espinores algebraicos como ideales minimales izquierdos sobre estas álgebras (equivalentemente representaciones irreducibles) y se mencionan algunas aplicaciones a la física.
~\\
~\\
~\\
~\\
~\\
~\\
~\\
~\\
~\\
\end{titlepage}

\newpage
\pagenumbering{Roman} 
\setcounter{page}{3}
\thispagestyle{empty} 

\newpage
\tableofcontents

\newpage
$\ $
\thispagestyle{empty} 

\newpage
\section{Prefacio}
Estas notas surgieron en el marco de un curso de postgrado en Estructuras Algebraicas que tomé en como parte del Doctorado en Ciencias Físicas en la Facultad de Ciencias Exactas y Naturales de la Universidad Nacional de Mar de Plata. En dicho curso tuve como tarea final elaborar una presentación sobre álgebras de Clifford basándome en el survey escrito por Jean Gallier titulado \emph{Clifford algebras, Clifford groups and a generalization of the quaternions: The Pin and Spin groups} \cite{Gallier}. A partir de dicha presentación preparé estas notas que fui extendiendo a lo largo del tiempo.

Quisiera agradecer a las docentes Sonia Trépode y Ana Clara García Elsener por haberme permitido explorar estos temas en el curso, y en particular a Ana por haber respondido siempre a mis consultas ayudándome a comprender mucho mejor lo estudiado.

\newpage
\section{Introducción}
\pagenumbering{arabic}
Las álgebras de Clifford son álgebras asociativas reales que surgen en distintas ramas de la física y que además resulta un campo interesante desde un punto de vista puramente matemático. En las ciencias físicas estas álgebras se introducen directamente a partir de generadores y relaciones, lo que permite encontrar una representación matricial de forma bastante directa. Si bien este tratamiento es muy práctico en física y simplifica los cálculos mediante la utilización de matrices, hay aspectos interesantes sobre el álgebra de Clifford que quedan ocultos. Más aún, las álgebras de Clifford dan lugar a conjuntos que son isomorfos a $\mathbb{R}$, $\mathbb{C}$ y $\mathbb{H}$ con lo cual, una forma de pensar a las mismas es como generalizaciones de estos conjuntos.

Un álgebra de Clifford está asociada a un espacio cuadrático, es decir, un espacio vectorial real con una forma cuadrática (o equivalentemente una forma bilineal) definida en él. Utilizando la existencia del álgebra tensorial para todo espacio vectorial, se definirá al álgebra de Clifford como un álgebra cociente de la misma. Este tratamiento proporciona gran claridad en cuanto a la relación entre las álgebras de Clifford y las transformaciones de simetría en el espacio cuadrático original. La relación entre el álgebra de Clifford y los grupos de isometrías se da a través de los grupos Pin y Spin que se introducen en el capítulo 3, donde el teorema de Cartan-Dieudonné juega un papel fundamental en las demostraciones de los teoremas. En todas estas notas las principales referencias son \cite{Gallier,VazdaRocha} y las fuentes citadas allí.

Otros objetos relacionado a las álgebras de Clifford son los espinores. Estos objetos surgen dentro de la física en la mecánica cuántica y la teoría de Dirac como entes que representan el estado cuántico de un sistema. En estas notas los introduciremos a través de representaciones lineales del álgebra de Clifford.

El trabajo se divide en tres secciones. La primera de ellas dedicada a brindar información preliminar necesaria para teoría de Clifford, esto es: la teoría de formas cuadráticas, isometrías y teorema de Cartán-Dieudonne (se sigue la referencia \cite{Guccione}), la teoría de álgebras reales y en particular álgebra tensorial y cocientes desde un punto de vista de estructuras algebraicas (donde utilicé los apuntes de la materia, los libros \cite{Artin}, \cite{Rotman} y el survey \cite{Gallier}) y cuestiones de representaciones lineales de álgebras y grupos (\cite{VazdaRocha,Artin,Hazewinkel}).

En la segunda parte se define el álgebra de Clifford y se analizan varios ejemplos. Luego se procede a establecer la relación mencionada entre álgebra de Clifford y las isometrías del espacio cuadrático a través de los grupos Pin y Spin y las representaciones adjunta y adjunta torcida del denominado grupo de Clifford. En esta sección se siguen a las referencias \cite{Gallier} y \cite{Lounesto}.

La última sección está dedicada a la introducción de los espinores algebraicos y mencionar aplicaciones en física. Para esto se define un tipo nuevo de representaciones del álgebra de Clifford. Se sigue principalmente la referencia \cite{VazdaRocha} y en menor medida \cite{Hazewinkel}.

\newpage
\section{Preliminares}
\subsection{Espacios cuadráticos, isometrías y teorema de Cartan-Dieudonné}\label{espacioscuad}

Para definir un álgebra de Clifford se necesita un espacio cuadrático, es decir un espacio vectorial con una forma bilineal simétrica definida en él. En esta sección introduciremos las definiciones asociadas a espacios vectoriales reales cuadráticos. Se sigue \cite{Guccione} como principal referencia.

\begin{defn}\label{esp-cuad}
Un \emph{espacio cuadrático real} es un espacio vectorial sobre $\mathbb{R}$, $\mathbb{V}$, provisto de una forma bilineal simétrica. Es decir una aplicación $\varphi{:}\mathbb{V}\times\mathbb{V}\rightarrow\mathbb{R}$ que cumple:
\begin{itemize}
\item $\varphi(a,b)=\varphi(b,a)$  $\forall a,b\in\mathbb{V}$\quad{(simetría)}
\item $\varphi(\lambda{a}+c,b)=\lambda\varphi(a,b)+\varphi(c,b)$\quad$\forall a,b,c\in\mathbb{V}${,}\;$\forall \lambda\in\mathbb{R}$\quad{(linealidad)}
\end{itemize}
Observemos que la linealidad en el primer argumento, junto con la simetría implican la bilinealidad.

Se introduce la \emph{forma cuadrática asociada}, $\Phi$, una aplicación de $\mathbb{V}$ en $\mathbb{R}$, definida por:
\begin{center}
{$\Phi(a)=\varphi(a,a)${,}}\\
\end{center}
que cumple: $\Phi(\lambda{a})=\lambda^2\Phi(a)\quad\forall\lambda\in\mathbb{R}{,}\;\forall{a}\in\mathbb{V}$.
\end{defn}
 Notemos que vale la siguiente igualdad:
\begin{equation}\label{bilin-cuad}
\varphi(a,b)=\frac{1}{2}(\Phi(a+b)-\Phi(a)-\Phi(b)),
\end{equation}
con lo cual se puede hablar de espacio cuadrático refiriendose indistintamente a la existencia de $\Phi$ o a la de $\varphi$.
Se definen además, de acuerdo a el valor que les asigna $\Phi$, tres tipos de vectores.
\begin{defn}\label{clasif-vectores}
Sea $a$ un vector de $\mathbb{V}$ distinto de cero, decimos que:
\begin{itemize}
\item$a$ es \emph{isotrópico} o \emph{de tipo luz} si $\Phi(a)=0$,
\item$a$ es \emph{de tipo tiempo} si $\Phi(a)>0$,
\item$a$ es \emph{de tipo espacio} si $\Phi(a)<0$.
\end{itemize}
A su vez un vector de tipo tiempo o de tipo espacio se dice \emph{anisotrópico}.
\end{defn}
La denominación tipo tiempo, luz y espacio tiene su origen en la teoría de relatividad especial de Einstein, donde se distinguen las trayectorias de rayos de luz (vectores tipo luz), de partículas (tipo tiempo) y de hipotéticos objetos que se moverían a velocidades superiores a la de la luz (tipo espacio). No hay una única convención en la literatura y las definiciones de vectores tipo espacio y tiempo suelen aparecer intercambiadas.

Notemos que si $\mathbb{V}$ es de dimensión finita, entonces dada una base podemos representar a la forma bilineal por una matriz cuadrada. En efecto, sea $B=\{e_{1},...,e_{n}\}$ una base de $\mathbb{V}$, entonces dos vectores $u$ y $v$ se pueden escribir como:
\begin{center}
$u=\sum_{i=1}^{n}u^{i}e_{i}$\;, $v=\sum_{i=1}^{n}v^{i}e_{i}$.
\end{center}
Por la bilinealidad de $\varphi$ se cumple que:
\begin{equation}\label{forma-matriz}
\varphi(u,v)=
\begin{pmatrix}
 u_{1} \ u_{2} \ \hdots \ u_{n} \\
\end{pmatrix}.
\begin{pmatrix}
   \ \ \ \varphi_{11} & \ \ \varphi_{12} & \ \ \hdots & \ \ \varphi_{1n} \ \ \\
         \ \ \varphi_{21} & \ \ \varphi_{22} & \ \ \hdots & \ \ \vdots \ \ \\
        \ \ \vdots & \ \ \vdots & \ \ \ddots & \ \ \vdots \ \ \\
        \ \ \varphi_{n1}  & \ \ \varphi_{n2} & \ \ \hdots & \ \ \varphi_{nn} \ \
\end{pmatrix}.
\begin{pmatrix}
	v_{1} \\
    	v_{2} \\
   	\vdots \\
         	v_{n} \\
\end{pmatrix}
\quad ,
\end{equation}
donde $\varphi_{ij}=\varphi(e_{i},e_{j})$.
Es un resultado conocido que todo espacio cuadrático admite una base ortogonal, es decir, una base $\{\tilde{e_{1}},...,\tilde{e_{n}}\}$, que cumple: $\varphi(\tilde{e_{i}},\tilde{e_{j}})=0\quad\forall{i\neq{j}}$. Dada tal base, la matriz de $\varphi$ es diagonal y contiene en su diagonal un número $p$ de números reales positivos, un número $q$ de reales negativos y un número $s$ de ceros. Como la dimensión del espacio es $n$, se tiene que cumplir $p+q+s=n$.

En el apéndice \ref{apA} se prueba que las cantidades $p$, $q$ y $s$ son invariabtes de la forma bilineal. Esto es, si pasamos a otra base ortogonal la matriz de $\varphi$ en dicha base seguirá teniendo el mismo número de elementos positivos, negativos y ceros en la diagonal. Esto permite relacionar los números $p$, $q$ y $s$ directamente con la forma bilineal y establecer la siguiente definición:
\begin{defn}\label{signat}
Dada una forma bilineal real $\varphi$ llamamos al par $(p,q)$ la \emph{signatura} de $\varphi$ y al número $s$ la \emph{nulidad} de $\varphi$
\end{defn}

Son de especial interés tanto matemáticamente como físicamente las formas bilineales no degeneradas.

\begin{defn}\label{no-degen}
Decimos que una forma bilineal $\varphi$ es \emph{no degenerada} si se cumple:
\begin{equation}
\varphi(u,v)=0\;\forall{v}\in\mathbb{V}\rightarrow{u}=0.
\end{equation}
En cuyo caso se dice que el espacio cuadrático es \emph{regular}.
\end{defn}

\begin{prop}\label{nodegen-nulidad}
$\varphi$ es no degenerada si y sólo si su nulidad es cero.
\end{prop}
\begin{proof}
$(\Rightarrow{)}$Probaremos el contrarrecíproco, es decir que si la nulidad es distinta de cero entonces $\varphi$ es degenerada.\\
Sea $s$ la nulidad de $\varphi$, entonces existe una base ortogonal donde la matriz de $\varphi$ tiene $s$ ceros en su diagonal. Podemos reordenar la base para que esos elementos sean los primeros $s$ vectores de la base. En dicho caso tendremos:
\begin{equation}
\varphi_{ii}=0\;\forall{i}\in\{1,...,s\}.
\end{equation}
Sea $u$ el vector que en dicha base tiene componentes $(\underbrace{1,...,1}_\text{s-veces},\underbrace{0,...,0}_\text{(n-s)-veces})$ y $v$ cualquier vector de componentes $(v_{1},..,v_{n})$ tenemos:
\begin{equation}
\begin{split}
\varphi(u,v)&=
\begin{pmatrix}
 1 \ \hdots \ 1 \, \ 0 \ \hdots \ 0 \\
\end{pmatrix}.
\begin{pmatrix}
   \ 0 & \ \  & \  & \ & \  & \ & \\
         \  & \ \ddots & \ & \ & \ \text{\huge{0}} & \ & \\
        \  & \ & \  0 & \ & \ & \ & \\
         \ & \ & \  & \ \varphi_{\small{{s+1,s+1}}} & \ & \ & \\
         \  & \ \text{\huge{0}} & \ & \ & \  \ddots \ &\\
        \  & \ & \ & \ & \ & \ &\varphi_{nn} \ \
\end{pmatrix}.
\begin{pmatrix}
	v_{1} \\
    	v_{2} \\
   	\vdots \\
         	v_{n} \\
\end{pmatrix}=\\
&=\underbrace{1\times{0}\times{v_{1}}+\hdots+1\times{0}\times{v_{s}}}_{\text{s-términos}}+\underbrace{0\times\varphi_{s+1,s+1}\times{v_{s+1}}+...+0\times\varphi_{n,n}\times{v_{n}}}_{\text{(n-s)-términos}} =0.
\end{split}
\end{equation}
Hemos encontrado un vector distinto de cero que es ortogonal a todos los vectores del espacio, por lo tanto $\varphi$ es degenerada.\\
($\Leftarrow$) Se puede ver fácilmente que si la nulida es cero, entonces es una forma bilineal no degenerada.
\end{proof}
\begin{obs}\label{nulidad-signat}
Si el espacio cuadrático es regular, entonces la signatura $(p,q)$ cumple $p+q=n$. Cuando la forma bilineal es definida positiva se tiene que $q=0$ y si es definida negativa $p=0$.
\end{obs}
A continuación analizaremos las isometrías lineales del espacio vectorial $\mathbb{V}$. $\text{Aut}{(\mathbb{V})}$ será el espacio de isomorfismos lineales de $\mathbb{V}$ en sí mismo.
\begin{defn}\label{isometria}
Sea $f\in\text{Aut}{(\mathbb{V})}$ decimos que $f$ es una \emph{isometría} respecto del espacio cuadrático $(\mathbb{V},\Phi)$ si se cumple $\Phi(f(u))=\Phi(u)\;\forall{u}\in\mathbb{V}$ , o equivalentemente si se cumple  $\varphi(f(u),f(v))=\varphi(u,v)\;\forall{u,v}\in{\mathbb{V}}$.

Se puede ver que las isometrías forman un grupo. Al mismo lo denotaremos por O($\Phi$) y lo llamamos \emph{el grupo de isometrías respecto de $\Phi$}. Denotamos por SO($\Phi$) al subgrupo de O($\Phi$) compuesto por las isometrías que tienen determinante positivo.
\end{defn}

Notemos que uno podría pensar que las isometrías respecto de la forma bilineal $\phi$ no son exactamente las mismas que las de la forma cuadrática $\Phi$, sin embargo, en virtud de la ecuación \ref{bilin-cuad} las isometrías de una y de la otra son exactamente las mismas.

Un caso particular de isometrías son las reflexiones respecto de hiperplanos. Sea $x$ un vector anisotrópico cualquiera, $H_{x}=x^{\perp}:=\{y\in\mathbb{V}|\;\varphi(x,y)=0\}$ define un subespacio lineal de $\mathbb{V}$ de dimensión $n-1$. Llamamamos a dicho subespacio \emph{hiperplano normal a x}.

\begin{defn}\label{defn:6}
Sea $H_{x}$ un hiperplano con vector normal $x$ anisotrópico, la \emph{reflexión $s_{x}$ respecto de $H_{x}$}, $s_{x}\in{O(\Phi)}$ está dada por:
\begin{equation}
\begin{split}
s_{x}&:\mathbb{V}\longrightarrow\mathbb{V}\\
s_{x}(u)=u-&\frac{2\varphi(u,x)}{\Phi(x)}x\quad\forall{u}\in\mathbb{V}.
\end{split}
\end{equation}
$s_{x}\notin{\text{SO}(\Phi)}$ pues $\det(s_{x})=-1$. Notemos que $s_{\lambda{x}}=s_{x}\;\forall{\lambda\nequal{0}}$
\end{defn}

Habiendo definido la reflexión respecto de un hiperplano, podemos enunciar el teorema de Catan-Dieudonné, de gran importancia en la geometría y la teoría algebraica de formas cuadráticas. La demostración del mismo no será expuesta dado que se requieren elementos de teoría algebraica de formas cuadráticas y excede los contenidos del presente documento (una demostración puede encontrarse en \cite{Guccione}). Este teorema es indispensable para demostrar resultados importantes en el contexto de álgebras de Clifford y los grupos Pin y Spin. 
\begin{teo}\label{teo-cd}
Sea $(\mathbb{V},\Phi)$ un espacio cuadrático regular de dimensión $n$, entonces cada isometría distinta de la identidad $f\in\text{O}(n)$ puede escribirse como composición de a lo sumo $n$ reflexiones respecto de hiperplanos.
\end{teo}

Una consecuencia importante de este teorema es el siguiente corolario.
\begin{corol}\label{grupo-so}
Cada isometría tiene determinante $+1$ ó $-1$. Las isometrías de determinante $+1$ forman un subgrupo $\text{SO}(\Phi)$ de índice 2 de $\text{O}(\Phi)$. Este subgrupo es el núcleo del morfismo de grupos $\det{:}\text{O}(\Phi)\longrightarrow\{-1,1\}$ y se llama \emph{grupo especial ortogonal de $\Phi$}.
\end{corol}

\subsection{Álgebras sobre un cuerpo, álgebra tensorial}

Esta sección estará dedicada a la introducción de las estructuras algebraicas que utilizaremos, esto es álgebra sobre un cuerpo y el álgebra tensorial de un espacio vectorial.

\begin{defn}\label{algebra-sobre-cuerpo}
Dado un cuerpo $\mathbb{K}$, una \emph{$\mathbb{K}$-álgebra} o un \emph{álgebra sobre $\mathbb{K}$} es un espacio vectorial $A$ sobre $\mathbb{K}$, junto con una operación bilineal interna $\cdot\,{:}\,A\times{A}\rightarrow{A}$ que hace de $A$ un anillo con unidad $1_{A}$.

Dadas dos $\mathbb{K}$-álgebras A y B, un \emph{morfismo de $\mathbb{K}$-álgebras}, $h\,{:}\,A\rightarrow{B}$ es una transformación lineal entre los espacios vectoriales $(A,+)$, $(B,+)$ que además es morfismo de anillos entre $(A,\cdot)$, $(B,\cdot)$ y que cumple $h(1_{A})=1_{B}$
\end{defn}

\begin{ejmp}\label{ej-algebras-sobre-cuerpos}
Los siguientes son ejemplos de álgebras sobre algún cuerpo:
\begin{itemize}
\item El anillo de matrices cuadradas de $n\times{n}$ sobre un cuerpo $\mathbb{K}$ es una $\mathbb{K}$-álgebra.
\item El \emph{álgebra de los cuaternios, $\mathbb{H}$}, es un álgebra sobre $\mathbb{R}$:
Sea $Q_{8}:=\{\pm{\bold{1}},\pm{\bold{i}},\pm{\bold{j}},\pm{\bold{k}}\}$ el grupo de cuaternios, visto como subgrupo finito del grupo de matrices cuadradas inversibles complejas de dimensión 2, donde:
\begin{equation}
\bold{1}= \begin{pmatrix}
		1 \ & 0\\
		 0 \ & 1\\
	       \end{pmatrix}
;\quad
\bold{i}= \begin{pmatrix}
		i \ & 0\\
		 0 \ & -i\\
	       \end{pmatrix}
;\quad
\bold{j}= \begin{pmatrix}
		0 \ & 1\\
		 -1 \ & 0\\
	       \end{pmatrix}
;\quad
\bold{k}= \begin{pmatrix}
		0 \ & i\\
		 i \ & 0\\
	       \end{pmatrix}.
\end{equation}
Sea $\mathbb{H}$ el conjunto definido por $\mathbb{H}:=\{a\bold{1}+b\bold{i}+c\bold{j}+d\bold{k}\,|\, a,b,c,d\in\mathbb{R}\}$, entonces $\mathbb{H}$ es un álgebra sobre $\mathbb{R}$. 

Claramente $\mathbb{H}$ es espacio vectorial sobre $\mathbb{R}$, puesto que el conjunto $\{\bold{1},\bold{i},\bold{j},\bold{k}\}$ conforma una base. Además, como el álgebra de matrices sobre $\mathbb{C}$ ya tiene estructura de anillo, para ver que $\mathbb{H}$ es álgebra, hay que ver que el producto en este conjunto es una operación cerrada. En efecto, se cumplen las igualdades:
\begin{equation}
\begin{split}
\bold{1}^2=&1{;}\\
\bold{i}^2=\bold{j}^2 =\bold{k}^2=&\bold{i}\bold{j}\bold{k}=-1{;}\\
\bold{i}\bold{j}=-\bold{j}\bold{i}&=\bold{k}{;}\\
\bold{i}\bold{k}=-\bold{k}\bold{i}&=\bold{j}{;}\\
\bold{j}\bold{k}=-\bold{k}\bold{j}&=\bold{i}{,}\\
\end{split}
\end{equation}
por lo que el producto es cerrado en $\mathbb{H}$.
\end{itemize}
\end{ejmp}

\begin{defn}\label{ideal}
Un \emph{ideal} $\mathcal{I}$ de una $\mathbb{K}$-álgebra $A$ es un subespacio vectorial de $A$ que además es un ideal bilátero de $A$ con la estructura de anillo dada por la multiplicación, es decir $a\mathcal{I}\subset{\mathcal{I}}, \mathcal{I}a\subset{\mathcal{I}}\;\forall{a}\in{A}$

El \emph{ideal generado por el conjunto $X$, $\langle{X}\rangle$}, es el menor ideal del álgebra (respecto de la inclusión) que contiene al conjunto $X$. Puede probarse que dado un conjunto $X$ dicho ideal está dado por:
\begin{equation}
\langle{X}\rangle:=\{a_1x_1b_1+...+a_kx_kb_k\;|\; a_i,b_i\in{A}\;;\;{x}_i\in{X}\;\forall{i\in\{1,...,k\}}\; \text{con }k\in{\mathbb{N}}\}
\end{equation}
\end{defn}

\begin{defn}\label{algsimple}
Diremos que una $\mathbb{K}$-álgebra $A$ es \emph{simple} si $A^2=\{ab\ | \ a\in{A}, \ b\in{A}\}\neq\{0\}$ y sus únicos ideales biláteros son los triviales, es decir $\{0\}$ y $A$. Además una $\mathbb{K}$-álgebra es \emph{semisimple} si se escribe como suma directa de subálgebras simples.
\end{defn}

\begin{defn}\label{cocientealgebraideal}
Dada un álgebra $A$ sobre un cuerpo $\mathbb{K}$ y un ideal $\mathcal{I}$ de $A$ definimos la siguiente relación de equivalencia en $A$. Decimos que dos elementos $a$ y $b$ en $A$ son equivalentes y notamos $a\sim{b}$ si $a-b\in\mathcal{I}$. Se puede ver fácilmente que $\sim$ es una relación de equivalencia en el álgebra $A$.

Definimos la clase de equivalencia del elemento $a$, $\overline{a}$ como el conjunto de todos los elementos de $A$ que son equivalentes a $a$, esto es, $\overline{a}=\{b\in{A}|a-b\in\mathcal{I}\}=a+\mathcal{I}$. Denotamos a la familia de todas las clases de equivalencia por $A/{\mathcal{I}}$ y lo llamamos el \emph{cociente de $A$ módulo el ideal $\mathcal{I}$}.
\end{defn}

\begin{prop}\label{estructuracociente}
Dada un álgebra $A$ sobre un cuerpo $\mathbb{K}$ y un ideal $\mathcal{I}$, el conjunto cociente $A/\mathcal{I}$ tiene esctructura de álgebra sobre el mismo cuerpo $\mathbb{K}$ y la función $\pi{:}\,{A}\rightarrow{A/\mathcal{I}}$ que asocia a cada elemento su clase de equivalencia, es un morfismo suryectivo de álgebras, con $\ker{\pi}=\mathcal{I}$.
\end{prop}

\begin{proof}
Es un resultado conocido que si $A$ es un anillo e $\mathcal{I}$ es un ideal bilátero, entonces $A/\mathcal{I}$ es un anillo y $\pi$ es morfismo suryectivo de anillos con $\ker{\pi}=\mathcal{I}$. Por otro lado, también se sabe que si $A$ es un espacio vectorial sobre $\mathbb{K}$ e $\mathcal{I}$ es un subespacio de $A$, entonces el cociente $A/I$ tiene estructura de espacio vectorial sobre $\mathbb{K}$ y $\pi$ es transformación lineal suryectiva con $\ker(\pi)=\mathcal{I}$. Notemos que como $\pi$ es suryectiva cualquier elemento en el cociente se puede expresar como $\overline{a}=\pi(a)$ con $a\in{A}$. Las operaciones $+$, $\cdot$ y el producto por escalar $\star$ en $A/\mathcal{I}$ son:
\begin{equation}
\begin{split}
\overline{a}\cdot\overline{b}=&\pi(ab)=ab+\mathcal{I}\quad\forall{\overline{a},\overline{b}}\in{A/\mathcal{I}}\\
\overline{a}+\overline{b}=&\pi(a+b)=a+b+\mathcal{I}\quad\forall{\overline{a},\overline{b}}\in{A/\mathcal{I}}\\
\lambda\star{\overline{a}}=&\pi(\lambda{a})=\lambda{a}+\mathcal{I}\quad\forall{\overline{a}}\in{A/\mathcal{I}}\;y\;\forall{\lambda}\in{\mathbb{K}}\\
\end{split}
\end{equation}
Se puede ver que estas operaciones están bien definidas, es decir que no dependen de los representantes elegidos. Por ejemplo, si $a\sim{a'}$, $b\sim{b'}$, entonces $\overline{a'}\cdot\overline{b'}=\overline{a}\cdot\overline{b}$. Todo esto prueba que $A/\mathcal{I}$ tiene estructura de álgebra y que $\pi$ es morfismo de anillos y morfismo de espacios vectoriales (transformación lineal) suryectivo, con $\ker{\pi}=\mathcal{I}$. Para ver que $\pi$ es morfismo de álgebras resta ver que $\pi(1_{A})=1_{A/\mathcal{I}}$.

Tenemos que $\overline{a}\cdot\pi(1_{A})=\pi(a)\cdot\pi(1_{A})=\pi(a1_{A})=\pi(a)=\overline{a}$, entonces efectivamente $\pi(1_{A})$ es el neutro respecto de $\cdot$ en $A/\mathcal{I}$
\end{proof}

A continuación se introducirán algunos elementos del álgebra tensorial que serán necesarios para la definición de álgebras de Clifford. Este apartado pretende ser simplemente una introducción. Para un análisis más detallado se puede consultar por ejemplo el capítulo 2 de \cite{Rotman}.

\begin{defn}\label{prodtensor}
Dados dos espacios vectoriales $E$ y $F$ sobre un cuerpo $\mathbb{K}$, un \emph{producto tensorial de $E$ y $F$} es un par $(E\otimes{F},\otimes)$ donde $E\otimes{F}$ es un $\mathbb{K}$-espacio vectorial y $\nobreak{\otimes\,{:}\,{E\times{F}}\rightarrow{E\otimes{F}}}$ es una función bilineal tal que para todo $\mathbb{K}$-espacio vectorial $G$ y toda función bilineal $f\,:\,E\times{F}\rightarrow{G}$ existe una única transformación lineal $f_{\otimes}\,:\,E\otimes{F}\rightarrow{G}$ tal que $f=f_{\otimes}\circ{\otimes}$. Esto es $f(u,v)=f_{\otimes}(u\otimes{v})$ o equivalentemente, el siguiente diagrama es conmutativo:
\[
\begin{tikzcd}[sep=5em]
E\times{F} \arrow{r}{\otimes} \arrow{rd}{f} 
  & E\otimes{F} \arrow[dashrightarrow]{d}{\exists{!}f_{\otimes}} \\
    & G
\end{tikzcd}
\]
Hemos definido a $E\otimes{F}$ a partir de una propiedad universal, con lo cual el mismo queda determinado a menos de isomorfismos.
\end{defn}

\begin{obs}\label{tensorprod-propiedades}
Notemos que:
\begin{itemize}
\item Los elementos $u\otimes{v}$ con $u\in{E}$ y $v\in{F}$ generan a $E\otimes{F}$.
\item Cuando pueda ser confuso deberá escribirse $E\otimes_{\mathbb{K}}{F}$
\item Se tienen los isomorfismos naturales:
\begin{equation}
E\otimes(F\otimes{G})\cong{(E\otimes{F})\otimes{G}}\quad{y}\quad E\otimes{F}\cong{F\otimes{E}}
\end{equation}
\item Si además $E$ y $F$ son álgebras sobre $\mathbb{K}$ se puede dotar a $E\otimes{F}$ también de una estructura de álgebra. La manera de hacerlo no es única. Un caso puede ser, definir el producto entre dos generadores (es decir, elementos de la forma $u\otimes{v}$) de la siguiente manera:
\begin{equation}
(a\otimes{b})\cdot{(c\otimes{d})}=(ac\otimes{bd})\quad{a,c\in{E}}\quad{b,d\in{F}}
\end{equation}
Como veremos más adelante, hay otros productos que se pueden definir en el producto tensorial que lo convierten en un álgebra.
\end{itemize}
\end{obs}

\begin{obs}\label{tensorprod-existencia}
Se puede ver que el espacio $E\otimes{F}$ existe para cualesquiera $E$ y $F$ espacios vectoriales sobre un cuerpo $\mathbb{K}$. Para ello basta con tomar al espacio vectorial ($\mathbb{R}$-módulo libre) que tiene por base a $E\times{F}$, por ejemplo el conjunto $\nobreak{\mathbb{R}^{E\times{F}}:=\{f\,:\,{E\times{F}}\rightarrow{\mathbb{R}}\;|\;\text{$f$ es función}\}}$ cumple con esto. Sea $H$ dicho conjunto, definimos $S$ como el subespacio de $H$ generado por los elementos de la base de la forma:
\begin{equation}
\begin{split}
(a,b+b')-&(a,b)-(a,b'),\\
(a+a',b)-&(a,b)-(a',b),\\
(ar,b)&-(a,rb),\\
\end{split}
\end{equation}
con $a,a'\in{E}$, $b,b'\in{F}$ y $r\in{\mathbb{K}}$. Se puede ver que tomando:
\begin{equation}
\begin{split}
E\otimes{F}&:=G/S\\
\otimes\,:\,E\times{F}\rightarrow&{E\otimes{F}}=G/S\\
(a,b)\mapsto{\pi}&(a,b)=(a,b)+S,
\end{split}
\end{equation}
se satisface la propiedad universal.
Un tratamiento más detallado puede encontrarse en \cite{Rotman}.
\end{obs}

\begin{obs}\label{pre-tensoralgebra}
Notemos que a partir de esta definición de producto tensorial entre dos espacios vectoriales podemos ``multiplicar'' un espacio vectorial por si mismo e ir obteniendo las sucesivas potencias. Por ejemplo: $\mathbb{V}\otimes\mathbb{V}$\,; $(\mathbb{V}\otimes\mathbb{V})\otimes\mathbb{V}\cong\mathbb{V}\otimes(\mathbb{V}\otimes\mathbb{V})$, etcétera. Definimos $\mathbb{V}^{\otimes^{i}}:=\underbrace{\mathbb{V}\otimes{...}\otimes\mathbb{V}}_{\text{i-veces}}$ para $i\geq1$ y además $\mathbb{V}^{\otimes^{0}}:=\mathbb{K}$. Cuando tomamos la suma directa de todas las potencias de $\mathbb{V}$ se obtiene un conjunto muy particular. Existe en él una estructura de álgebra y se puede ver como el álgebra ``más grande'' que contiene a $\mathbb{V}$. Esto se expresa en la siguiente proposición.
\end{obs}

\begin{prop}\label{tensoralgebra}
Dado un $\mathbb{K}$-espacio vectorial $\mathbb{V}$ existe un álgebra especial $T(\mathbb{V})$ junto con una transformación lineal $i\,:\,\mathbb{V}\rightarrow{T(\mathbb{V })}$ que satisface la siguiente propiedad universal:
Dada cualquier $\mathbb{K}$-álgebra $A$ y una transformación lineal $f\,:\,\mathbb{V}\rightarrow{A}$ existe un único morfismo de álgebras $\bar{f}\,;\,{T(\mathbb{V})}\rightarrow{A}$ tal que $f=\bar{f}\circ{i}$. El siguiente diagrama conmuta:
\[
\begin{tikzcd}[sep=5em]
\mathbb{V} \arrow{r}{i} \arrow{rd}{f} 
  & T(\mathbb{V}) \arrow[dashrightarrow]{d}{\exists{!}\;\bar{f}} \\
    & A
\end{tikzcd}
\]
$T(V)$ se llama el \emph{álgebra tensorial de $\mathbb{V}$}.
\end{prop}

\begin{proof}
Basta con tomar
\begin{equation}
T(V)=\bigoplus\limits_{i\geq{0}}\mathbb{V}^{\otimes^{i}}=\mathbb{K}\oplus\mathbb{V}\oplus\mathbb{V}\otimes\mathbb{V}\oplus...,
\end{equation}
esto es, un $t$ en ${T}(\mathbb{V})$ será de la forma $t=t_0+t_1+...+t_k=\sum_{i=1}^k{t_i}$ con $t_i\in\mathbb{V}^{\otimes^{i}}$ y con k finito.

Como la proposición dice que $T(\mathbb{V})$ es un álgebra debemos definir en ella un producto.
Dados $u_1\otimes{...}\otimes{u_k}\in\mathbb{V}^{\otimes^k}\quad$ y $\quad{v}_1\otimes{...}\otimes{v_l}\in\mathbb{V}^{\otimes^l}$, definimos:
\begin{equation}
(u_1\otimes{...}\otimes{u_k})\cdot(v_1\otimes{...}\otimes{v_l})=u_1\otimes{...}\otimes{u_k}\otimes{v_1}\otimes{...}\otimes{v_l}\;\in\mathbb{V}^{\otimes^{k+l}} , 
\end{equation}
y por bilinealidad de $\otimes$ se generaliza a cualquier elemento en la suma directa $T(\mathbb{V})$.

Notemos que hay inclusiones naturales $i_{k}\,:\mathbb{V}^{\otimes^k}\,\hookrightarrow{T(\mathbb{V})}$. En particular $i_0\,:\,\mathbb{K}\hookrightarrow{T(\mathbb{V})}$ cumple $i_0(1)=1_{T(\mathbb{V})}$
\end{proof}

\begin{obs}\label{tensoralgebra-propiedades}
Se tiene que:
\begin{itemize}
\item El álgebra $(T(\mathbb{V}),\cdot)$ es no conmutativa.
\item Muchas álgebras relacionadas con $\mathbb{V}$ se definen a partir de cocientes de $T(\mathbb{V})$ con ideales. Por ejemplo, el \emph{álgebra exterior $\bigwedge(\mathbb{V})$} se define por $\bigwedge(\mathbb{V})=T(\mathbb{V})/\langle{v}\otimes{v}\rangle$, donde $\langle{v}\otimes{v}\rangle$ es el ideal bilátero generado por elementos de la forma $v\otimes{v}$. Otro caso es el \emph{álgebra simétrica} $\text{Sym}(\mathbb{V})=T(\mathbb{V})/\langle{v}\otimes{w}-w\otimes{v}\rangle$.
\end{itemize}
\end{obs}

\subsection{Representaciones lineales de álgebras y grupos}

En esta sección definiremos el concepto de representación o acción de un grupo sobre un espacio vectorial. Estos conceptos nos serán útiles más adelante, en dos contextos. Por un lado, la acción lineal de un grupo será utilizada en la introducción de los grupos Pin y Spin y por otro, se empleará una representación particular del álgebra de Clifford a la hora de definir los espinores algebraicos.  

\begin{defn}\label{accionlinealgrupo}
Dado un grupo $G$ y un espacio vectorial $\mathbb{V}$, una \emph{acción lineal $\rho$ de $G$ sobre $\mathbb{V}$} es una aplicación:
\begin{equation}
\rho:G\longrightarrow\text{Aut}{\mathbb{V}},
\end{equation}
que satisface
\begin{enumerate}
\item $\rho(gh)=\rho(g)\circ\rho(h)\quad\forall{g,h}\in{G}$
\item $\rho(e)=id_{\mathbb{V}}\quad\text{donde $e$ es el neutro del grupo $G$}$
\end{enumerate}
\end{defn}
\begin{obs}\label{inversoaccion}
Notemos que esta definición implica:
\begin{itemize}
\item Dado $g$ en el grupo, $\rho(g)$ es lineal e inversible puesto que $\rho{(g)}\in{\text{Aut}{\mathbb{V}}}\;\forall{g}\in{G}$
\item $\rho(g^{-1})=\rho(g)^{-1}$
\end{itemize}
\end{obs}

En lo que queda de esta sección se sigue principalmente el capítulo 4 de \cite{VazdaRocha}. Se define representación irreducible de una $\mathbb{R}$-álgebra y se introduce especialmente la representación regular.  

\begin{defn}\label{replinealalgebra}
Sea $A$ un álgebra real y $\mathbb{V}$ un espacio vectorial sobre $\mathbb{K}=\mathbb{R}$ ó $\mathbb{C}$, una aplicación:
\begin{equation}
\rho:A\longrightarrow\text{End}_{\mathbb{K}}{\mathbb{V}},
\end{equation}
que satisface
\begin{enumerate}
\item $\rho(ab)=\rho(a)\circ\rho(b)\quad\forall{a,b}\in{A}$
\item $\rho(1_{A})=id_{\mathbb{V}}\quad\text{donde $1_{A}$ es la unidad del álgebra $A$}$,
\end{enumerate}
 se denomina una \emph{$\mathbb{K}$-representación lineal de $A$} y el espacio $\mathbb{V}$ se llama el \emph{espacio de la representación}. En alguna bibliografía suele llamarse representación al par $(\rho,\mathbb{V})$.
 
 Dos representaciones $\rho_{1}:A\rightarrow\text{End}_{\mathbb{K}}(\mathbb{V}_1)$ y $\rho_{2}:A\rightarrow\text{End}_{\mathbb{K}}(\mathbb{V}_2)$ se dicen equivalentes si existe un $\mathbb{K}$-isomorfismo $\phi:\mathbb{V}_{1}\rightarrow\mathbb{V}_{2}$, tal que $\rho_{2}(a)=\phi\circ\rho_{1}(a)\circ\phi^{-1}$  $\forall{a}\in{A}$.
 
Una representación se dice \emph{fiel} si $\ker{\rho}=\{0\}$ y se dice \emph{irreducible} si se cumple que los únicos subespacios de $\mathbb{V}$ invariantes frente a todos los endomorfismos $\rho(a)$ (o sea $a$ cualquier elemento en $A$) son $\mathbb{V}$ y $\{0\}$. Puesto en símbolos sería que se cumple la siguiente implicación: $\rho(a)(U)\subset{U}$  $\forall{a}\in{A}\ \implies \ U=\mathbb{V} \text{ ó } U=\{0\}$.

Una representación se dice \emph{semisimple} si se cumple que $\mathbb{V}=\mathbb{V}_1\oplus...\oplus\mathbb{V}_m$, con $\mathbb{V}_1,...,\mathbb{V}_m$ subespacios de $\mathbb{V}$ invariantes de $\rho(a)$ para cualquier $a\in{A}$.

\end{defn}

\newpage
\section{Álgebras de Clifford y grupos asociados}
 
Este capítulo estará dedicado al estudio de las álgebra de Clifford de un espacio cuadrático, los grupos que se desprenden de ellas y su relación con los grupos de isometrías del espacio cuadrático.

\subsection{Motivación. Rotaciones en $\mathbb{R}^3$ y cuaternios unitarios}\label{motivacion}

El álgebra de Clifford de un espacio cuadrático es un objeto universal que depende de la forma cuadrática definida en él. Un caso central que motivó el desarrollo de la teoría fue el caso del espacio tridimensional euclídeo. Las isometrías asociadas a la forma cuadrática euclídea son las rotaciones que dejan fijo el origen. Existe una relación estrecha entre este grupo de isometrías y el grupo de cuaternios unitarios.

Si recordamos la definición del álgebra de cuaternios $\mathbb{H}$ dada en el ejemplo \ref{ej-algebras-sobre-cuerpos}, tenemos que un elemento $x\in\mathbb{H}$ se escribe:
\begin{equation}
x=\begin{pmatrix}
	a+ib\quad \ c+id\\
	-c+id \quad \ a-ib\\
    \end{pmatrix}
    =\begin{pmatrix}
	z\quad \ \omega\\
	-\bar{\omega} \quad \bar{z}\\
    \end{pmatrix},
\quad z,\omega\in\mathbb{C},
\end{equation}
donde la barra superior indica conjugación compleja. Si calculamos el determinante de una matriz de este estilo obtenemos:
\begin{equation}
\det(x)=a^2+b^2+c^2+d^2.
\end{equation}
Definimos a los \emph{cuaternios unitarios} y los denotamos por $\mathbb{H}_{1}$ a los $x\in\mathbb{H}$ tales que $\det(x)=a^2+b^2+c^2+d^2=1$.

Podemos además definir una operación adicional en $\mathbb{H}$, la \emph{conjugación}. Dado $x=a\bold{1}+b\bold{i}+c\bold{j}+d\bold{k}$ con $a,b,c,d\in{\mathbb{R}}$ se define el conjugado de $x$, $\bar{x}$ por:
\begin{equation}
\bar{x}=a\bold{1}-b\bold{i}-c\bold{j}-d\bold{k},
\end{equation}
de este modo $\det(x)=x\bar{x}=a^2+b^2+c^2+d^2$.
Se puede ver que el conjunto $\mathbb{H}_{1}$ es un grupo y como tal es isomorfo al grupo $\text{SU}(2)$ de matrices unitarias complejas de $2\times2$ con determinante positivo.
Podemos definir una acción de $\text{SU}(2)\cong\mathbb{H}_{1}$ sobre $\mathbb{R}^3$. Para ello identificamos un elemento $x=(x_1,x_2,x_3)$ en $\mathbb{R}^3$ con un $X$ en $\mathbb{H}$, $X=x_1\bold{i}+x_2\bold{j}+x_3\bold{k}$.
Dado $Z\in\text{SU}(2)\cong\mathbb{H}_{1}$, definimos la acción de $Z$ sobre un punto de $\mathbb{R}^3$ como:
\begin{equation}\label{adjsu2}
\rho_{Z}(x)=Z\cdot{X}\cdot{Z^{-1}}=Z\cdot{X}\cdot\bar{Z}.
\end{equation}
Se puede demostrar (y lo haremos más adelante) que dado $Z\in\text{SU}(2)$, $\rho_{Z}$ es una isometría lineal de $\mathbb{R}^3$, es decir, una rotación. Además, para toda rotación $\Omega$ existe un elemento $Z\in\text{SU}(2)$ tal que $\rho_{Z}=\Omega$.
Más aún, $\rho$ es un homomorfismo suryectivo de grupos $\rho\,:\,\text{SU}(2)\rightarrow\text{SO}(3)$ cuyo núcleo es el conjunto $\{\bold{-1},\bold{1}\}\cong\mathbb{Z}_{2}$.
Notemos que por el teorema del isomorfismo para grupos tenemos que $\text{SO}(3)\cong\text{SU}(2)/\mathbb{Z}_{2}$. Dado $\Omega\in\text{SO}(3)$, existen dos representantes de $\Omega$ en $\text{SU}(2)$, $Z$ y $-Z$. La acción que tiene la forma \ref{adjsu2} se denomina \emph{acción adjunta del grupo $\mathbb{H}_1$}.

Las álgebras de Clifford y los grupos asociados nos permiten generalizar este procedimiento a cualquier espacio cuadrático regular. Esto es, dado un espacio cuadrático regular obtenemos un grupo que se relacione con el grupo especial ortogonal de dicho espacio de la misma manera en que $\text{SU}(2)$ se relaciona con $\text{SO}(3)$.

\subsection{Álgebra de Clifford de un espacio cuadrático real}

\begin{defn}\label{algebradeclifford}
Sea $(\mathbb{V},\Phi)$ un espacio cuadrático real con forma cuadrática $\Phi$ y forma bilineal asociada $\varphi$, un \emph{álgebra de Clifford asociada a $(\mathbb{V},\Phi)$} es un álgebra real $\text{Cl}(\mathbb{V},\Phi)$ junto con una transformación lineal $i_{\Phi}\,:\,\mathbb{V}\rightarrow\text{Cl}(\mathbb{V},\Phi)$ que satisface $\nobreak{(i(v))^2=\Phi(v)\bf{1_{\text{Cl}(\mathbb{V},\Phi)}}\;\forall{v\in{\mathbb{V}}}}$, donde $\bf{1_{\text{Cl}(\mathbb{V},\Phi)}}$ es la identidad del álgebra $\text{Cl}(\mathbb{V},\Phi)$ y tal que para toda álgebra real $A$ y toda transformación lineal $f\,:\,{V}\rightarrow{A}$ que satisface $(f(v))^2=\Phi(v)\bf{1_{\text{Cl}(\mathbb{V},\Phi)}}\;\forall{v\in{\mathbb{V}}}$ existe un único morfismo de álgebras $\bar{f}\,:\,\text{Cl}(\mathbb{V},\Phi)\rightarrow{A}$ tal que $f=\bar{f}\circ{i_\Phi}$. Es decir se cumple el siguiente diagrama:
\[
\begin{tikzcd}[sep=5em]
\mathbb{V} \arrow{r}{i_{\Phi}} \arrow{rd}{f} 
  & \text{Cl}(\mathbb{V},\Phi) \arrow[dashrightarrow]{d}{\exists{!}\;\bar{f}} \\
    & A
\end{tikzcd}
\]
\end{defn}
\begin{obs}\label{algebradeclifford-notacion}
De la definición se deduce que $\text{Cl}(\mathbb{V},\Phi)$ es única a menos de isomorfismos. Para simplificar la escritura suele usarse la notación $\text{Cl}(\Phi)$ o $\text{Cl}$ para $\text{Cl}(\mathbb{V},\Phi)$ e $i$ para $i_\Phi$. En lo sucesivo, siempre que no sea confuso, se utilizará $\bf{1}$ para referirse a la identidad en $\text{Cl}(\mathbb{V},\Phi)$.  El tipo de funciones que cumplen $f(v)^2=\Phi(v)\bf{1}$ se denominan \emph{funciones de Clifford}.
\end{obs}
\begin{prop}\label{algebradeclifford-existencia}
Para todo espacio cuadrático existe un álgebra de Clifford.
\end{prop}
\begin{proof}
Vamos a ver que podemos definirla como un cociente del álgebra tensorial del espacio vectorial $\mathbb{V}$ con un ideal.

Sea $\mathcal{I}$ el ideal generado por los elementos en $T(\mathbb{V})$ de la forma $v\otimes{v}-\Phi(v){\bf{1}}_{T(\mathbb{V})}$, esto es $\mathcal{I}=\langle{\{v\otimes{v}-\Phi(v){\bf{1}}_{T(\mathbb{V})}}|v\in{\mathbb{V}}\}\rangle\overset{\text{not}}{=}\langle{v\otimes{v}-\Phi(v){\bf{1}}_{T(\mathbb{V})}}\rangle$. Sean $i_1$ la inclusión $V\overset{i_1}{\hookrightarrow}T(\mathbb{V})$ y $\pi$ la proyección $T(V)\overset{\pi}{\rightarrow}{T(\mathbb{V})/\mathcal{I}}$, esto es $\pi(a)=a+\mathcal{I}\;\forall{a\in{T(\mathbb{V})}}$. Basta con definir:
\begin{equation}
\begin{split}
\text{Cl}(\Phi)&=T(V)/\mathcal{I}\\
i_{\Phi}&=\pi\circ{i_1}\\
\end{split}
\end{equation}
Por la proposición \ref{estructuracociente} sabemos que $T(\mathbb{V})/\mathcal{I}$ tiene una estructura de álgebra sobre $\mathbb{R}$ bien definida y que además $\pi$ es morfismo suryectivo de álgebras con $\ker{\pi}=\mathcal{I}$. Veremos que dada un álgebra real $A$ cualquiera y una transformación lineal $f:\,\mathbb{V}\rightarrow{A}$ que satisface $(f(v))^2=\Phi(v)1_{A}$, existe un único morfismo de álgebras $\overline{f}:\,T(\mathbb{V})/\mathcal{I}\rightarrow{A}$ que satisface $f=\overline{f}\circ{i_{\Phi}}$

Por la definición de álgebra tensorial sabemos que va a existir un único morfismo de álgebras $g$ tal que el siguiente diagrama es conmutativo
\[
\begin{tikzcd}[sep=5em]
\mathbb{V} \arrow{r}{i_{1}} \arrow{rd}{f} 
  & T(\mathbb{V}) \arrow[rightarrow]{d}{\exists{!}\;g} \\
    & A
\end{tikzcd}
\]

Si agregamos la proyección en el cociente $T(\mathbb{V})/\mathcal{I}$, podemos pensar en el siguiente diagrama:

\[
\begin{tikzcd}[sep=5em]
\mathbb{V} 
  \ar{dr}[swap]{f} 
  \ar{r}{i_1} 
& 
 T(\mathbb{V})
  \ar{d}[swap]{g} 
  \ar{r}{\pi}
& 
 T(\mathbb{V})/\mathcal{I}
  \ar[dashed]{dl}{\overline{f}{?}} \\
& A
\end{tikzcd}
\]
en donde lo que buscamos es una forma de definir un morfismo de álgebras $\bar{f}$ que cumpla $f=\bar{f}\circ{i_{\Phi}}=\overline{f}\circ{\pi}\circ{i_1}$. Para ello utilizaremos el hecho de que $\pi$ es suryectiva. De este modo dado $\overline{a}=a+\mathcal{I}\in{T(\mathbb{V})/\mathcal{I}}$, definimos la función $\overline{f}$ como sigue:
\begin{equation}
\overline{f}(\overline{a})=g(a).
\end{equation}
Tenemos que ver que $\overline{f}$ está bien definida y que cumple el diagrama conmutativo. Sean $a$ y $b$ en $\overline{a}$, para que $\overline{f}(\overline{a})$ sea independiente del representante, entonces necesitaremos que $g(a)=g(b)$.
Como $a\sim{b}$, entonces $a-b\in{I}$, tenemos entonces que:
\begin{equation}
\begin{split}
g(a)-g(b)&=g(a-b)=g\Big{(}\sum_{i=1}^{m}{a_i{(v_i\otimes{v_i}}-\Phi(v_i)\mathbf{1}){b_i}}\Big{)}=\\
&=\sum_{i=1}^{m}\Big{[}{g(a_i)g(v_i\otimes{v_i})g({b_i}})-\Phi(v_i)g(a_i)g(\mathbf{1})g({b_i})\Big{]}=\\
&=\sum_{i=1}^{m}\Big{[}{g(a_i)g(i_1(v_i)i_1(v_i))g({b_i}})-\Phi(v_i)g(a_i)g({b_i})\Big{]}=\\
&=\sum_{i=1}^{m}\Big{[}{g(a_i)g(i_1(v_i))g(i_1(v_i))g({b_i}})-\Phi(v_i)g(a_i)g({b_i})\Big{]}=\\
&=\sum_{i=1}^{m}\Big{[}{g(a_i)f(v_i)^{2}g({b_i}})-\Phi(v_i)g(a_i)g({b_i})\Big{]}=\\
&=\sum_{i=1}^{m}\Big{[}{g(a_i)\big{(}f(v_i)^{2}-\Phi(v_i)\big{)}g({b_i})}\Big{]}=0.\\
\end{split}
\end{equation}
Donde la última igualdad se da porque $f$ cumple la condición $f(v)^2=\Phi(v)\mathbf{1}_{A}$ para todo $v\in{\mathbb{V}}$. Por lo tanto tenemos que $g(a)-g(b)=0$, con lo cual $g(a)=g(b)$ y $\overline{f}$ está bien definida.

Como $g$ es morfismo de álgebras, entonces $\overline{f}$ también lo será. Más aún, como $g$ es el único morfismo de álgebras que cumple $f=g\circ{i_1}$, entonces $\overline{f}$ será el único que cumpla $f=\overline{f}\circ{i_\Phi}$, a menos de isomorfismos.

Hemos encontrado un álgebra $\text{Cl}(\mathbb{V},\Phi)$ y una aplicación $i_{\Phi}$ que cumplen lo deseado, resta ver que si existen dos álgebras $\text{Cl}(\mathbb{V},\Phi)$ y  $\text{Cl}'(\mathbb{V},\Phi)$ con respectivas aplicaciones $i_{\Phi}$ e $i'_{\Phi}$ que satisfacen la propiedad universal son isomorfas.

Por un lado tenemos el siguiente diagrama conmutativo:
\[
\begin{tikzcd}[sep=5em]
\mathbb{V} \arrow{r}{i_{\Phi}} \arrow{rd}{i'_{\Phi}} 
  & \text{Cl}(\mathbb{V},\Phi) \arrow[rightarrow]{d}{\exists{!}\;i} \\
    & \text{Cl}'(\mathbb{V},\Phi)
\end{tikzcd}
\]
donde $i$ es el único morfismo de álgebras que satisface $i'_{\Phi}=i\circ{i_{\Phi}}$. 
Además se cumple:
\[
\begin{tikzcd}[sep=5em]
\mathbb{V} \arrow{r}{i'_{\Phi}} \arrow{rd}{i_{\Phi}} 
  & \text{Cl}'(\mathbb{V},\Phi) \arrow[rightarrow]{d}{\exists{!}\;i'} \\
    & \text{Cl}(\mathbb{V},\Phi)
\end{tikzcd}
\]
donde $i'$ es el único morfismo de álgebras que satisface $i_{\Phi}=i'\circ{i'_{\Phi}}$.
Combinando estos dos diagramas tenemos que:
\begin{equation}
i'_\Phi=i\circ{i'}\circ{i'_{\Phi}}\quad\quad y \quad\quad i_\Phi=i'\circ{i}\circ{i_{\Phi}}
\end{equation}
Por otro lado tenemos que el morfismo identidad satisface el siguiente diagrama conmutativo:
\[
\begin{tikzcd}[sep=5em]
\mathbb{V} \arrow{r}{i_{\Phi}} \arrow{rd}{i_{\Phi}} 
  & \text{Cl}(\mathbb{V},\Phi) \arrow[rightarrow]{d}{id} \\
    & \text{Cl}(\mathbb{V},\Phi)
\end{tikzcd}
\]
Por la propiedad universal, $id$ será el único morfismo que satisfaga el diagrama, con lo cual $i'\circ{i}=id$. Análogamente se ve que $i'\circ{i}=id'$, donde $id'$ es la identidad en $\text{Cl}'(\mathbb{V},\Phi)$. 
Así $i^{-1}=i'$, $i'^{-1}=i$, $i$ es un isomorfismo y $\text{Cl}(\mathbb{V},\Phi)\cong{\text{Cl}'(\mathbb{V},\Phi)}$
\end{proof}

\begin{obs}\label{funcionclifford-propiedades}
Cuando no sea confuso se notará la función $i_{\Phi}$ por $i$. Se tienen las siguientes propiedades:
\begin{itemize}
\item{$i(u)i(v)+i(v)i(u)=\varphi(u,v){\bf{1}}\quad\forall{u,v\in{\mathbb{V}}}$. Esto implica que si $u$ y $v$ son ortogonales, entonces $i(u)i(v)=-i(v)i(u)$}
\item{Supongamos que existe un vector isotrópico $u\neq{0}$ en $\mathbb{V}$, entonces $i(u)^{2}=0$, con lo cual, en general el álgebra de Clifford tendrá divisores del cero.}
\item{$\text{Cl}(\Phi)$ está generada por el conjunto $\{{\bf{1}}, i_{\Phi}(\mathbb{V})\}$. Es decir, todo elemento en el álgebra de Clifford es combinación lineal de productos finitos de elementos en dicho conjunto. Esto se deduce del hecho de que el álgebra tensorial está generada por $\{1, i(\mathbb{V})\}$}
\item{Como todo espacio cuadrático regular admite una base ortogonal, dos álgebras de Clifford de métricas no degeneradas de la misma signatura $(p,q)$ van a ser isomorfas. Es por esto que notaremos directamente $\text{Cl}_{p,q}$, al álgebra de Clifford correspondiente a $\Phi(x_1,...,x_{p+1})=x_1^2+...+x_p^2-(x_{p+1}^2+...+x_{p+q}^2)$ y llamaremos $\text{Cl}_n$ a $\text{Cl}_{0,n}$}
\end{itemize}
\end{obs}

\begin{ejmp}\label{ejemplos-algebrasdeclifford}
Calculemos el álgebra de Clifford para algunos casos
\begin{enumerate}
\item Sea $\mathbb{V}=\mathbb{R}$, como $\mathbb{R}$ espacio vectorial. Sea $\{e_1\}$ base de $\mathbb{R}$ tal que $\Phi(x_1e_1)=-x_1^2$. En ese caso $\Phi(e_1)=-1$, con lo cual $(i(e_1))^2=-{\bf{1}}$. Esta álgebra es $\text{Cl}_1$ por definición y podemos ver que es isomorfa a $\mathbb{C}$, bajo el isomorfismo:
\begin{equation}
\begin{split}
\text{Cl}_1&\rightarrow{\mathbb{C}},\\
{\bf{1}}&\mapsto{1},\\
i(e_1)&\mapsto{i}.\\
\end{split}
\end{equation}
Otra forma de construir el álgebra de Clifford es directamente a partir de su definición, es decir, construir $T(\mathbb{R})$, el ideal $\mathcal{I}=\langle{v\otimes{v}-\Phi(v){\bf{1}}}\rangle$ y construir el álgebra cociente.
Tenemos que $T(\mathbb{R})\cong{\mathbb{R}[x]}$, los polinomios con coeficientes reales en una indeterminada, donde identificamos $x=e_1$. Tenemos que el ideal es:
\begin{equation}
 \mathcal{I}=\langle{x_1^2e_1\otimes{e_1}+x_1^2{\bf{1}}}\rangle=\langle{e_1\otimes{e_1}+{\bf{1}}}\rangle=\langle{x^2+1}\rangle ,
\end{equation}
con lo cual el álgebra de Clifford está dada por:
\begin{equation}\label{Cl1}
\text{Cl}_1=T(\mathbb{R})/\mathcal{I}=\mathbb{R}[x]/\langle{x^2+1}\rangle=\mathbb{R}\overline{1}+\mathbb{R}\overline{x},
\end{equation}
donde la barra indica la proyección de dicho elemento en el cociente. Podemos ver que esto es así, puesto que, dado un polinomio $p(x)\in{\mathbb{R}[x]}$, por el algoritmo de la división en el anillo de polinomios tenemos que:
\begin{equation}
p(x)=(x^2+1)q(x)+r(x)\text{ con $r(x)=0$ o gr$(r(x))\in\{0,1\}$} ,
\end{equation}
al tomar clases de equivalencia en el cociente se tiene:
\begin{equation}
\overline{p(x)}=\overline{(x^2+1)q(x)}+\overline{r(x)}=\overline{0}+\overline{r(x)}=\overline{r(x)}.
\end{equation}
Donde hemos usado que $(x^2+1)q(x)$ está en el ideal. Entonces, la proyección de cualquier polinomio en el cociente es siempre igual a la proyección de algún polinomio de grado a lo sumo uno, con lo cual el álgebra está dada por la ecuación \ref{Cl1}.

Para ver el isomorfismo con los complejos basta con calcular el producto entre dos elementos en el álgebra:
\begin{equation}
(a\overline{1}+b\overline{x})(c\overline{1}+d\overline{x})=ac\overline{1}+bd\overline{x^2}+(ad+bc)\overline{x},
\end{equation}
donde $\overline{x^2}$ se puede escribir como:
\begin{equation}
\overline{x^2}=x^2+\langle{x^2+1}\rangle=x^2-x^2-1+\langle{x^2+1}\rangle=-1+\langle{x^2+1}\rangle={-}\overline{1}.
\end{equation}
De esta manera la ecuación anterior queda:
\begin{equation}
(a\overline{1}+b\overline{x})(c\overline{1}+d\overline{x})=(ac-bd)\overline{1}+(ad+bc)\overline{x},
\end{equation}
lo que imita exactamente al producto entre dos números complejos.

\item Calcularemos el álgebra para el caso en que $\mathbb{V}=\mathbb{R}$ pero la forma cuadrática es definida positiva, es decir $\text{Cl}_{1,0}$. Tomaremos $\{e_1\}$ base ortonormal de $\mathbb{R}$, esto es $\Phi{(x_1e_1)}=x_1^2$.

En este caso tenemos que $\text{Cl}_{1,0}\cong{\mathbb{R}\oplus{\mathbb{R}}}$, esto es el espacio vectorial $\mathbb{R}\times\mathbb{R}$ con el producto definido por:
\begin{equation}\label{productosplitcomplex}
(a,b)(c,d)=(ac,db).
\end{equation}
Este conjunto (que es un álgebra conmutativa con divisores del cero) se conoce con el nombre de \emph{números complejos hiperbólicos} o \emph{números de Study} (también en inglés se conocen como \emph{split-complex numbers}). El isomorfismo queda determinado por la correspondencia:
\begin{equation}\label{Cl10}
\begin{split}
\text{Cl}_{1,0}&\rightarrow{\mathbb{R}\oplus\mathbb{R}},\\
{\bf{1}}&\mapsto{(1,1)},\\
i(e_1)&\mapsto{(1,-1)}.\\
\end{split}
\end{equation}
Se puede ver que $(i(e_1))^2=1$ y que el producto está dado por:
\begin{equation}
(a{\bf{1}}+b{i(e_1)})(c{\bf{1}}+d{i(e_1)}))=(ac+bd){\bf{1}}+(ad+bc){i(e_1)},
\end{equation}
lo que es consistente con la correspondencia expresada en la ecuación \ref{Cl10} y la regla de multiplicación \ref{productosplitcomplex}.

\item Por úlimo calcularemos $\text{Cl}_2=\text{Cl}_{0,2}$. En este caso el espacio vectorial será $\mathbb{R}^2$ y $\{e_1,e_2\}$ una base ortonormal, de modo que $\Phi{(x_1e_1+x_2e_2)}=-(x_1^2+x_2^2)$ para todo par $x_1, x_2$ de números reales. 

Se puede ver facilmente que $i(e_1)^2=i(e_2)^2=-1$ y que $(i(e_1)i(e_2))^2=-1$ y que el álgebra es isomorfa al álgebra de los cuaternios $\mathbb{H}$, bajo la siguiente correspondencia:
\begin{equation}
\begin{split}
\text{Cl}_{2}&\rightarrow{\mathbb{H}},\\
{\bf{1}}&\mapsto{\bf{1}},\\
i(e_1)&\mapsto{\bf{i}},\\
i(e_2)&\mapsto{\bf{j}}, \\
i(e_1)i(e_2)&\mapsto{\bf{k}}.\\
\end{split}
\end{equation}
\end{enumerate}
\end{ejmp}

\subsection{Involuciones canónicas y estructura lineal del álgebra de Clifford}

Tanto en el álgebra de números complejos como en los cuaternios tenemos operaciones de conjugación. Este tipo de operaciones, que al aplicarse dos veces dan como resultado la identidad, se conocen como operaciones de \emph{involución}. En el contexto de álgebras de Clifford existen dos involuciones canónicas que definiremos en esta sección y además la conjugación de Clifford que se define a partir de ellas.
En esta sección definiremos involuciones canónicas en el álgebra de Clifford y estudiaremos la estructura de espacio vectorial del álgebra, concluyendo con un teorema en el cual se construye una base para el álgebra y se garantiza la inyectividad de la función $i_{\Phi}$ cuando el espacio cuadrático es de dimensión finita.

\begin{prop}
Toda álgebra de Clifford tiene un único \emph{antiautomorfismo canónico} $\nobreak{t{:}\text{\emph{Cl}}(\Phi)\rightarrow\text{\emph{Cl}}(\Phi)}$ satisfaciendo las siguientes propiedades:
\begin{equation}
t\circ{t}=id\ \ ;\ \ t(xy)=t(y)t(x) \ \ ;\ \ t(i(v))=i(v) \forall{v}\in{\mathbb{V}}, 
\end{equation}
donde la segunda igualdad es la definición de \emph{antimorfismo}. Notación: $t(x):=x^t$
\end{prop}

\begin{proof}
Consideremos el álgebra opuesta del álgebra de Clifford $\text{Cl}(\Phi)^{op}$ que consiste de los elementos de $\text{Cl}(\Phi)$ pero con el producto $*$, dado por $a*b=ba$, donde el producto en el miembro derecho es el producto en $\text{Cl}(\Phi)$. Sabemos que se cumple el siguiente diagrama:
\[
\begin{tikzcd}[sep=5em]
\mathbb{V} \arrow{r}{i_{\Phi}} \arrow{rd}{i_{\Phi}} 
  & \text{Cl}(\mathbb{V},\Phi) \ar[dashed]{d}{t} \\
    & \text{Cl}^{op}(\mathbb{V},\Phi)
\end{tikzcd}
\]
Existe un único morfismo $t$ tal que cumple $i_{\Phi}=t\circ{i_{\Phi}}$, esto es equivalente a decir que $t(i(v))=i(v)$ con $i=i_{\Phi}$. Además $t(xy)=t(x)*t(y)=t(y)t(x)$, con lo cual es un antimorfismo en $\text{Cl}(\Phi)$.
Resta ver que $t\circ{t}=id$.

Podemos pensar a t como una función de $\text{Cl}(\Phi)$ en $\text{Cl}(\Phi)$, en cuyo caso podemos componerla consigo misma. Notemos que $t\circ{t}$ es un morfismo. En efecto,
\begin{equation}
t\circ{t}(xy)=t(t(xy))=t(t(y)t(x))=t(t(x))t(t(y))=(t\circ{t})(x)(t\circ{t})(y),
\end{equation}
además tenemos:
\begin{equation}
t\circ{t}(i(v))=t(t(i(v)))=t(i(v))=i(v),
\end{equation}
con lo cual $(t\circ{t})\circ{i}=i$, pero el mapeo identidad en $\text{Cl}(\Phi)$ es el único que cumple esto, con lo cual $t\circ{t}=id$.
\end{proof}

\begin{prop}
Toda álgebra de Clifford tiene un único \emph{automorfismo canónico} $\nobreak{\alpha:\text{\emph{Cl}}(\Phi)\rightarrow\text{\emph{Cl}}(\Phi)}$ satisfaciendo las siguientes propiedades:
\begin{equation}
\alpha\circ{\alpha}=id\ \ ;\ \ \alpha(xy)=\alpha(x)\alpha(y) \ \ ;\ \ \alpha(i(v))=-i(v) \forall{v}\in{\mathbb{V}}. 
\end{equation}
\end{prop}

\begin{proof}
Consideremos el siguiente diagrama:
\[
\begin{tikzcd}[sep=5em]
\mathbb{V} \arrow{r}{i_{\Phi}} \arrow{rd}{-i_{\Phi}} 
  & \text{Cl}(\mathbb{V},\Phi) \ar[dashed]{d}{\alpha} \\
    & \text{Cl}(\mathbb{V},\Phi)
\end{tikzcd}
\]
Existe un único morfismo $\alpha$ tal que cumple $-i_{\Phi}=\alpha\circ{i_{\Phi}}$, con lo cual $\alpha$ es un automorfismo en $\text{Cl}(\Phi)$.
Se puede ver fácilmente que $(\alpha\circ{\alpha})\circ{i}=i$ y como el único morfismo que cumple esto es la identidad, tenemos que $\alpha\circ\alpha=\text{id}$.
\end{proof}

Las funciones $\alpha$ y $t$ se conocen respectivamente como $\emph{involución de grado}$ (grade involution en inglés) y $\emph{reversión}$ respectivamente. Los mismos quedan definidos en productos de vectores por:
\begin{equation}
\alpha(i(v_1)i(v_2)...i(v_k))=(-1)^{k}i(v_1)i(v_2)...i(v_k)
\end{equation}
\begin{equation}
t(i(v_1)i(v_2)...i(v_k))=i(v_k)...i(v_2)i(v_1)
\end{equation}
Además se define la \emph{conjugación de Clifford} como $\alpha\circ{t}=t\circ{\alpha}$, que se denota por $\overline{x}:=\alpha\circ{t}(x)=t\circ{\alpha}(x)$. Esta operación queda definida en productos de vectores por:
\begin{equation}
\overline{i(v_1)i(v_2)...i(v_k)}=(-1)^{k}i(v_k)...i(v_2)i(v_1).
\end{equation}

Una cuestión que aún no hemos abordado es la de la dimensión del álgebra de Clifford. Resulta que para un espacio cuadrático regular de dimensión finita, la función $i_{\Phi}$ es inyectiva y la dimensión del álgebra de Clifford es $2^n$. Para probar este resultado primero deberemos introducir el concepto de álgebra graduada y súper-álgebra tensorial.

Observemos en primer lugar que el álgebra de Clifford de cualquier espacio cuadrático admite la siguiente descomposición:

\begin{equation}\label{Z2grading}
\text{Cl}(\Phi)=\text{Cl}^{0}(\Phi)\oplus{\text{Cl}^{1}(\Phi)},
\end{equation}
donde:
\begin{equation}\label{gradingprop}
\text{Cl}^{i}(\Phi)=\{x\in\text{Cl}(\Phi)|\alpha(x)=(-1)^{i}x\}\ \ ; \ \ i\in\{0,1\}.
\end{equation}
Dados $x\in\text{Cl}^{i}(\Phi)$, $y\in\text{Cl}^{j}(\Phi)$, se cumple que $xy\in\text{Cl}^{i+j (mod 2)}(\Phi)$. Entonces decimos que $\text{Cl}(\Phi)$ es un \emph{álgebra $\mathbb{Z}_2$- graduada} o que tiene una graduación $\mathbb{Z}_2$.

\begin{obs}Notemos lo siguiente:
\begin{itemize}
\item De la propiedad \ref{gradingprop} se deduce que $\text{Cl}^{0}(\Phi)$ es una subálgebra pero $\text{Cl}^{1}(\Phi)$ no lo es.
\item En caso de que $\mathbb{V}$ sea de dimensión finita se puede ver que si $a\in\text{Cl}^{0}(\Phi)$ entonces todos los sumandos en $a$ son productos pares de elementos en $i(\mathbb{V})$ y que todo $b\in\text{Cl}^{1}(\Phi)$ está formado por sumandos de productos impares de elementos en $i(\mathbb{V})$.
\end{itemize}
\end{obs}

Dadas dos álgebras $\mathbb{Z}_2$-graduadas $A=A^0\oplus{A^1}$ y $B=B^0\oplus{B^1}$, se define el producto tensorial $\mathbb{Z}_2$ graduado por:
\begin{equation}
A\hat{\otimes}B:=(A\hat{\otimes}B)^0\oplus(A\hat{\otimes}B)^1,
\end{equation} 
con
\begin{equation}
(A\hat{\otimes}B)^0:=(A^0{\otimes}B^0)\oplus(A^1{\otimes}B^1) \ \ \ ; \ \ \ (A\hat{\otimes}B)^1:=(A^0{\otimes}B^1)\oplus(A^1{\otimes}B^0).
\end{equation}
El producto que definimos en $A\hat{\otimes}B$ está dado por:
\begin{equation}
(a'\otimes{b})(a\otimes{b'})=(-1)^{ij}(aa')\otimes(bb'),
\end{equation}
donde $a\in{A}^i$, $b\in{B}^j$. Denotamos $i=\deg(a)$, $j=\deg(b)$. El producto se extiende a elementos $a$ y $b$ no homogéneos (esto es, no contenidos en una única componente de la graduación $\mathbb{Z}_2$) por bilinealidad de $\otimes$. Se puede ver que con esta regla de multiplicación el conjunto $A\hat\otimes{B}$ es un álgebra $\mathbb{Z}_2$-graduada.

Ahora que ya han sido introducidos los conceptos de graduación $\mathbb{Z}_2$ estamos en condiciones de enunciar la siguiente proposición, que tiene una importancia central en la demostración del teorema que nos da la dimensión del álgebra de Clifford en el caso de dimensión finita de $\mathbb{V}$.

\begin{prop}
Sean $\mathbb{V}$ y $\mathbb{W}$ dos espacios cuadráticos reales de dimensión finita con formas cuadráticas $\Phi$ y $\Psi$ respectivamente, existe una forma cuadrática $\Phi\oplus\Psi$ definida en $\mathbb{V}\oplus\mathbb{W}$ por:
\begin{equation}
\Phi\oplus\Psi(v,w)=\Phi(v)+\Psi(w).
\end{equation}
Si llamamos $i:=i_\Phi: \mathbb{V}\rightarrow\text{Cl}(\Phi)$, $j:=i_\Psi:\mathbb{W}\rightarrow\text{Cl}(\Psi)$, se puede definir la transformación lineal:
\begin{equation}
\begin{split}
f:\mathbb{V}\oplus\mathbb{W}&\rightarrow\text{Cl}(\Phi)\hat\otimes{\text{Cl}(\Psi)},\\
(v,w)&\mapsto{i(v)\otimes{1}+1\otimes{j(w)}},
\end{split}
\end{equation}
y ésta induce un isomorfismo $\overline{f}:\text{Cl}(\Phi\oplus\Psi)\rightarrow\text{Cl}(\Phi)\hat\otimes{\text{Cl}(\Psi)}$
\end{prop}
\begin{proof}
Notemos que $f(v,w)^2=\Phi\oplus\Psi(v,w)1\otimes1$. En efecto,
\begin{equation}
\begin{gathered}
f(v,w)^2=(i(v)\otimes1+1\otimes{j(w)})(i(v)\otimes1+1\otimes{j(w)})=\\
=(-1)^{0\times{1}}(i(v)^2\otimes{1})+(-1)^{0\times{0}}(i(v)\otimes{j(w)})+(-1)^{1\times{1}}(i(v)j(w))+(-1)^{0\times{1}}(1\otimes{j(w)}^2)=\\
=i(v)^2\otimes{1}+i(v)\otimes{j(w)}-i(v)j(w)+1\otimes{j(w)^2}=\\
=i(v)^2\otimes{1}+1\otimes{j(w)^2}=\Phi(v)1\otimes{1}+\Psi(w)1\otimes{1}=\\
=(\Phi(v)+\Psi(w))1\otimes{1}=\\
=\Phi\oplus\Psi(v,w)1\otimes1,
\end{gathered}
\end{equation}
con lo cual tenemos el siguiente diagrama conmutativo:
\[
\begin{tikzcd}[sep=5em]
\mathbb{V}\oplus{\mathbb{W}} \arrow{r}{i_{\Phi\oplus\Psi}} \arrow{rd}{f} 
  & \text{Cl}(\Phi\oplus\Psi) \ar[dashed]{d}{\tilde{f}} \\
    & \text{Cl}(\Phi)\hat\otimes\text{Cl}(\Psi)
\end{tikzcd},
\]
con $\tilde{f}$ única satisfaciendo $f=\tilde{f}\circ{i_{\Phi\oplus\Psi}}$. Para ver que $\tilde{f}$ es isomorfismo basta con encontrar un morfismo inverso $\tilde{f'}$:
\begin{equation}
\tilde{f'}: \text{Cl}(\Phi)\hat\otimes\text{Cl}(\Psi)\rightarrow \text{Cl}(\Phi\oplus\Psi).
\end{equation}
Sean las inclusiones $\mathbb{V}\overset{i_1}{\hookrightarrow}\mathbb{V}\oplus\mathbb{W}$ y $\mathbb{W}\overset{i_2}{\hookrightarrow}\mathbb{V}\oplus\mathbb{W}$, éstas inducen los morfismos:
\begin{equation}
\begin{gathered}
\phi:\text{Cl}(\Phi)\rightarrow\text{Cl}(\Phi\oplus\Psi),\\
\psi:\text{Cl}(\Psi)\rightarrow\text{Cl}(\Phi\oplus\Psi).\\
\end{gathered}
\end{equation}
En efecto, si tomamos $g:\mathbb{V}\rightarrow\text{Cl}(\Phi\oplus\Psi)$, dada por $g:=i_{\Phi\oplus\Psi}\circ{i_1}$, se cumple que:
\begin{equation}
g(v)^2=[i_{\Phi\oplus\Psi}(i_1(v))]^2=(i_{\Phi\oplus\Psi}(v,0))^2=\Phi\oplus\Psi(v,0)1=\Phi(v)1,
\end{equation}
con lo cual tenemos el siguiente diagrama conmutativo:
\[
\begin{tikzcd}[sep=5em]\label{diagramag}
\mathbb{V}\oplus{\mathbb{W}} \arrow{r}{i_{\Phi}} \arrow{rd}{g} 
  & \text{Cl}(\Phi) \ar[dashed]{d}{\phi} \\
    & \text{Cl}(\Phi\oplus\Psi)
\end{tikzcd},
\]
donde el $\phi$ es único y $g=\phi\circ{i_{\Phi}}$. De manera análoga, proponiendo $\tilde{g}=i_{\Phi\oplus\Psi}\circ{i_2}$, se obtiene el morfismo $\psi$.
Definimos la función $\tilde{f'}$:
\begin{equation}
\begin{split}
\tilde{f'}:\text{Cl}(\Phi)\hat\otimes\text{Cl}(\Psi)&\rightarrow\text{Cl}(\Phi\oplus\Psi),\\
a\otimes{b}&\mapsto{\phi(a)\psi(b)}
\end{split}
\end{equation}
para $a\in\text{Cl}(\Phi)$ y $b\in\text{Cl}(\Psi)$ y se extiende a elementos arbitrarios en $\text{Cl}(\Phi)\hat\otimes\text{Cl}(\Psi)$ de manera que sea transformación lineal. Veremos que esta función es un morfismo de álgebras y que es la función inversa de $\tilde{f}$ que estamos buscando.

Por definición $\tilde{f'}$ es lineal, resta ver que es morfismo de anillos con $\tilde{f'}(1_{\text{Cl}(\Phi)\hat\otimes\text{Cl}(\Psi)})=1_{\text{Cl}(\Phi\oplus\Psi)}$. La segunda condición se cumple puesto que $\phi$ y $\psi$ son morfismos de álgebras, con lo cual $\tilde{f'}(1\otimes{1})=\phi(1)\psi(1)=1\times1=1$.

Ver que es morfismo de anillos no es tan simple. Veamos cómo actúa $\tilde{f'}$ sobre un elemento de la forma $(a\otimes{b})(c\otimes{d})$, donde $b\in{\text{Cl}(\Phi)^{i}}$ y $c\in{\text{Cl}(\Psi)^{j}}$ para algún par $(i,j)$ con $i,j\in\{0,1\}$.
\begin{equation}
\begin{gathered}
\tilde{f'}((a\otimes{b})(c\otimes{d}))=\tilde{f'}((-1)^{\deg{b}\deg{c}}(ac\otimes{bd}))=\\
=(-1)^{\deg{b}\deg{c}}\phi(ac)\psi(bd)=(-1)^{\deg{b}\deg{c}}\phi(a)\phi(c)\psi(b)\psi(d).
\end{gathered}
\end{equation}
Como $\mathbb{V}$ y $\mathbb{W}$ son de dimensión finita, cada elemento es combinación lineal de productos finitos en $i(\mathbb{V})$ o $j(\mathbb{W})$. Veamos qué sucede con $\phi(c)\psi(b)$ para $c=i(v_1)...i(v_k)$ y $\nobreak{b=j(w_1)...j(w_m)}$:
\begin{equation}\label{productoelementos}
\begin{gathered}
\phi(c)\psi(b)=\phi(i(v_1)...i(v_k))\psi(j(w_1)...j(w_m))=\phi(i(v_1))...\phi((v_k))\psi(j(w_1))...\psi(j(w_m))=\\
=i_{\Phi\oplus\Psi}(v_1,0)...i_{\Phi\oplus\Psi}(v_k,0)i_{\Phi\oplus\Psi}(0,w_1)...i_{\Phi\oplus\Psi}(0,w_m).
\end{gathered}
\end{equation}
Notemos que $v_l$ y $w_s$ son ortogonales $\forall{l,s}$ en el sentido en que $(v_l,0)$ es ortogonal a $(0,w_s)$ respecto de la forma bilineal $\langle,\rangle$ inducida por $\Phi\oplus\Psi$. En efecto,
\begin{equation}
\langle{(v,0),(0,w)}\rangle=\Phi\oplus\Psi(v,w)-\Phi\oplus\Psi(v,0)-\Phi\oplus\Psi(0,w)=\Phi(v)+\Psi(w)-\Phi(v)-\Psi(w)=0,
\end{equation}
con lo cual $i_{\Phi\oplus\Psi}(v,0)i_{\Phi\oplus\Psi}(0,w)=- i_{\Phi\oplus\Psi}(0,w)i_{\Phi\oplus\Psi}(v,0)$. Esto implica que en la ecuación \ref{productoelementos} podemos conmutar sucesivamente cada $i(0,w_s)$ con los $i(v_l,0)$ multiplicando por un $-1$ en cada intercambio. Haciendo esto se obtiene:
\begin{equation}
\phi(c)\psi(b)=(-1)^{k\times{m}}\psi(b)\phi(c)=(-1)^{\deg{b}\deg{b}}\psi(b)\phi(c),
\end{equation}
con lo cual se cumple:
\begin{equation}
\begin{gathered}
\tilde{f'}((a\otimes{b})(c\otimes{d}))=(-1)^{\deg{b}\deg{c}}(-1)^{\deg{b}\deg{c}}\phi(a)\psi(b)\phi(c)\psi(d)=\\
(-1)^{2\deg{b}\deg{c}}\phi(a)\psi(b)\phi(c)\psi(d)=\phi(a)\psi(b)\phi(c)\psi(d)=\tilde{f'}(a\otimes{b})\tilde{f'}(c\otimes{d}).
\end{gathered}
\end{equation}
Dado que todo elemento se puede escribir como suma de elementos de la forma que hemos consuderado en esta prueba, tenemos que $\tilde{f'}(a.b)=\tilde{f'}(a)\tilde{f'}(b)$ para todos $a,b\in{\text{Cl}(\Phi)\hat\otimes\text{Cl}(\Psi)}$.

Con esto concluimos la prueba de que $\tilde{f'}$ es un morfismo entre las álgebras $\text{Cl}(\Phi)\hat\otimes\text{Cl}(\Psi)$ y $\text{Cl}(\Phi\oplus\Psi)$. Resta ver que $\tilde{f'}$ es la inversa de $\tilde{f}$, esto es:
\begin{equation}
\begin{gathered}
\tilde{f'}\circ{\tilde{f}}=\text{id}:\text{Cl}(\Phi\oplus\Psi)\rightarrow\text{Cl}(\Phi\oplus\Psi),\\
\tilde{f}\circ{\tilde{f'}}=\text{id}:\text{Cl}(\Phi)\hat\otimes\text{Cl}(\Psi)\rightarrow\text{Cl}(\Phi)\hat\otimes\text{Cl}(\Psi).
\end{gathered}
\end{equation}
Notemos que:
\begin{equation}
\begin{gathered}
(\tilde{f'}\circ{\tilde{f}})(i_{\Phi\oplus\Psi})(v,w)=\tilde{f'}(\tilde{f}(v,w))=\tilde{f'}(i(v)\otimes{1}+1\otimes{j(w)})=\\
=\phi(i(v))+\psi(j(w))=i_{\Phi\oplus\Psi}(v,0)+i_{\Phi\oplus\Psi}(0,w)=i_{\Phi\oplus\Psi}(v,w),
\end{gathered}
\end{equation}
esto es $(\tilde{f'}\circ\tilde{f})\circ{i_{\Phi\oplus\Psi}}={i_{\Phi\oplus\Psi}}$, pero recordemos que la identidad era el único morfismo que cumplía con esto, con lo cual $\tilde{f'}\circ\tilde{f}=\text{id}$.

Para ver que $\tilde{f}\circ{\tilde{f'}}:\text{Cl}(\mathbb{V})\hat\otimes\text{Cl}(\mathbb{W})\rightarrow\text{Cl}(\mathbb{V})\hat\otimes\text{Cl}(\mathbb{W})$ notemos que los siguientes diagramas conmutan:
\[
\begin{tikzcd}[sep=5em]
\mathbb{V} \arrow{r}{i_{\Phi}} \arrow{rd}{i_{\Phi}\otimes{1_{\text{Cl}(\Psi)}}} 
  & \text{Cl}(\Phi) \ar[dashed]{d}{h} \\
    & \text{Cl}(\Phi)\hat\otimes\text{Cl}(\Psi)
\end{tikzcd},
\begin{tikzcd}[sep=5em]
{\mathbb{W}} \arrow{r}{i_{\Phi}} \arrow{rd}{{1_{\text{Cl}(\Phi)}}\otimes{i_{\Psi}}}
  & \text{Cl}(\Psi) \ar[dashed]{d}{h'} \\
    & \text{Cl}(\Phi)\hat\otimes\text{Cl}(\Psi)
\end{tikzcd}.
\]
Como el morfismo $h$ es único y $h:\text{Cl}(\Phi)\rightarrow\text{Cl}(\Phi)\hat\otimes\text{Cl}(\Psi)$ dado por $a\mapsto{a\otimes{1_{\text{Cl}(\Psi)}}}$ cumple con el diagrama, entonces este es el morfismo $h$ buscado. Análogamente $h'(b)=1_{\text{Cl}(\Phi)}\otimes{b}$.

Por otro lado, notemos que:
\begin{equation}
\begin{gathered}
(\tilde{f}\circ\tilde{f'}\circ{h})\circ{i_\Phi}(v)=(\tilde{f}\circ\tilde{f'})\circ(h\circ{i_\Phi}(v))=(\tilde{f}\circ\tilde{f'})(i_\Phi(v)\otimes{1})=\\
=\tilde{f}(\phi(i_{\Phi}(v)))=\tilde{f}(g(v))=\tilde{f}(i_{\Phi\oplus\Psi}(v,0))=f(v,0)=i_{\Phi}(v)\otimes{1},
\end{gathered}
\end{equation}
donde se usó el diagrama \ref{diagramag}.
Esto es $\tilde{f}\circ\tilde{f'}\circ{h}=h$, pero como $\text{Im}(h)=\text{Cl}(\Phi)\otimes{1}$ entonces podemos decir que:
\begin{equation}
\tilde{f}\circ{\tilde{f'}}|_{\text{Cl}(\Phi)\otimes{1}}=\text{id}:\text{Cl}(\Phi)\otimes{1}\rightarrow\text{Cl}(\Phi)\otimes{1}.
\end{equation}
Análogamente se obtiene $\tilde{f}\circ{\tilde{f'}}|_{{1}\otimes\text{Cl}(\Phi)}=\text{id}:{1}\otimes\text{Cl}(\Phi)\rightarrow{1}\otimes\text{Cl}(\Phi)$.

Entonces, sean $a\in{\text{Cl}(\Phi)}$ y $b\in{\text{Cl}(\Psi)}$, entonces $a\otimes{b}=(a\otimes{1})(1\otimes{b})\in{\text{Cl}(\Phi)\hat\otimes\text{Cl}(\Psi)}$. Como $\tilde{f}\circ\tilde{f'}$ es morfismo entonces tenemos:
\begin{equation}
(\tilde{f}\circ\tilde{f'})(a\otimes{b})=(\tilde{f}\circ\tilde{f'})((a\otimes{1})({1}\otimes{b}))=\tilde{f}\circ\tilde{f'}(a\otimes{1})\tilde{f}\circ\tilde{f'}(1\otimes{b})=(a\otimes{1})(1\otimes{b})=a\otimes{b}.
\end{equation}

Todo elemento en $\text{Cl}(\Phi)\hat\otimes{\text{Cl}(\Psi)}$ se escribe como combinación lineal de los elementos considerados en la ecuación previa, con lo cual tenemos que $\tilde{f}\circ\tilde{f'}=\text{id}:\text{Cl}(\Phi)\hat\otimes{\text{Cl}(\Psi)}\rightarrow\text{Cl}(\Phi)\hat\otimes{\text{Cl}(\Psi)}$. Así $\tilde{f}$ es isomorfismo y $\text{Cl}(\Phi\oplus\Psi)\simeq\text{Cl}(\Phi)\hat\otimes{\text{Cl}(\Psi)}$.
\end{proof}

\begin{teo}\label{baseClifford}
Dado un espacio cuadrático real de dimensión finita $(\mathbb{V},\Phi)$ entonces se cumple que la transformación lineal $i:\mathbb{V}\rightarrow\text{Cl}(\Phi)$ es inyectiva. Además dada una base $\{e_1,...e_n\}$ de $\mathbb{V}$, los $2^n-1$ productos:
\begin{equation}
i(e_{i_1})...i(e_{i_k})\ \ \text{con} \ \ 1\leq{i_1}<i_2<...<i_k\leq{n},
\end{equation}
y la unidad $\bf{1}$ forman una base de $\text{Cl}(\Phi)$, con lo cual el álgebra tiene dimensión $2^n$ sobre $\mathbb{R}$. Esta base se denomina \emph{base estándar de $\text{Cl}(\Phi)$ asociada a la base $\{e_1,...e_n\}$}
\end{teo}

\begin{proof}
La demostración se hará por inducción sobre $n$, la dimensión del espacio vectorial $\mathbb{V}$.

Si n=1, como ya vimos $T(\mathbb{V})\simeq{\mathbb{R}[x]}$, con $x=e_1$ base de $\mathbb{V}$. Ya vimos también en el ejemplo \ref{ejemplos-algebrasdeclifford} que esta álgebra es isomorfa a $\mathbb{C}$ o a $\mathbb{R}\oplus\mathbb{R}$, con lo cual ya se ve que son de dimensión $2=2^1$ y como el elemento $i_{\Phi}(e_1)$ es parte de la base, se ve que la función $i_\Phi$ es inyectiva.

Sea $n=\dim{\mathbb{V}}$, supongamos que el teorema es válido para todos los espacios vectoriales de dimensión menor a $n$. Todo espacio cuadrático admite base ortogonal, con lo cual tenemos la descomposición:
\begin{equation}
\mathbb{V}=\mathbb{R}e_1\oplus\mathbb{V}',
\end{equation}
con $\dim{\mathbb{V}'}=n-1$ y $e_1\perp{\mathbb{V}'}$. Debido a esto, la forma cuadrática $\Phi$ se descompone como $\Phi=\Phi_1\oplus\Phi_{\mathbb{V}'}$ y por lo tanto estamos bajo las condiciones de la proposición previa. Esto implica entonces el siguiente isomorfismo:
\begin{equation}
\text{Cl}(\Phi)=\text{Cl}(\Phi_1\oplus\Phi_{\mathbb{V}'})\simeq\text{Cl}(\Phi_1)\hat\otimes\text{Cl}(\Phi_{\mathbb{V}'}).
\end{equation}
Por hipótesis inductiva sabemos que $\text{Cl}(\Phi_{\mathbb{V}'})$ tiene por base al conjunto:
\begin{equation}
i(e_{i_1})...i(e_{i_k})\ \ \text{con} \ \ 2\leq{i_1}<i_2<...<i_k\leq{n} \ \ y \ \ {\bf{1}}_{\text{Cl}(\Phi_{\mathbb{V}'})},
\end{equation}
con lo cual, como además ${\bf{1}}_{\text{Cl}(\Phi_1)}$ y $i_{\Phi_1}(e_1)$ forman una base de $\text{Cl}(\Phi_1)$, tenemos que el siguiente conjunto forma una base de $\text{Cl}(\Phi_1)\hat\otimes\text{Cl}(\Phi_{\mathbb{V}'})$:
\begin{equation}\label{baseprodtensor}
\begin{gathered}
{\bf{1}}_{\text{Cl}(\Phi_1)}\otimes{i(e_{i_1})...i(e_{i_k})}\ \ \text{con} \ \ 2\leq{i_1}<i_2<...<i_k\leq{n}\\
{\bf{1}}_{\text{Cl}(\Phi_1)}\otimes{\bf{1}}_{\text{Cl}(\Phi_{\mathbb{V}'})}\\
i(e_1)\otimes{i(e_{i_1})...i(e_{i_k})}\ \ \text{con} \ \ 2\leq{i_1}<i_2<...<i_k\leq{n}\\
i(e_1)\otimes{\bf{1}}_{\text{Cl}(\Phi_{\mathbb{V}'})}.
\end{gathered}
\end{equation}
Recordemos que el isomorfismo entre $\text{Cl}(\Phi_1)\hat\otimes\text{Cl}(\Phi_{\mathbb{V}'})$ y $\text{Cl}(\Phi_1\oplus\Phi_{\mathbb{V}'})$ está dado por: $a\otimes{b}\mapsto{\phi(a)\psi(b)}$. Como los elementos de la ecuación \ref{baseprodtensor} son una base de $\text{Cl}(\Phi_1)\hat\otimes\text{Cl}(\Phi_{\mathbb{V}'})$, las imágenes de estos elementos vía el isomorfismo descripto conformarán una base de $\text{Cl}(\Phi_1\oplus\Phi_{\mathbb{V}'})=\text{Cl}(\Phi)$. Se puede ver entonces que estos elementos son enviados por el isomorfismo al siguiente conjunto:
\begin{equation}
i(e_{i_1})...i(e_{i_k})\ \ \text{con} \ \ 1\leq{i_1}<i_2<...<i_k\leq{n}\ \ \text{y} \ \ {\bf{1}},
\end{equation}
como queríamos demostrar. Además claramente $i$ es inyectiva puesto que los elementos $i(e_{i})$ con $i\in\{1,...,n\}$ son parte de la base.
\end{proof}

\begin{obs}
El hecho de que mapeo $i$ sea inyectivo nos permite referirnos a los vectores $v\in\mathbb{V}$ como si fueran elementos del álgebra de Clifford. Es por esto que de aquí en adelante utilizaremos la notación $v:=i(v)$ para todo $v\in\mathbb{V}$, es decir identificamos a $v$ con su imagen por $i$. En este sentido decimos que $\mathbb{V}\subset{\text{Cl}(\Phi)}$.

Notemos que dada $\{e_1,...,e_n\}$ base ortogonal de $\mathbb{V}$ entonces el álgebra de Clifford está presentada por los generadores  $\{e_1,...,e_n\}$ y las relaciones:
\begin{equation}
\begin{gathered}
e_j^2=\Phi(e_j)1\\
e_je_k=-e_ke_j \ \ \forall{j\neq{k}}
\end{gathered}
\end{equation}
\end{obs}

\begin{obs}
Bajo esta notación, las involuciones $\alpha$, $t$ y $\alpha\circ{t}=\bar{  }$ en los elementos de la base quedan determinadas por:
\begin{equation}
\begin{gathered}
t(e_i)=e_i\ \ , \ \ t(1)=1\\
t(e_{i_1}...e_{i_k})=e_{i_k}...e_{i_1}\\
\alpha(e_i)=-e_i\ \ , \ \ \alpha(1)=1\\
\alpha(e_{i_1}...e_{i_k})=(-1)^ke_{i_1}...e_{i_k}\\
\overline{(e_i)}=-e_i\ \ , \ \ \overline{(1)}=1\\
\overline{(e_{i_1}...e_{i_k})}=(-1)^ke_{i_k}...e_{i_1}\\
\text{con} \ \ 1\leq{i_1}<i_2<...<i_k\leq{n}
\end{gathered}
\end{equation}
Se puede ver a partir de esto que $\text{Cl}(\Phi)^0$ está generada como espacio vectorial por el $1$ y los elementos $e_{i_1}...e_{i_k}$ con $k$ par, mientras que $\text{Cl}(\Phi)^1$ tiene por base a los $e_{i_1}...e_{i_k}$ con $k$ impar.
\end{obs}

\subsection{Grupo de Clifford-Lipschitz y los grupos Pin y Spin}

El objetivo de esta sección es introducir los distintos grupos contenidos en el álgebra de Clifford que dan lugar a la generalización del ejemplo analizado en la sección \ref{motivacion}. Recordemos que siguiendo el espíritu de ese ejemplo, queremos encontrar un grupo que se relacione con el grupo de isometrías del espacio cuadrático a través de alguna acción similar a la adjunta.

Denotaremos por $\text{Cl}(\Phi)^{*}$ al grupo de elementos inversibles del álgebra $\text{Cl}(\Phi)$. Notemos que el conjunto $\{x\in{\mathbb{V}}|\Phi(x)\neq{0}\}$ está contenido en $\text{Cl}(\Phi)^{*}$.

Queremos definir una acción lineal del grupo $\text{Cl}(\Phi)^{*}$ sobre el espacio vectorial $\mathbb{V}$ que sea del estilo de la acción adjunta. De hecho la acción que buscamos definir está dada por:
\begin{equation}
v\mapsto{\alpha(x)i_{\Phi}(v)x^{-1}}=\alpha(x)vx^{-1} \ \ \ \forall{v\in{\mathbb{V}}}.
\end{equation}
Resulta que esto no está del todo bien definido, pues existen elementos en $x\in\text{Cl}(\Phi)^{*}$ para los cuales $\alpha(x)\mathbb{V}x^{-1}\not\subset{\mathbb{V}}$, con lo cual no sería un endomorfismo sobre $\mathbb{V}$. Sucede que no podemos definir esta acción sobre todo el grupo de unidades del álgebra de Clifford, sino que debemos restringirnos a un subgrupo.
\begin{defn}
Dado un espacio cuadrático $(\mathbb{V},\Phi)$ definimos el \emph{grupo de Clifford-Lipschitz} $\Gamma(\Phi)$, como sigue:
\begin{equation}
\Gamma(\Phi):=\{x\in{\text{Cl}(\Phi)^{*}}|\alpha(x)vx^{-1}\in{\mathbb{V}}\ \ \forall{v\in\mathbb{V}}\}
\end{equation}
Definimos además la \emph{acción adjunta torcida} del grupo $\Gamma(\Phi)$ sobre el espacio vectorial $\mathbb{V}$, $\rho$ por:
\begin{equation}
\rho_X(v)=\alpha(x)i_{\Phi}(v)x^{-1}=\alpha(x)vx^{-1} \ \ \ \forall{v\in{\mathbb{V}}}.
\end{equation}
\end{defn}
Para ver que ésta sea una buena definición hay que probar que $\Gamma(\Phi)$ es grupo y que $\rho$ es acción lineal.
\begin{proof}
Veamos primero que $\rho$ es acción lineal. Esto es, para todo $x\in{\Gamma(\Phi)}$, $\rho_{x}$ es transformación lineal en $\mathbb{V}$, para todos ${x,y}\in{\Gamma(\Phi)}$ se cumple $\rho_{xy}=\rho_x\circ\rho_y$ y que $\rho_1=id$.
\begin{itemize}
\item $\rho_{x}(\lambda{u}+v)=\alpha(x)(\lambda{u}+v)x^{-1}=\lambda\alpha(x)ux^{-1}+\alpha(x)vx^{-1}=\lambda\rho_x(u)+\rho_x(v)$, esto es $\rho_x$ es lineal para todo $x$. Más aún notemos que $\rho_x$ es isomorfismo en $\mathbb{V}$ para todo $x$. Es inyectiva, pues, supongamos que $\rho_x(u)=0$, entonces tenemos que $\alpha(x)ux^{-1}=0$. Por otro lado, como $x$ tiene inverso $x^{-1}$, $\alpha(x)$ tiene inverso $\alpha(x^{-1})$, con lo cual multiplicando a derecha por $x$ y a izquierda por $\alpha(x^{-1})$ se obtiene que $u=0$. Más aún, como $\mathbb{V}$ es de dimensión finita, $\rho_x:\mathbb{V}\rightarrow\mathbb{V}$ es isomorfismo $\forall{x\in\Gamma{(\Phi)}}$.
\item Trivialmente $1\in{\Gamma(\Phi)}$ y $\rho_1(v)=\alpha(1)v1=v=id(v)$.
\item $\rho_x\circ\rho_y(v)=\rho_x(\rho_y(v))=\alpha(x)\rho_y(v)x^{-1}=\alpha(x)\alpha(y)vy^{-1}x^{-1}=\alpha(xy)v(xy)^{-1}=\rho_{xy}(v)$
\end{itemize}
La otra parte es ver que $\Gamma(\Phi)$ es grupo:
\begin{itemize}
\item Trivialmente $1\in\Gamma(\Phi)$
\item Dado $x\in{\Gamma(\Phi)}$, entonces veremos que $x^{-1}\in{\Gamma(\Phi)}$. En efecto, como $\rho_x$ es biyectiva para todo $x$, entonces dado $v\in\mathbb{V}$ existe $w\in\mathbb{V}$ tal que $v=\alpha(x)wx^{-1}$. Entonces, $\alpha(x^{-1})vx=\alpha(x^{-1})\alpha(x)wx^{-1}x=w\in\mathbb{V} \ \ \forall{v\in{\mathbb{V}}}\forall{x\in{\Gamma(\Phi)}}$. Luego, $x\in{\Gamma(\Phi)}$.
\item Sean $x,y\in\Gamma(\Phi)$, entonces $\alpha(xy)v(xy)^{-1}=\alpha(x)\alpha(y)vy^{-1}x^{-1}$. Como $y\in\Gamma(\Phi)$, entonces $\alpha(y)vy^{-1}=w\in{\mathbb{V}}$, con lo cual $\alpha(x)\alpha(y)vy^{-1}x^{-1}=\alpha(x)wx^{-1}\in{\mathbb{V}}$, pues $x\in{\Gamma(\Phi)}$. Por lo tanto $xy\in\Gamma(\Phi)$.
\end{itemize}
Luego, $\Gamma(\Phi)$ es un grupo y $\rho$ es acción lineal.
\end{proof}

\begin{obs}\label{adjuntatorcida}
El término acción adjunta torcida surge para diferenciarla de la \emph{acción adjunta} $v\mapsto{xvx^{-1}}$. Ésta será una acción lineal de otro subgrupo $G\in\subset\text{Cl}(\Phi)^{*}$, construido por el requisito de que $xvx^{-1}\in\mathbb{V} \ \ \forall{x\in{G},v\in\mathbb{V}}$. Esta acción se denota por Ad. En la bibliografía donde se trata con las dos acciones en simultáneo suele llamarse a $G$ el grupo de Clifford-Lipschitz y a $\Gamma(\Phi)$, el grupo de Clifford-Lipschitz torcido \cite{VazdaRocha}. En lo que sigue en estas notas utilizaremos solamente la acción adjunta torcida.

La acción adjunta torcida fue introducida por Atiyah, Bott y Shapiro en el artículo \emph{``Clifford Modules''}\cite{ABS} de central importancia en teoría de Clifford. La ventaja que tiene frente a la acción adjunta usual y que veremos en lo sucesivo, es que dado un vector isotrópico $v$, la transformación $\rho_v$ corresponde a la reflexión respecto del hiperplano ortogonal a $v$, $s_v$. Además el núcleo de esta acción está dado por $\mathbb{R}^{*}1$. En el caso de la acción adjunta, uno obtiene $Ad(v)=-s_v$ y el núcleo de esta acción es el centro del álgebra de Clifford, que puede ser más grande que $\mathbb{R}^{*}1$.
\end{obs}

\begin{defn}
Definimos la aplicación $N:\text{Cl}(\Phi)\rightarrow\text{Cl}(\Phi)$ por $N(x)=x\overline{x}$. $N(x)$ se denomina \emph{la norma de $x$}.
\end{defn}
\begin{defn}
Definimos el \emph{grupo especial de Clifford} $\Gamma(\Phi)^{0}:=\Gamma(\Phi)\cap\text{Cl}(\Phi)^{0}$.
\end{defn}

\begin{obs}
Observemos que
\begin{itemize}
\item $N(v)=v\overline{v}=v(-v)=-v^2=-\Phi(v)1$ para todo vector $v\in\mathbb{V}$.
\item Además $N(e_{i_1}...e_{i_k})=e_{i_1}...e_{i_k}\overline{e_{i_1}...e_{i_k}}=e_{i_1}...e_{i_k}\overline{e_{i_k}}...\overline{e_{i_k}}=(-1)^k\Phi(e_{i_1})...\Phi(e_{i_k})1$
\end{itemize}
\end{obs}

\begin{prop}
Las funciones $\alpha$ y $t$ inducen un automorfismo y antiautomorfismo respectivamente en $\Gamma{\Phi}$
\end{prop}

\begin{proof}
Las propiedades se heredan de $\text{Cl}(\Phi)$, hay que ver que son cerradas en $\Gamma(\Phi)$. Esto es, dado $x\in{\Gamma}(\Phi)$, entonces hay que ver que $\alpha(x)$ y $t(x)$ también lo están.
Sea $v\in\mathbb{V}$,
\begin{equation}
\alpha(\alpha(x))v\alpha(x)^{-1}=\alpha(\alpha(x))(-\alpha(v))\alpha(x^{-1})=-\alpha(\alpha(x)vx^{-1}).
\end{equation}
Como $x\in{\mathbb{V}}$, se cumple que $\alpha(x)vx^{-1}\in\mathbb{V}$. Por otro lado, $\alpha(w)=-w$ para todo $w\in{\mathbb{V}}$, con lo cual se obtiene que:
\begin{equation}
\alpha(\alpha(x))v\alpha(x)^{-1}=\alpha(x)vx^{-1}\in{\mathbb{V}},
\end{equation}
con lo cual obtenemos que $\alpha(x)\in{\Gamma(\Phi)}$.

Por otro lado notemos que:
\begin{equation}
\alpha(t(x))vt(x)^{-1}=t(\alpha(x))t(v)t(x^{-1})=t(x^{-1}v\alpha(x)).
\end{equation}
Como $\Gamma(\Phi)$ es grupo y $\alpha(x)\in\Gamma(\Phi)$, entonces $\alpha(x)^{-1}=\alpha(x^{-1})\in\Gamma(\Phi)$, con lo cual:
\begin{equation}
\alpha(\alpha(x^{-1}))v\alpha(x)=x^{-1}v\alpha(x)\in\mathbb{V},
\end{equation}
y como además $t(w)=w$ para todo $w\in\mathbb{V}$, se tiene que:
\begin{equation}
\alpha(t(x))vt(x)^{-1}=x^{-1}v\alpha(x)\in\mathbb{V}.
\end{equation}
Por lo tanto $\Gamma(\Phi)$ es cerrado por $\alpha$ y $t$ y queda demostrada la proposición.
\end{proof}

Como se mencionó previamente, en la observación \ref{adjuntatorcida}, hay una relación entre la acción adjunta torcida del grupo de Clifford y las reflexiones en el espacio cuadrático $\mathbb{V}$. Este hecho, junto con el teorema de Cartan-Dieudonné nos permite obtener todas las isometrías del espacio $(\mathbb{V},\Phi)$ a partir de la acción adjunta torcida.

\begin{prop}\label{reflexiones}
Sea $(\mathbb{V},\Phi)$ un espacio cuadrático de dimensión finita. Para cada elemento $x\in\mathbb{V}$ con $\Phi(v)\neq{0}$, entonces
\begin{equation}
\begin{split}
\rho_x:\mathbb{V}&\longrightarrow\mathbb{V}\\
v&\mapsto\alpha(x)vx^{-1},
\end{split}
\end{equation}
es la reflexión respecto del hiperplano $H=x^{\perp}$, de todos los vectores ortogonales al vector $x$.
\end{prop}

\begin{proof}
La reflexión $s_x$ respecto del hiperplano $H$, ortogonal a $x$, con $\Phi(x)\neq{0}$, entonces:
\begin{equation}
s_x(u)=u-2\frac{\varphi(u,x)}{\Phi(x)}x\ \ \ \ \forall{u}\in\mathbb{V}.
\end{equation}
Tenemos además que $x^2=\Phi(x)1$ $(x^{-1}=x/\Phi(x))$ y además $ux+xu=2\varphi(u,x)1$, por lo tanto:
\begin{equation}
s_x(u)=u-\frac{(ux+xu)}{\Phi(x)}x=u-\frac{ux^2}{\Phi(x)}-\frac{xux}{\Phi(x)}=\alpha(x)ux^{-1}=\rho_x(u).
\end{equation}
\end{proof}

A continuación probaremos el otro hecho notado en la observación \ref{adjuntatorcida} respecto del núcleo de la acción.

\begin{lema}\label{nucleo}
El núcleo de la función $\rho:\Gamma(\Phi)\longrightarrow{\text{Aut}(\mathbb{V})}$ es $\mathbb{R}^{*}1$, el grupo de múltiplos escalares no nulos de $1$.
\end{lema}

\begin{proof}
Supongamos que $\rho(x)=\text{id}:\mathbb{V}\rightarrow\mathbb{V}$. Entonces se cumple que:
\begin{equation}\label{nucleoeq1}
\alpha(x)vx^{-1}=v \ \ \Rightarrow \alpha(x)v=vx \ \ \forall{v\in\mathbb{V}}.
\end{equation}
Por la graduación $\mathbb{Z}_{2}$ del álgebra de Clifford, se tiene que:
\begin{equation}
x=x^0+x^{1}\ \ \ \text{, con} x^{i}\in\text{Cl}^{i},
\end{equation}
con lo cual $\alpha(x)=\alpha(x^{0})+\alpha(x^{1})=x^{0}-x^{1}$. De la ecuación \ref{nucleoeq1}, se puede escribir:
\begin{equation}
(x^0-x^1)v=v(x^0+x^1) \Rightarrow \underbrace{x^{0}v-vx^{0}}_{\in\text{Cl}^{1}}- \underbrace{(x^1v+vx^{1})}_{\in\text{Cl}^{0}}=0,
\end{equation}
y al igualar tanto la parte perteneciente a $\text{Cl}^0$, como la parte en $\text{Cl}^1$ a cero, se obtiene que:
\begin{equation}\label{nucleoeq2}
x^0v=vx^0 \ \ \ , \ \ \ -x^1v=vx^1.
\end{equation}
Por el teorema \ref{baseClifford}, podemos separar de $x^0$ la parte que contiene a $e_1$ de la que no. Esto es existen $a^0\in\text{Cl}^{0}$ y $b^1\in\text{Cl}^{1}$, tales que:
\begin{equation}
x^0=a^0+e_1b^1,
\end{equation}
donde ni $a^0$, ni $b^1$ contienen sumandos con $e_1$ como factor. Podría decirse que $a^0$ y $b^1$ están en la sub-álgebra generada por $\{e_2,...,e_n\}$.

Aplicando esto a la primera igualdad en \ref{nucleoeq2} y utilizando $v=e_1$, se obtiene:
\begin{equation}
\begin{gathered}
(a^++e_1b^1)e_1=e_1(a^0+e_1b^1)=\\
=a^0e_1+e_1b^1e_1=e_1a^0+e_1^2b^1.
\end{gathered}
\end{equation}
Como la base es ortogonal se tiene que $e_ie_j=-e_je_i \ \ \forall{i\neq{j}}$ y además $a^0\in\text{Cl}^0$ y no contiene a $e_1$ en su expansión, se tiene que $a^0e_1=e_1a^0$. Análogamente, como $b^1\in\text{Cl}^1$ y no contiene a $e_1$, $e_1b^1=-b^1e_1$. Con todo esto, la ecuación se reduce a:
\begin{equation}
e_1b^1e_1=e_1^2b^1 \ \ \ \equiv \ \ \ -e_1^2b^1=e_1^2b^1.
\end{equation}
Además $e_1^2=\Phi(e_1)1$, con lo cual la ecuación queda $b^1=-b^1$ o equivalentemente $b^1=0$.

De acá deducimos que $x^0$ no contiene ningún sumando que tenga a $e_1$ como factor en su expansión. Haciendo esto mismo con todos los vectores de la base ortogonal de $\mathbb{V}$, concluimos que $x^0\in\mathbb{R}1$.

Podemos aplicar un argumento similar para $x^1$. Tenemos que $x^1=a^1+e_1b^0$, con $a^1\in\text{Cl}^{1}$, $b^0\in\text{Cl}^{0}$ y ninguno contiene a $e_1$ en su expansión,con lo cual, la segunda igualdad en \ref{nucleoeq2} para $v=e_1$ se obtiene:
\begin{equation}
\begin{gathered}
-(a^1+e_1b^0)e_1=e_1(a^1+e_1b^0)=\\
=-a^1e_1-e_1b^0e_1=e_1a^1+e_1^2b^0.
\end{gathered}
\end{equation}
Al igual que en el caso anterior, como $a^1$ está en $\text{Cl}^{1}$ y no contiene a $e_1$ en su expansión, entonces $a^1e_1=-e_1a^1$. Además como $b^0$ es par y no contiene $e_1$, $b^0e_1=e_1b^0$ y se obtiene:
\begin{equation}
-e_1^2b_0=e_1^2b^0 \ \ \equiv \ \ b^0=-b^0 \ \ \equiv \ \ b^0=0,
\end{equation}
donde hemos usado $e_1^2=\Phi(e_1)1$. Esto implica que $x^1$ no contiene a $e_1$ en su expansión. Repitiendo este procedimiento con todos los vectores de la base, se concluye que $x^1$ no depende de nungún $e_i$. Esto, junto al hecho de que $x^1\in\text{Cl}^1$, implica que $x^1=0$. Con lo cual $x=x^0\in\mathbb{R}1$. Por otro lado $x\in\Gamma(\Phi)$, entonces $x\in(\mathbb{R}1\cap\Gamma(\Phi))=\mathbb{R}^*1$. Luego $x\in\mathbb{R}^*1$ y $\ker{\rho}=\mathbb{R}^*1$.
\end{proof}

\begin{obs}
El lema es válido para toda forma no degenerada, pero por ejemplo si $\Phi\equiv{0}$ en $\mathbb{R}^2$y tomamos $x=1+e_1e_2$, se tiene que $x^{-1}=1-e_1e_2$. Entonces:
\begin{equation}
\begin{gathered}
\alpha(1+e_1e_2)v(1+e_1e_2)^{-1}=(1+e_1e_2)v(1-e_1e_2)=\\
=(v+e_1e_2v)(1-e_1e_2)=v-ve_1e_2-e_1e_2ve_1e_2.
\end{gathered}
\end{equation}
Un vector $v$ se escribe $v=v^1e_1+v^2e_2$, con lo cual tenemos que:
\begin{equation}
e_1e_2v=e_1e_2(v^1e_1+v^2e_2)=v^1e_1e_2e_1=-v^1e_2e_1^2=0,
\end{equation}
y por lo tanto $\alpha(x)vx^{-1}=v$, esto es $x\in{\ker{\rho}}$, aún cuando $x\notin{\mathbb{R}^{*}1}$.
\end{obs}

Las siguientes proposiciones utilizan el lema \ref{nucleo}, por lo que, en luz de la observación de arriba, solo tendrán validez en el caso de formas cuadráticas no degeneradas.

\begin{prop}
Si $x\in\Gamma(\Phi)$, entonces $N(x)\in\mathbb{R}^{*}1$
\end{prop}

\begin{proof}
La idea para demostrar esta proposición es, haciendo uso del lema previo, ver que $N(x)$ pertenece al núcleo de $\rho$ para todo $x$ en el grupo de Clifford. Sea $x\in\Gamma(\Phi)$, entonces $\alpha(x)vx^{-1}\in{\mathbb{V}}$ para todo $v\in\mathbb{V}$.

Entonces se cumple: $\mathbb{V}\ni\alpha(x)vx^{-1}={t(\alpha(x)vx^{-1})}=t(x^{-1})vt(\alpha(x))=t(x)^{-1}v\overline{x}$. Multiplicando por $t(x)$ a izquierda y por $\overline{x^{-1}}$ a derecha en estas igualdades se obtiene:

\begin{equation}
\begin{gathered}
v=t(x)\alpha(x)vx^{-1}\overline{x}^{-1}=t(x)\alpha(x)v(\overline{x}x)^{-1}=\\
=t(\alpha\circ\alpha(x))\alpha(x)v(\overline{x}x)^{-1}=\alpha(\overline{x}x)v(\overline{x}x)=\rho_{N(x)}(v),
\end{gathered}
\end{equation}
con lo cual $N(x)\in\ker(\rho)$, esto es $N(x)=\overline{x}x=x\overline{x}\in\mathbb{R}^{*}1$ para todo $x\in\Gamma(\Phi)$
\end{proof}

\begin{obs}
Notemos que esto implica que $\overline{x}\in\Gamma(\Phi)$ para todo $x\in\Gamma(\Phi)$, y que además si $\Phi$ no es una forma definida, $N(x)$ puede tener signo positivo o negativo, pues, por ejemplo $N(v)=v^2=\Phi(v)$ para cualquier $v\in\mathbb{V}$.
\end{obs}

\begin{prop}
La restricción de la norma $N$ a $\Gamma(\Phi)$, $N:\Gamma(\Phi)\rightarrow\mathbb{R}^{*}1$ es un morfismo de grupos y $N(\alpha(x))=N(x) \ \ \forall {x\in\Gamma(\Phi)}$.
\end{prop}

\begin{proof}
\begin{equation}
\begin{gathered}
N(xy)=xy\overline(xy)=xy\overline{y}\overline{x}=xN(y)\overline(x)=N(x)N(y)=N(y)N(x)\\
N(\alpha(x))=\alpha(x)\overline{\alpha(x)}=\alpha(x)\alpha\circ{t}\circ\alpha(x)=\alpha(x)\alpha(\overline{x})=\alpha(x\overline{x})=\alpha(N(x))=N(x)
\end{gathered}
\end{equation}
\end{proof}

En las proposiciones recién demostradas se ha calculado el núcleo de la representación adjunta torcida del grupo de Clifford y se puede apreciar que no es inyectiva. Ahora, ¿qué sucede con la suryectividad? Veremos a continuación que esta representación (acción) no es suryectiva y que de hecho cubre solo una parte bastante específica del grupo de automorfismos del espacio vectorial $\mathbb{V}$, esto es, sólo las isometrías de $\Phi$ en $\mathbb{V}$.

\begin{prop}
El conjunto de vectores no isotrópicos $\{v\in\mathbb{V}|\Phi(v)\neq{0}\}$ está contenido en el grupo de Clifford y además $\rho(\Gamma(\Phi))\subseteq{O(\Phi)}$.
\end{prop}

\begin{proof}
Sabemos que dado un vector no isotrópico $v$, entonces $\rho_v=s_v$ es la reflexión respecto del hiperplano ortogonal a $v$, esto es $\rho_v(w)\in\mathbb{V}, \ \ \forall{w\in\mathbb{V}}$. Por lo tanto $\nobreak{\{v\in\mathbb{V}|\Phi(v)\neq{0}\}\subseteq\Gamma(\Phi)}$.

Para ver que $\rho(\Gamma(\Phi))\subseteq{O(\Phi)}$, es suficiente con probar que $\Phi(\rho_x(v))=\Phi(v)$ $\forall{x\in\Gamma(\Phi)}$, $\forall{v\in\mathbb{V}}$. Recordemos que $\Phi(v)1=-N(v)$, con lo cual podemos reemplazar la invariancia de $\Phi$ por la invariancia de $N$. Esto es $N(\alpha(x)vx^{-1})=N(v) \forall{v\in\mathbb{V}} \ \forall{x\in\Gamma(\Phi)}$. En efecto, tenemos que:
\begin{equation}
\begin{gathered}
N(\alpha(x)vx^{-1})=\alpha(x)vx^{-1}\overline(\alpha(x)vx^{-1})=\alpha(x)vx^{-1}\overline{x^{-1}}\overline{v}\alpha(\overline{x})=\\
=N(x^{-1})N(v)N(\alpha(x))=N(x^{-1})N(x)N(v)=N(x^{-1}x)N(v)=N(v),
\end{gathered}
\end{equation}
donde hemos usado el hecho de que $N$ restringido a $\Gamma(\Phi)$ es un morfismo de grupos.
\end{proof}

Recordemos que en la sección \ref{motivacion} mencionamos y exploramos la relación entre los grupos $\text{SU}(2)$ y $\text{SO}(3)$. El álgebra de Clifford nos permite construir para cada grupo de isometrías $\text{SO}(p,q)$ un grupo al que denominaremos $\text{Spin}(p,q)$, que se relaciona con el primero de la misma manera en que $\text{SU}(2)$ lo hace con $\text{SO}(3)$ en el ejemplo motivante.

A este punto, hemos recolectado todo lo necesario para introducir este grupo. Más aún, también definiremos para cada espacio cuadrático $(\mathbb{V},\Phi_{p,q})$ el grupo $\text{Pin}(p,q)$ que guarda una relación con el grupo $\text{O}(p,q)$ análoga a la que se establece entre los grupos $\text{Spin}(p,q)$ y $\text{SO}(p,q)$. Estos grupos se pueden construir empleando herramientas del algebra de Clifford y es lo que haremos en las secciones siguientes.

\subsubsection{Los grupos Pin(n) y Spin(n)}

Esta sección será dedicada al estudio de métricas del tipo $\Phi(v)=-||v||^2$, donde $|| ||$ es la norma euclídea usual en $\mathbb{R}^n$. Notemos que en este caso, $N(v)=||v||^2$.

Recordemos la notación $\text{Cl}_n:=\text{Cl}_{0,n}$ y además que $\Gamma_n:=\Gamma_{0,n}$.

\begin{defn}
Definimos \emph{el grupo Pin$(n)$} como el núcleo de $ker(N)$ del homomorfismo $N:\Gamma_n\longrightarrow{\mathbb{R}^{*}1}$ y \emph{el grupo Spin$(n)$} como:
\begin{equation}
\text{Spin}(n)=\text{Pin}(n)\cap\Gamma_n^{+}=\text{Pin}(n)\cap{\text{Cl}_n^{0}}.
\end{equation}
\end{defn}

Notemos que $\ker(N)=\{x\in\Gamma_n|N(x)=1\}=\text{Pin}(n)$ y $\text{Spin}(n)=\{x\in\Gamma_n^{+}|N(x)=1\}$. En particular si $N(x)=x\overline{x}=1$ entonces $x^{-1}=\overline{x}$, con lo cual la condición de inversibilidad está contenida en $N(x)=1$.

Además podemos escribir a estos dos grupos como:
\begin{equation}
\text{Pin}(n)=\{x\in\text{Cl}_n|\alpha(x)\mathbb{V}x^{-1}\subseteq\mathbb{V} \ \ y \ \ x\overline{x}=1\}=\{x\in\text{Cl}_n|\alpha(x)\mathbb{V}\overline{x}\subseteq\mathbb{V} \ \ y \ \ x\overline{x}=1\}\\
\end{equation}
\begin{equation}
\text{Spin}(n)=\{x\in\text{Cl}_n^0|x\mathbb{V}x^{-1}\subseteq\mathbb{V} \ \ y \ \ x\overline{x}=1\}=\{x\in\text{Cl}_n^0|x\mathbb{V}\overline{x}\subseteq\mathbb{V} \ \ y \ \ x\overline{x}=1\}\\
\end{equation}

El siguiente teorema es el que permite la generalización del grupo $\text{SU}(2)$, en relación al grupo $\text{SO}(3)$.

\begin{teo}\label{PinSpinOSO(n)}
La restricción $\rho:\text{Pin}(n)\longrightarrow{\text{O}(n)}$ es un homomorfismo suryectivo de grupos con $\ker(\rho)=\{-1,1\}\cong\mathbb{Z}_2$. También se cumple que la restricción $\rho:\text{Spin}(n)\longrightarrow{\text{SO}(n)}$ es un homomorfismo suryectivo con el mismo núcleo $\ker(\rho)=\{-1,1\}\cong\mathbb{Z}_2$.
\end{teo}

\begin{proof}
Dado que $\rho(\Gamma_n)\subseteq\text{O}(n)$ tenemos una aplicación $\rho:\text{Pin}(n)\longrightarrow{\text{O}(n)}$. Además como $\text{Pin}(n)\underset{sg}{\subset}\Gamma_n$, entonces $\rho$ va a ser morfismo de grupos entre $\text{Pin}(n)$ y $\text{O}(n)$.

Por el teorema de Cartan-Dieudonné, toda isometría es la composición de a lo sumo $n$ reflexiones $s_{w_j}$ respecto de hiperplanos, $H_{w_j}$, con $\Phi(w_j)\neq{0}$. Como cada reflexión no depende del módulo del vector $w_j$ considerado, los elegimos con $\Phi(w_j)=-1$ (o equivalentemente $N(w_j)=1$). Entonces, dada una aplicación $f\in\text{O}(n)$, se cumple que $f=s_{w_1}\circ...\circ{s_{w_k}}$, con $k\leq{n}$. Tenemos entonces que:
\begin{equation}
s_{w_1}\circ...\circ{s_{w_k}}=f=\rho_{w_1}\circ...\circ{\rho_{w_k}}=\rho_{w_1...w_k},
\end{equation}
donde la última igualdad se deduce por el hecho de que $\rho$ es morfismo. Si calculamos la norma de $w_1...w_k$, teniendo en cuenta que cada $w_j$ está en el grupo de Clifford y que allí $N$ es un morfismo, tenemos $N(w_1...w_2)=N(w_1)...N(w_k)=1$. Esto es, $w_1...w_k\in\ker{N}=\text{Pin}(n)$, con lo cual $\rho:\text{Pin}(n)\longrightarrow{\text{O}(n)}$ es suryectiva.

El núcleo de $\rho$ está dado por:
\begin{equation}
\ker(\rho|_{\text{Pin}(n)})=\ker{(\rho)}\cap\text{Pin}(n)=\mathbb{R}^{*}1\cap\text{Pin}(n)=\{t\in\mathbb{R}^{*}1|N(t)=1\}=\{-1,1\},
\end{equation}
como queríamos demostrar.

Ahora queremos ver que $\rho(\text{Spin}(n))=\text{SO}(n)$ y que $\ker(\rho|_{\text{Spin}(n)})=\ker(\rho|_{\text{Pin}(n)})=\{-1,1\}$, para lo cual bastaría con ver que $\{-1,1\}\in\text{Spin}(n)$.

Sea $f\in\text{SO}(n)\subset\text{O}(n)$, entonces $\det(f)=1$. Por otro lado, $f=s_{w_1}\circ...\circ{s_{w_k}}$, con $\det(s_{w_j})=-1$ para cada $s_{w_j}$. Como $\det:\text{O}(n)\longrightarrow\{-1,1\}$ es un morfismo de grupos, entonces $\det(f)=\det(s_{w_1}\circ...\circ{s_{w_k}})=\det(s_{w_1})...\det(s_{w_k})=(-1)^{k}=1$, esto implica que $k$ es un número par. Así, $f=\rho_{w_1...w_k}$, con $k$ par, con lo cual $w_1...w_k\in{\text{Cl}_n^{0}}$  y además $N(w_1...w_k)=1$.

Concluimos entonces que $w_1...w_k\in\text{Pin}\cap\text{Cl}_n^{0}=\text{Pin}\cap\Gamma_{n}^{+}=\text{Spin}(n)$, esto es $\rho|_{\text{Spin}(n)}\rightarrow\text{SO}(n)$ es suryectiva. Además notemos que $\{-1,1\}\subset\Gamma_n^{+}$, con lo cual $\ker(\rho|_{\text{Spin}(n)})=\rho|_{\text{Pin}(n)}=\{-1,1\}$.
\end{proof} 

Definimos a la \emph{$(n-1)$-esfera unitaria}, $S^{n-1}=\{v\in\mathbb{V}|N(v)=1\}$.El teorema previo da lugar al siguiente corolario:

\begin{corol}\label{generadoporesfera}
El grupo $\text{Pin}(n)$ es generado como grupo por $S^{n-1}$ y todo elemento en $\text{Spin}(n)$ puede escribirse como producto de un número par de elementos en $S^{n-1}$
\end{corol}

A continuación calcularemos a modo de ejemplo los grupos $\text{Spin}(n)$ y $\text{Pin}(n)$ para algunos casos

\begin{ejmp}
Calculemos $\text{Pin}(1)$. Si recordamos del ejemplo \ref{ejemplos-algebrasdeclifford}, $\text{Cl}_1\cong\mathbb{C}$. En este caso, $S^{n-1}=\{-i,i\}$, con lo cual $\text{Pin}(1)=\langle{i}\rangle=\{i,-i,1,-1\}\cong\mathbb{Z}_{4}$. Además tenemos que $\text{Spin}(1)=\{-1,1\}\cong\mathbb{Z}_2$.
\end{ejmp}

\begin{ejmp}\label{PinSpin2}
Consideremos el caso de $\text{Pin}(2)$. Recordemos del ejemplo \ref{ejemplos-algebrasdeclifford} que $\text{Cl}_2\cong\mathbb{H}$, bajo la identificación:
\begin{equation}
1\leftrightarrow{1} \ \ , \ \ e_1\leftrightarrow{\mathbf{i}} \ \ , \ \ e_2\leftrightarrow{\mathbf{j}} \ \ , \ \ e_1e_2\leftrightarrow{\mathbf{k}}
\end{equation}

En este caso, $S^{1}=\{a{\mathbf{i}}+b{\mathbf{j}} | a^2+b^2=1\}\subset{\text{Cl}_n^{1}}$, con lo cual, si tomamos dos elementos en $S^{1}$, $a{\mathbf{i}}+b{\mathbf{j}}$ y $c{\mathbf{i}}+d{\mathbf{j}}$, tenemos que el producto de estos dos está dado por:
\begin{equation}
(a{\mathbf{i}}+b{\mathbf{j}})(c{\mathbf{i}}+d{\mathbf{j}})=-(ac+bd)1+(ad-bc){\mathbf{k}},
\end{equation}
donde se cumple que:
\begin{equation}
(ac+bd)^2+(ad-bc)^2=a^2c^2+\cancel{2abcd}+b^2d^2+a^2d^2\cancel{-2abcd}+b^2c^2=(a^2+b^2)(c^2+d^2)=1.
\end{equation}
Podemos decir entonces que los $(a{\mathbf{i}}+b{\mathbf{j}})(c{\mathbf{i}}+d{\mathbf{j}})=(\alpha1+\beta{\mathbf{k}})$, con $\alpha^2+\beta^2=1$.

Si consideramos ahora el producto de tres elementos en $S^{1}$, obtenemos:
\begin{equation}
\begin{gathered}
(a{\mathbf{i}}+b{\mathbf{j}})(c{\mathbf{i}}+d{\mathbf{j}})(e{\mathbf{i}}+f{\mathbf{j}})=(\alpha1+\beta{\mathbf{k}})(e{\mathbf{i}}+f{\mathbf{j}})=\\
=(\alpha{e}+\beta{f}){\mathbf{i}}+(\alpha{f}-\beta{e}){\mathbf{j}}.
\end{gathered}
\end{equation}
Se cumple que:
\begin{equation}
\begin{gathered}
(\alpha{e}+\beta{f})^2+(\alpha{f}-\beta{e})^2=\alpha^2e^2+\cancel{2\alpha\beta{ef}}+\beta^2f^2+\alpha^2f^2\cancel{-2\alpha\beta{fe}}+\beta^2{e^2}=\\
=(\alpha^2+\beta^2)(e^2+f^2)=1,
\end{gathered}
\end{equation}
con lo cual este elemento vuelve a caer en $S^{1}$. Concluimos entonces que:
\begin{equation}
\text{Pin}(2)=\{a{\mathbf{i}}+b{\mathbf{j}} | a^2+b^2=1\}\cup\{a1+b{\mathbf{k}} | a^2+b^2=1\},
\end{equation}
y que:
\begin{equation}
\text{Spin}(2)=\{a1+b{\mathbf{k}} | a^2+b^2=1\}\cong\text{U}(1).
\end{equation}
Vemos que $\text{Spin}(2)\cong\text{U}(1)$, con lo cual $\text{Spin}(2)$ es un grupo de Lie.
\end{ejmp}

\begin{ejmp}
En este ejemplo calcularemos los grupos Pin$(3)$ y Spin$(3)$ y veremos que, efectivamente, el grupo Spin$(3)$ es isomorfo al grupo SU$(2)$ (y por lo tanto al grupo de cuaternios unitarios) como se expuso en la sección \ref{motivacion}.

Se puede ver que $\text{Cl}_{3}\cong\mathbb{H}\oplus\mathbb{H}$ (por ejemplo en la sección 16.1, página 211 en \cite{Lounesto}), bajo la siguiente correspondencia:

\begin{equation}\label{Cl3cuaternios1}
\begin{gathered}
1\mapsto(1,1) \ \ , \ \ e_1\mapsto(-i,i) \\
e_2\mapsto(j,-j) \ \ , \ \ e_3\mapsto(k,-k).
\end{gathered}
\end{equation}
Se deriva entonces:
\begin{equation}\label{Cl3cuaternios2}
\begin{gathered}
e_{1}e_2\mapsto(k,k) \ \ , \ \ e_{2}e_3\mapsto(i,i) \\
e_{3}e_1\mapsto(j,j) \ \ , \ \ e_{1}e_{2}e_3\mapsto(-1,1).
\end{gathered}
\end{equation}
Vamos a obtener los grupos Pin$(n)$ y Spin$(n)$ utilizando el corolario \ref{generadoporesfera}, con lo cual necesitaremos calcular $S^2$.
\begin{equation}
S^2=\{ae_1+be_2+ce_3|a^2+b^2+c^2=1\},
\end{equation}
en el lenguaje de la sección \ref{motivacion}, son los cuaternios puros de norma $1$.

¿Cómo son los elementos generados por $S^2$ como grupo? Si multiplicamos dos elementos en $S^2$ obtenemos:
\begin{equation}
(ae_1+be_2+ce_3)(de_1+fe_2+ge_3)=k_{1}1+k_{2}e_{1}e_{2}+k_{3}e_{2}e_{3}+k_{4}e_{1}e_{3},
\end{equation}
donde se cumple $k_{1}^2+k_{2}^2+k_{3}^2+k_{4}^2=1$. Esto se puede ver de forma similar a como se hizo para calcular Spin$(2)$. Si consideramos un producto de tres elementos en $S^2$ obtenemos uno de la forma:
\begin{equation}
h_{1}e_{1}+h_{2}e_{2}+h_{3}e_{3}+h_{4}e_{1}e_{2}e_{3}\ ,
\end{equation}
donde $h_{1}^2+h_{2}^2+h_{3}^2+h_{4}^2=1$. De aquí deducimos:
\begin{equation}
\text{Pin}(3)=\{a1+be_{1}e_2+ce_{2}e_3+de_{1}e_{3}|a^2+b^2+c^2+d^2=1\}\cup\{ae_1+be_{2}+ce_3+de_{1}e_{2}e_{3}|a^2+b^2+c^2+d^2=1\}.
\end{equation}
Haciendo la intersección entre Pin$(3)$ y $\text{Cl}_{3}^0$, obtenemos el grupo Spin$(3)$:
\begin{equation}
\text{Spin}(3)=\{a1+be_{1}e_2+ce_{2}e_3+de_{1}e_{3}|a^2+b^2+c^2+d^2=1\}.
\end{equation}
Si usamos las ecuaciones \ref{Cl3cuaternios1} y \ref{Cl3cuaternios2} vemos que el subespacio generado por $\{1,e_{1}e_{2},e_{2}e_{3},e_{3}e_{1}\}$ es isomorfo al álgebra de cuaternios $\mathbb{H}$, con lo cual el grupo Spin$(3)$ cumple:
\begin{equation}
\text{Spin}(3)\cong\{a1+bi+cj+dk\in\mathbb{H}|a^2+b^2+c^2+d^2=1\}.
\end{equation}
Tal como vimos en la sección \ref{motivacion}, el grupo Spin asociado al espacio tridimensional euclídeo es el grupo de cuaternios unitarios que, como ya vimos, es isomorfo al grupo SU$(2)$. Esto es Spin$(3)\cong$SU$(2)$.
\end{ejmp}

\subsubsection{Los grupos Pin(p,q) y Spin(p,q)}

Como mencionamos en la sección \ref{espacioscuad}, si $\Phi$ es una métrica no degenerada de signatura $(p,q)$, entonces existe una base donde la misma se escribe:
\begin{equation}
\Phi(x_1,...,x_{p+q})=\Phi_{p,q}(x_1,...,x_{p+q}):=x_1^2+...+x_{p}^2-(x_{p+1}+...+x_{p+q}).
\end{equation}
O$(p,q)$ será el grupo de isometrías con respecto a la forma cuadrática $\Phi_{p,q}$, esto es:
\begin{equation}
\text{O}(p,q)=\{f\in{\text{Aut}(\mathbb{V})}\ | \ \Phi_{p,q}(f(v))=\Phi(v) \ \forall{v}\in\mathbb{V}\}.
\end{equation}
SO$(p,q)$ es el subgrupo de O$(p,q)$ de las isometrías con determinante $+1$, y además definimos la siguiente notación:
\begin{equation}
\text{Cl}_{p,q}=\text{Cl}(\Phi_{p,q}) \ \ ; \ \ \Gamma_{p,q}=\Gamma(\Phi_{p,q}) \ \ ;\ \ \Gamma^{+}_{p,q}=\Gamma^{+}(\Phi_{p,q}).
\end{equation}

En la sección anterior, debido a que la norma era una función definida positiva en $\Gamma$,  se definió al grupo Pin$(n)$ como el núcleo de la aplicación de la norma $N$. En el caso en que la forma cuadrática no es definida, no nos alcanzará el núcleo de la norma $N$ para cubrir a todo el espacio de isometrías, es por ello que en lugar de esa definición se adopta la siguiente:
\begin{defn}
El grupo Pin$(p,q)$ se define por:
\begin{equation}
\text{Pin}(p,q)=\{x\in\Gamma_{p,q}|N(x)\in\{-1,+1\}\},
\end{equation}
el grupo Spin$(p,q)$, está dado por Spin$(p,q)=\text{Pin}(p,q)\cap\Gamma^{+}_{p,q}$.
Con esto podemos decir que:
\begin{equation}
\begin{gathered}
\text{Pin}(p,q)=\{x\in\text{Cl}_{p,q}\ | \ \alpha(x)v\overline{x}\in\mathbb{V} \ \forall{v}\in\mathbb{V} ; \ N(x)\in\{-1,+1\}\}\\
\text{Spin}(p,q)=\{x\in\text{Cl}_{p,q}^{0}\ | \ xv\overline{x}\in\mathbb{V} \ \forall{v}\in\mathbb{V} ; \ N(x)\in\{-1,+1\}\}
\end{gathered}
\end{equation}
\end{defn}
\begin{obs}
Dado que $\alpha(x)v\overline{x}\in\mathbb{V}$ implica que $\alpha(\alpha(x)v\overline{x})=-\alpha(x)v\overline{x}\in\mathbb{V}$, se tiene:
\begin{equation}
\alpha(\alpha(x)v\overline{x})=x\alpha(v)t(x)=-xvt(x)=-\alpha(x)v\overline{x}\ \implies \ xvt(x)=\alpha(x)v\overline{x},
\end{equation}
con lo cual otra forma de definir al conjunto Pin es la siguente:
\begin{equation}
\text{Pin}(p,q)=\{x\in\text{Cl}_{p,q}\ | \ xvt(x)\in\mathbb{V} \ \forall{v}\in\mathbb{V} ; \ N(x)\in\{-1,+1\}\}\\.
\end{equation}
\end{obs}

Al igual que en el caso de los grupos Pin$(n)$ y Spin$(n)$, aquí también existe un teorema análogo al \ref{PinSpinOSO(n)} y es el siguiente:

\begin{teo}
La restricción $\rho:\Gamma_{p,q}\rightarrow\text{O}(p,q)$ al grupo Pin$(p,q)$ es un homomorfismo suryectivo $\rho:\text{Pin}(p,q)\rightarrow\text{O}(p,q)$ con $\ker(\rho)=\{-1,+1\}$. Además, la restricción $\rho:\text{Spin}(p,q)\rightarrow\text{SO}(p,q)$ también es un morfismo suryectivo de grupos con el mismo núcleo $\ker(\rho)=\{-1,+1\}$. Dicho de otro modo, como Spin$(p,q)\subset$Pin$(p,q)$, $\{-1,+1\}\in$Spin$(p,q)$.
\end{teo}

\begin{proof}
Como el teorema de Cartán-Dieudonné vale para todo espacio cuadrático regular, la prueba de este teorema es análoga a la hecha para los grupos $Pin(n)$ y $Spin(n)$.
\end{proof}

\begin{corol}
El grupo Pin$(p,q)$ es generado por el conjunto $S^{n-1}_{p,q}$, definido por:
\begin{equation}
\text{S}^{n-1}_{p,q}=\{x\in\mathbb{V}\ |\ N(x)=\pm{1}\}.
\end{equation} 
Además, todo elemento en Spin$(p,q)$ es producto de un número par de elementos en $S^{n-1}_{p,q}$.
\end{corol}

\begin{ejmp}
Ya vimos en el ejemplo \ref{ejemplos-algebrasdeclifford}, que el álgebra de Clifford $\text{Cl}_{1,0}=\mathbb{R}\oplus\mathbb{R}$ y llamamos al conjunto $\mathbb{R}\oplus\mathbb{R}$ los números complejos hiperbólicos o números de Study. Una forma de definir estos números, que imita a la definición de números complejos, es considerar al espacio vectorial $\mathbb{R}$, que tiene a $1$ como base e incorporar un elemento $j$ que sería una ``unidad imaginaria'' que cumple $j^2=1$. En este caso $\text{Cl}_{1,0}=\mathbb{R}1\oplus\mathbb{R}j$, con $j=i_{\Phi}(e_1)$. Teniendo esto en cuenta, tenemos:
\begin{equation}
{S}^{1}_{1,0}=\{j,-j\},
\end{equation}
con lo cual Pin$(1,0)=\{j,j^2,-j,-j^2\}=\{j,-j,1,-1\}\cong\mathbb{Z}_{2}\times\mathbb{Z}_{2}$. El ismorfismo se puede ver porque Pin$(1,0)$ es un grupo finito de orden $4$ y todos sus elementos, excepto el neutro, son de orden $2$. Además $\text{Spin}(1,0)\cong\mathbb{Z}_{2}$.
\end{ejmp}

\begin{ejmp}\label{Cl20}
Calculemos los grupos $\text{Pin}(2,0)$ y $\text{Spin}(2,0)$. Se puede ver que $\text{Cl}_{2,0}\cong{\text{M}}_{2}(\mathbb{R})$, las matrices reales cuadradas de $2\times{2}$ (sección 16.1 de \cite{Lounesto}, página 207), bajo la correspondencia:
\begin{equation}\label{Cl2_matriz}
\bold{1}\mapsto \begin{pmatrix}
		1 \ & 0\\
		 0 \ & 1\\
	       \end{pmatrix}
;\quad
\bold{e_1} \mapsto \begin{pmatrix}
		0 \ & 1\\
		 1 \ & 0\\
	       \end{pmatrix}
;\quad
{\bold{e_2}}\mapsto \begin{pmatrix}
		1 \ & 0\\
		 0 \ & -1\\
	       \end{pmatrix}
;\quad
\bold{e_1e_2}\mapsto \begin{pmatrix}
		0 \ & -1\\
		 1 \ & 0\\
	       \end{pmatrix}.
\end{equation}
Tenemos entonces que $S^{1}_{2,0}=\{ae_1+be_2\ | \ a^2+b^2=1\}$.
Se puede ver que:
\begin{equation}
\begin{gathered}
(ae_1+be_2)(ce_1+de_2)=\alpha{1}+\beta{e_{1}e_2}\text{, con }\alpha^2+\beta^2=1,\\
(\alpha{1}+\beta{e_{1}e_2})(ae_1+be_2)=ce_1+de_2\text{, con }c^2+b^2=1.
\end{gathered}
\end{equation}
Con esto se ve que 
\begin{equation}
\begin{gathered}
\text{Pin}(p,q)=\{ae_1+be_2\ | \ a^2+b^2=1\}\cup\{a1+be_{1}e_2\ | \ a^2+b^2=1\},\\
\text{Spin}(p,q)=\{a1+be_{1}e_2\ | \ a^2+b^2=1\}\cong{\text{U}(1)}.
\end{gathered}
\end{equation}
Si recordamos el ejemplo \ref{PinSpin2}, teníamos que $\text{Spin}(2,0)\cong\text{U}(1)$. Entonces, vemos que $\text{Spin}(2,0)=\text{Spin}(0,2)$. Se puede probar que $\text{Cl}_{p,q}^{0}\cong\text{Cl}_{q,p}^{0}$ para cualesquiera $p$ y $q$, con lo cual $\text{Spin}(p,q)=\text{Spin}(q,p)$, sin embargo como $\text{Cl}_{p,q}\ncong\text{Cl}_{q,p}$, entonces $\text{Pin}(p,q)\ncong\text{Pin}(q,p)$.
\end{ejmp}

\newpage
\section{Espinores Algebraicos. Aplicaciones a la física}

El objetivo de esta sección es la introducción del concepto de espinores algebraicos y su aplicación a la física. Para ello primero será necesario definir las representaciones regulares irreducibles de un álgebra y conocer ciertas propiedades de la misma. Esto último ocupará las primeras dos partes de esta sección y la última estará avocada a la definición de los espinores y las aplicaciones. A lo largo de esta sección se siguen principalmente los capítulos 4 y 6 de \cite{VazdaRocha}.

\subsection{Representaciones regulares irreducibles}\label{repreg}

Dado que un álgebra es un espacio vectorial en sí mismo, podemos hablar de representaciones de un álgebra que tienen por espacio de representación a la misma álgebra. Ese es el caso de las representaciones regulares.

\begin{defn}\label{repreg}
Dada un álgebra $A$, $B$ un subespacio de $A$ y la función $L:A\rightarrow{\text{End}{(B)}}$, dada por:
\begin{equation}
L_{a}\in\text{End}{(B)}\text{ tal que }{L}_{a}{(b)}=ab \ \ \forall{b}\in{B},
\end{equation}
el par (L,B) se dice que es una \emph{representación regular izquierda}.
De manera análoga una representación se dice \emph{representación regular derecha} si tiene por espacio de representación a $B$ un subespacio de $A$ y por función a $R:A\rightarrow{\text{End}{(B)}}$, donde: 
\begin{equation}
R_{a}\in\text{End}{(B)}\text{ tal que }{R}_{a}{(b)}=ba \ \ \forall{b}\in{B}.
\end{equation}
\end{defn}

\begin{obs}
Notemos que dada esta definición, no cualquier subespacio de $A$ es un posible espacio de representación. Una posibilidad obvia es tomar $B=A$ y esto va a ser una representación de $A$ sobre $A$, pero para subespacios propios hay que tener cierto cuidado, dado que el mismo deberá cumplir ciertas condiciones que veremos a continuación.

En el caso de una representación regular izquierda, dados $a\in{A}$ y $b\in{B}$, si $L_{a}\in\text{End}(B)$ entonces $ab\in{B}$ para todos $a\in{A}, b\in{B}$. Una forma equivalente de expresar esto es decir que $AB\subset{B}$, es decir $B$ es un \emph{ideal izquierdo de $A$}. Análogamente, se ve que todos los posibles espacios de representación regulares derechos de $A$ son \emph{ideales derechos de A}. 

Qué sucede si a esta representación izquierda le pedimos además que sea irreducible. En ese caso, dado $S$ un subespacio de $B$, se cumple que $L_{a}(S)\subset{S}$ si y solo si $S=\{0\}$ o $S=B$. Equivalentemente podemos decir que $AS\subset{S}$ si y sólo si $S=\{0\}$ o $S=B$, pero ésta es la definición de \emph{ideal minimal izquierdo}, es decir un ideal de $A$ que no tiene subideales propios. Luego, el espacio de representación de toda representación regular izquierda irreducible es siempre un ideal minimal izquierdo de $A$. Análogamente se puede ver que los espacios de representación regulares derechos de $A$ son ideales minimales derechos de $A$.
\end{obs}

\subsection{Ideales izquierdos (o derechos), idempotentes y álgebras semisimples}

Esta sección pretende solamente mencionar algunos resultados importantes sobre álgebras semisimples (como es el caso de las álgebras de Clifford \cite{VazdaRocha}), sin ahondar en las demostraciones ni en los fundamentos. Para mayor profundidad en estos temas se puede consultar el capítulo 2 de \cite{Hazewinkel}, \cite{Lam}, y para los temas específicos de álgebras de Clifford \cite{VazdaRocha}, que es la referencia que se sigue en toda esta sección. 

Los resultados expresados en esta sección valen para cualquier anillo semismiple sin necesidad de imponer la estructura adicional de espacio vectorial que tiene un álgebra. Un conepto de gran importancia a lo largo de esta parte es el de idempotente.

\begin{defn}
Un \emph{idempotente} de una $\mathbb{K}$-álgebra $A$ es un elemento $f\in{A}$ que cumple $f^2=f$. Dos idempotentes $f$ y $g$ se dicen \emph{ortogonales} si $fg=gf=0$. Un idempotente $f$ se dice \emph{primitivo} si no existen idempotentes $g,h\in{A}$ ortogonales tales que $f=g+h$.
\end{defn}

Un resultado conocido es que en un álgebra (o anillo) semisimple (recordar definicion \ref{algsimple}) todo ideal (izquierdo o derecho) es un ideal principal generado por un idempotente, con lo cual, dada una representación regular izquierda, el espacio de representación está dado por $Af$ para algún idempotente $f$. Para el caso de una representación irreducible, es necesario y suficiente pedir que el idempotente $f$ sea primitivo. En efecto, si $f$ no se descompone en suma de idempotentes ortogonales, entonces el ideal $Af$ no tiene subideales, puesto que todo ideal es generado por un idempotente.

\begin{defn}
Un conjunto de idempotentes primitivos ortogonales $\{f_1,...,f_k\}$ que cumplen $f_1+...+f_k=1$ se denomina \emph{conjunto completo de idempotentes primitivos ortogonales (ccipo)}.
\end{defn}

\begin{ejmp}\label{ejemplomatriz}
Consideremos el álgebra $M_{3}(\mathbb{R})$. En esta álgebra las matrices que tienen algunas columnas (pero no todas) iguales a cero son ideales izquierdos no triviales, por ejemplo:
\begin{equation}
\begin{pmatrix}
		\mathbb{R} \ & \mathbb{R} \ & 0 \\
		\mathbb{R} \ & \mathbb{R} \ & 0 \\
		\mathbb{R} \ & \mathbb{R} \ & 0 \\
\end{pmatrix}
\end{equation}
es un ideal izquierdo, pero no es minimal. Los ideales minimales de este tipo son los que tienen sólo una columna no nula y son isomorfos como espacios vectoriales a $\mathbb{R}^{3}$.

El idempotente que genera al ideal que tiene solo la primera columna distinta de cero es el siguiente:
\begin{equation}
\begin{pmatrix}
		1 \ & 0 \ & 0 \\
		0 \ & 0 \ & 0 \\
		0 \ & 0 \ & 0 \\
\end{pmatrix}
.
\end{equation}
y el siguiente conjunto conforma un conjunto completo de idempotentes primitivos ortogonales para el álgebra:
\begin{equation}
\begin{pmatrix}
		1 \ & 0 \ & 0 \\
		0 \ & 0 \ & 0 \\
		0 \ & 0 \ & 0 \\
\end{pmatrix}
;\quad
\begin{pmatrix}
		0 \ & 0 \ & 0 \\
		0 \ & 1 \ & 0 \\
		0 \ & 0 \ & 0 \\ 
\end{pmatrix}
;\quad
\begin{pmatrix}
		0 \ & 0 \ & 0 \\
		0 \ & 0 \ & 0 \\
		0 \ & 0 \ & 1 \\
\end{pmatrix}
.
\end{equation}
Los ideales derechos generados por estos idempotentes son los que tienen alguna fila totalmente nula.

Se puede ver que para toda álgebra de matrices cuadradas de $d\times{d}$ con entradas en algún álgebra de división $\mathbb{K}$, los ideales minimales son isomorfos a $\mathbb{K}^{d}$. Este es un hecho que se rescata y utiliza en las álgebras de Clifford.
\end{ejmp}

Notemos que si existe un ccipo en un álgebra $A$, se tienen las siguientes descomposiciones del álgebra:
\begin{equation}\label{descomp}
\begin{gathered}
A=Af_1\oplus...\oplus{A}f_k, \\
A=f_{1}A\oplus...\oplus{f}_{k}A ,\\
A=\bigoplus_{i,j}{f}_{i}A{f}_{j},
\end{gathered}
\end{equation}
que podrían verse como aplicaciones de la \emph{descomposición de Peirce}\cite{Lam,Hazewinkel}. La primera descomposición es en suma directa de ideales izquierdos, mientras que la segunda es como suma directa de ideales derechos. 

Analicemos por un momento los sumandos directos en la tercera descomposición. Nos interesa ver qué clase de conjunto son los $f_{i}Af_{j}$. Definimos $A_{ij}:=f_{i}Af_{j}$, notemos que dados $i\neq{j}$ tenemos que si $a,b\in{A_{ij}}$ entonces se cumple que $ab=0$, puesto que los idempotentes distintos son ortogonales, con lo cual los conjuntos $A_{ij}$ son los que se conocen como \emph{pseudo-anillos de cuadrado nulo}; es decir que la regla de multiplicación es $ab=0$ para todo $a,b\in{A_{ij}}$ y no tiene un elemento neutro respecto del producto.

El caso de los conjuntos $A_{ii}$ es distinto, por ejemplo $f_{i}\in{A_{ii}}$ y $f_{i}^2=f_{i}\neq{0}$, con lo cual la regla de multiplicación es distinta. Sucede que si el álgebra $A$ es simple, los conjuntos $A_{ii}$ son \emph{álgebras de división} para cualquier $i\in\{1,...,k\}$, con la identidad en $A_{ii}$ dada por $f_{i}$. No se probará este hecho pero puede encontrarse una prueba en la página 105 de \cite{VazdaRocha}. 
En el caso semisimple, $A_{ii}$ es lo que se denomina un \emph{anillo doble}\cite{Lounesto}, es decir, $A_{ii}\cong\mathbb{K}\oplus\mathbb{K}$ con $\mathbb{K}$ un álgebra de división.

Volviendo a la cuestión de las representaciones, supongamos que elegimos uno de los ideales izquierdos en la primera descomposición, $Af_{i}$, con ${i}\in\{1,...,k\}$ y tomamos la representación regular izquierda que lo tiene por espacio de representación. Podemos preguntarnos si esta representación es esencialmente distinta a la que se obtiene al tomar otro de los ideales $Af_{j}$ con $i\neq{j}$. Es decir, ¿son $(L_{i},Af_{i})$ y $(L_{j},Af_{j})$ representaciones equivalentes? La respuesta es que sí, sin embargo, para probarlo debemos primero introducir nuevos conceptos. Comenzaremos revisando este hecho para el caso de álgebras simples dado que resulta más sencillo e instructivo.

\subsubsection{Caso de álgebras simples}

Analicemos el caso particular de un álgebra $A$ que es simple (no tiene ideales biláteros propios). En dicho caso sucede que para cada uno de los idempotentes $f_{i}$ se cumple que $Af_{i}A=A$. En efecto, como $A$ es simple y $Af_{i}A$ es un ideal bilátero no nulo (pues contiene a $f_{i}$), entonces dicha igualdad es la que vale. Teniendo en cuenta esto, deducimos:
\begin{equation}\label{algsimp1}
f_{i}Af_{i}=f_{i}Af_{j}Af_{i}=f_{i}Af_{j}f_{j}Af_{i}=A_{ij}A_{ji} \ \ \forall{i, j}.
\end{equation}
Como cada idempotente $f_{i}$ pertenece a su respectivo $A_{ii}$, entonces tenemos que siempre existen elementos $\mathcal{E}_{ij}\in{A_{ij}}$ y $\mathcal{E}_{ji}\in{A_{ji}}$ tales que:
\begin{equation}\label{algsimp2}
f_{i}=\mathcal{E}_{ij}\mathcal{E}_{ji}.
\end{equation}

Se puede ver que estos mismos elementos también cumplen $\mathcal{E}_{ji}\mathcal{E}_{ij}=f_{j}$\footnote{Notemos que $\mathcal{E}_{ji}\mathcal{E}_{ij}$ es un idempotente que pertenece al álgebra de división $A_{jj}$, pero el único idempotente allí es $f_{j}$ (por ser primitivo), luego $\mathcal{E}_{ji}\mathcal{E}_{ij}=f_{j}$.}, con lo cual podemos definir $\mathcal{E}_{ii}:=f_{i}$ y decir que:

\begin{equation}\label{algsimp3}
\mathcal{E}_{ij}\mathcal{E}_{lm}=\delta_{jl}\mathcal{E}_{im} \ \forall{i,j,l,m}\in\{1,...,k\},
\end{equation}
donde $\delta$ es la función delta de Kronecker.

Una propiedad interesante de las álgebras de división $A_{ii}$ es que son todas isomorfas entre sí. 

\begin{prop}
$A_{ii}\simeq{A_{jj}}$ para todos $i,j$ si $A$ es un álgebra simple.
\end{prop} 

\begin{proof}
Basta con definir el isomorfismo entre álgebras:
\begin{equation}
\begin{split}
\phi_{ij}:A_{ii}&\longrightarrow{A_{jj}}\\
x&\longmapsto{\mathcal{E}_{ji}x\mathcal{E}_{ij}},
\end{split}
\end{equation}
que tiene morfismo inverso:
\begin{equation}
\begin{split}
\phi_{ij}^{-1}:A_{jj}&\longrightarrow{A_{ii}}\\
x&\longmapsto{\mathcal{E}_{ij}x\mathcal{E}_{ji}},
\end{split}
\end{equation}
y notemos que se preserva la unidad: $\phi_{ij}(f_{i})=\mathcal{E}_{ji}f_{i}\mathcal{E}_{ij}=\mathcal{E}_{ji}\mathcal{E}_{ij}=f_{j}$.
\end{proof}

Todas las álgebras $A_{ii}$ son entonces isomorfas a una única álgebra de división (con $1\cong{f_{1}}\cong...\cong{f_{k}}$) que denotaremos por $\mathbb{K}$. Si $A$ es un álgebra de Clifford puede probarse que esta álgebra es $\mathbb{R}$, $\mathbb{C}$ o $\mathbb{H}$.

\begin{obs}\label{Kmoduloder} Si tomamos el álgebra de división $\mathbb{K}$, un conjunto $A_{ii}$ y los isomorfismos $\mathbb{K}\xrightarrow{\varphi}A_{ii}\xrightarrow{\phi_{im}}A_{mm}$, podemos dar a cada ideal $Af_{m}$ una estructura lineal derecha sobre el álgebra de división $\mathbb{K}$.

Utilizando el isomorfismo $\varphi_{m}=\phi_{im}\varphi$ (con $\varphi_{i}=\varphi$) podemos definir en $Af_{m}$ la operación externa derecha por:
\begin{equation}
\begin{split}
Af_{m}\times\mathbb{K}&\rightarrow{Af_{m}}\\
(\psi,\lambda)&\mapsto\psi\bullet\lambda=\psi\varphi_{m}(\lambda)
\end{split}
\end{equation}

Este producto está bien definido, pues como $A$ es simple $Af_{m}A=A$, lo que implica $Af_{m}(f_{m}Af_{m})=Af_{m}Af_{m}=Af_{m}$. Se puede ver que con este producto $Af_{m}$ es un $\mathbb{K}$-módulo derecho.

Se podrían definir los isomorfismos $\varphi_{m}:\mathbb{K}\rightarrow{A_{mm}}$ sin elegir un $A_{ii}$ en particular, sino simplemente pidiendo sobre ellos la relación de compatibilidad $\varphi_{m}=\phi_{lm}\varphi_{l}$ para todo par $l,m\in\{1,...,k\}$.

En general se asumirá que se trabaja con $\mathbb{K}=f_{m}Af_{m}$ y se obviará $\varphi_{m}$ en el producto, quedando el mismo definido por $(\psi,\lambda)\mapsto\psi\lambda$
\end{obs}

\begin{obs}Notemos que:
\begin{enumerate}[a)]
\item La proposición previa se puede generalizar a cualquier par de idempotentes primitivos. Es decir, para todo par de idempotentes primitivos $f$, $g$ se tiene que $fAf\cong{gAg}$.
\item Los ideales $Af_{i}$ no son $\mathbb{K}$-módulos izquierdos pues, si bien $(f_{i}Af_{i})Af_{i}=f_{i}Af_{i}\subseteq{Af_{i}}$, se tiene que dado $\psi\in{Af_{i}}$, $f_{i}\psi\neq{\psi}$, con lo cual no se verifica la propiedad $1_{\mathbb{K}}\psi=\psi$.
\end{enumerate}
\end{obs}

\begin{teo}
El $\mathbb{K}$- módulo derecho $Af_{i}$ es un módulo libre y el conjunto $\{\mathcal{E}_{ji}\ : \ j=1...k\}$ es una base.
\end{teo}
\begin{proof}
Dado $\psi=\psi{f_{i}}\in{A}f_{i}$ buscamos una expansión del tipo \[\psi=\sum_{j=1}^{k}\mathcal{E}_{ji}\bullet\lambda^{j}=\sum_{j=1}^{k}\mathcal{E}_{ji}\alpha^{j}\] con $\alpha^{j}=\varphi_{i}(\lambda^{j})\in{A}_{ii}$. Multiplicando a izquierda por $\mathcal{E}_{im}$ con $m\in\{1,...,k\}$ tenemos:
\begin{equation}
\mathcal{E}_{im}\psi=\sum_{j=1}^{k}\mathcal{E}_{im}\mathcal{E}_{ji}\alpha^{j}=\sum_{j=1}^{k}\delta_{mj}\mathcal{E}_{ii}\alpha^{j}=\mathcal{E}_{ii}\alpha^{m}=\alpha^{m}.
\end{equation}
Luego, todo elemento en $\psi\in{Af_{i}}$ puede escribirse como combinación lineal con coeficientes $\lambda^{j}=\varphi^{-1}_{i}(\mathcal{E}_{ij}\psi)=\varphi^{-1}_{i}(\mathcal{E}_{ij}\psi\mathcal{E}_{ii})$.

Resta ver que $\{\mathcal{E}_{ji}\ : \ j=1...k\}$ es linealmente independiente. Sea $\sum_{j=1}^{k}\mathcal{E}_{ji}\alpha^{j}=0$, entonces multiplicando a izquierda por $\mathcal{E}_{im}$ tenemos que:
\begin{equation}
0=\sum_{j=1}^{k}\mathcal{E}_{im}\mathcal{E}_{ji}\alpha^{j}=\sum_{j=1}^{k}\delta_{mj}\mathcal{E}_{ii}\alpha^{j}=\alpha^{m},
\end{equation} 
por lo tanto el conjunto es una base.
\end{proof}

\begin{teo}
Los espacios $A_{ij}=f_{i}Af_{j}$ tienen estructura de $\mathbb{K}$-bimódulo dada por:
\begin{equation}
\begin{split}
\mathbb{K}\times{A}_{ij}&\rightarrow{A_{ij}}\\
(\lambda, a_{ij})&\mapsto\lambda\bullet{a}_{ij}=\varphi_{i}(\lambda)a_{ij}
\end{split}\quad\quad
\begin{split}
{A}_{ij}\times\mathbb{K}&\rightarrow{A_{ij}}\\
(a_{ij},\lambda)&\mapsto{a}_{ij}\bullet{\lambda}={a}_{ij}\varphi_{j}(\lambda)
\end{split}.
\end{equation}
Más aún, si $\mathbb{K}$ es conmutativa (o sea un cuerpo) se cumple que $\lambda\bullet{a_{ij}}={a_{ij}}\bullet\lambda$ para todo $a_{ij}\in{A_{ij}}$ y $\lambda\in{\mathbb{K}}$.

Además cada $A_{ij}$ es un $\mathbb{K}$-bimódulo libre de dimensión uno sobre $\mathbb{K}$ con base $\{\mathcal{E}_{ij}\}$.
\end{teo}

Notemos que si bien usamos el mismo símbolo ($\bullet$) para ambos productos, estas operaciones son distintas.

\begin{proof}
No es muy dificil ver que son módulos derechos e izquierdos sobre $\mathbb{K}$, puesto que $\varphi_{i}(1_{\mathbb{K}})a_{ij}=f_{i}a_{ij}=a_{ij}$ y $a_{ij}\varphi_{j}(1_{\mathbb{K}})=a_{ij}f_{j}=a_{ij}$ para cualquier $a_{ij}\in{A}_{ij}$.

La condición de compatibilidad de bimódulo:
\begin{equation}
(\lambda\bullet{a_{ij}})\bullet\mu=\lambda\bullet(a_{ij}\bullet\mu) ,
\end{equation}
se deduce trivialmente de la asociatividad del álgebra de Clifford.

Para probar que $\{\mathcal{E}_{ij}\}$ es un conjunto de generadores a izquierda y derecha de $A_{ij}$ primero veamos que $\lambda\bullet\mathcal{E}_{ij}=\mathcal{E}_{ij}\bullet{\lambda}$ para todos $i,j\in\{1,...,k\}$ y todo $\lambda\in{\mathbb{K}}$:
\begin{equation}
{\mathcal{E}_{ij}}\bullet\lambda=\mathcal{E}_{ij}\varphi_{j}(\lambda)=\mathcal{E}_{ij}\mathcal{E}_{ji}\varphi_{i}(\lambda)\mathcal{E}_{ij}=f_{i}\varphi_{i}(\lambda)\mathcal{E}_{ij}=\varphi_{i}(\lambda)\mathcal{E}_{ij}=\lambda\bullet\mathcal{E}_{ij},
\end{equation}

Sea $a_{ij}\in{A_{ij}}$ veamos que $a_{ij}=\mathcal{E}_{ij}\bullet\lambda^{ij}=\lambda^{ij}\bullet\mathcal{E}_{ij}$ con $\lambda^{ij}\in{\mathbb{K}}$. Esto será equivalente a ver que $a_{ij}=\mathcal{E}_{ij}\alpha^{ij}=\beta^{ij}\mathcal{E}_{ij}$ con $\alpha^{ij}\in{A_{jj}}$, $\beta^{ij}\in{A_{ii}}$ y $\alpha^{ij}=\varphi_{j}(\lambda^{ij})$, $\beta^{ij}=\varphi_{i}(\lambda^{ij})$.

Tomando $\alpha^{ij}=\mathcal{E}_{ji}a_{ij}$ y $\beta^{ij}=a_{ij}\mathcal{E}_{ji}$ se cumple lo buscado. Además podemos ver que $\alpha^{ij}=\phi_{ij}(\beta^{ij})$, con lo cual $\alpha^{ij}$ y $\beta^{ij}$ corresponden al mismo elemento en $\mathbb{K}$ via los isomorfismos $\varphi_{i}$ y $\varphi_{j}$.

Veamos ahora que $\lambda\bullet{a_{ij}}=a_{ij}\bullet\lambda$ si $\mathbb{K}$ es un cuerpo. Para la base de $A_{ij}$ se cumple ${\mathcal{E}_{ij}\bullet}\lambda=\lambda\bullet\mathcal{E}_{ij}$, con lo cual:
\begin{equation}
\begin{split}
\lambda\bullet{a_{ij}}=&\lambda\bullet(\mathcal{E}_{ij}\bullet\lambda^{ij})=(\lambda\bullet\mathcal{E}_{ij})\bullet\lambda^{ij}=(\mathcal{E}_{ij}\bullet\lambda)\bullet\lambda^{ij}=\\
=&\mathcal{E}_{ij}\bullet(\lambda\lambda^{ij})=\mathcal{E}_{ij}\bullet(\lambda^{ij}\lambda)=(\mathcal{E}_{ij}\bullet\lambda^{ij})\lambda=a_{ij}\bullet\lambda.
\end{split}
\end{equation}
\end{proof}

\begin{teo}
El álgebra $A$ es un $\mathbb{K}$-bimódulo libre de dimensión $k^{2}$ que tiene por base al conjunto $\{\mathcal{E}_{ij} : i,j\in\{1,...,k\}\}$ y es isomorfa al álgebra de matrices de $k\times{k}$ sobre $\mathbb{K}$.
\end{teo}
\begin{proof}
Como
\begin{equation}
A=\bigoplus_{i,j=1}^{k}A_{ij},
\end{equation}
todo $a\in{A}$ se puede expresar como una suma $a=\sum_{i,j=1}^{k}a_{ij}$, pero por el teorema previo se tiene que $a=\sum_{i,j=1}^{k}\mathcal{E}_{ij}\bullet\lambda^{ij}$. Luego $\{\mathcal{E}_{ij} : i,j\in\{1,...,k\}\}$ es un conjunto de generadores de $A$. Notemos que cada $\lambda^{ij}$ está dado por $\lambda^{ij}=\varphi_{j}^{-1}(\mathcal{E}_{ji}a\mathcal{E}_{jj}$) y que como $\mathcal{E}_{ij}\lambda=\lambda\mathcal{E}_{ij}$, entonces $a=\sum_{i,j=1}^{k}\mathcal{E}_{ij}\bullet\lambda^{ij}=\sum_{i,j=1}^{k}\lambda^{ij}\bullet\mathcal{E}_{ij}$; por lo tanto la base es como módulo izquierdo y derecho sobre $\mathbb{K}$.

Veamos que es el conjunto de generadores es linealmente independiente:
\begin{equation}
0=\sum_{i,j=1}^{k}\mathcal{E}_{ij}\bullet\lambda^{ij}=\sum_{i,j=1}^{k}\mathcal{E}_{ij}\varphi_{j}(\lambda^{ij}),
\end{equation}
multiplicando a izquierda por $\mathcal{E}_{ln}$ y a derecha por $\mathcal{E}_{ll}$ tenemos:
\begin{equation}
\begin{split}
0=&\sum_{i,j=1}^{k}\mathcal{E}_{ln}\mathcal{E}_{ij}\varphi_{j}(\lambda^{ij})\mathcal{E}_{ll}=\sum_{i,j=1}^{k}\delta_{ni}\mathcal{E}_{lj}\varphi_{j}(\lambda^{ij})\mathcal{E}_{ll}=\\
=&\sum_{j=1}^{k}\mathcal{E}_{lj}\varphi_{j}(\lambda^{nj})\mathcal{E}_{jj}\mathcal{E}_{ll}=\sum_{j=1}^{k}\delta_{jl}\mathcal{E}_{lj}\varphi_{j}(\lambda^{nj})\mathcal{E}_{jl}=\mathcal{E}_{ll}\varphi_{j}(\lambda^{nl})\mathcal{E}_{ll}=\varphi_{l}(\lambda^{nl})
\end{split},
\end{equation}
luego, como $\varphi_{l}$ es un isomorfismo para todo $l$ se tiene que $\lambda^{nl}=0$ para todos $n,l\in\{1,...,k\}$.
\end{proof}

En este punto de las notas procederemos a revisar y extender la definición de representación regular irreducible. Recordando la definición \ref{repreg}, la representación regular queda definida en principio como una representación \emph{real} del álgebra de Clifford, pues es la única estructura subyacente de módulo libre que se encuentra en el ideal $Af_{i}$. Sin embargo, acabamos de probar que estos ideales tienen, en el caso simple, una estructura de $\mathbb{K}$-módulo libre. Si $\mathbb{K}$ es $\mathbb{R}$ o $\mathbb{C}$, entonces tenemos en $Af_{i}$ una estructura de $\mathbb{K}$-espacio vectorial; sino simplemente la estructura de $\mathbb{K}$-módulo libre.

Utilizando esta estructura es que podemos identificar al álgebra de Clifford como un álgebra de matrices sobre $\mathbb{K}$. Podemos considerar el espacio de los endomorfismos sobre $Af_{i}$ como $\mathbb{K}$-módulo derecho, $\text{End}_{\mathbb{K}}(Af_{i})$. Sea $a\in{A}$ veamos que multiplicar a izquierda un elemento $\psi\in{Af_{i}}$ es uno de estos endomorfismos. En efecto:
\begin{equation}
a(\psi\lambda+\xi)=a\psi\lambda+a\xi=a(\psi)\lambda+a(\xi).
\end{equation}
Esto prueba que el álgebra $A$ está contenida en el álgebra de endomorfismos de $Af_{i}$ como $\mathbb{K}$-módulo derecho.

Entonces pasaremos a definir una variante de las representaciones regulares a fines de incluir la estructura de $\mathbb{K}$-módulo derecho.

\begin{defn}
Llamamos \emph{representación regular izquierda enriquecida sobre $\mathbb{K}$} a la representación regular izquierda de A, donde ahora el ideal $Af$ es un módulo libre derecho sobre $\mathbb{K}$ y pensamos a $A$ como un $\mathbb{K}$-bimódulo.
\end{defn}

Teniendo todas estas definiciones y resultados podemos responder a la cuestión sobre la equivalencia de representaciones mediante la siguiente proposición:

\begin{teo}\label{equivalrepsimp}
Dada un álgebra simple $A$, una descomposición $A=Af_1\oplus...\oplus{A}f_k$ asociada a un ccipo, las representaciones izquierdas regulares enriquecidas sobe $\mathbb{K}$ $(L_{i},Af_i)$ y $(L_{j},Af_j)$ son equivalentes para cualquier par $i,j$.
\end{teo}

\begin{proof}
Para demostrar que las representaciones $(L_{i},Af_{i})$ y $(L_{j},Af_{j})$ son equivalentes basta con encontrar un isomorfismo lineal $\phi$ entre los $\mathbb{K}$-módulos libres $Af_{i}$ y $Af_{j}$ y ver que $L_{j}(a)=\phi\circ{L_{i}(a)}\circ\phi^{-1}$ para todo $a\in{A}$.
Dicha transformación lineal existe y está dada por:
\begin{equation}
\phi(af_{i})=af_{i}\mathcal{E}_{ij}=a\mathcal{E}_{ii}\mathcal{E}_{ij}=a\mathcal{E}_{ij}\in{Af_{j}}.
\end{equation}
Claramente es lineal sobre $\mathbb{R}$, veamos que es $\mathbb{K}$-lineal. Sea $\lambda\in{\mathbb{K}}$,
\begin{equation}
\begin{split}
\phi(af_{i}\bullet{\lambda})=&\phi(a\mathcal{E}_{ii}\varphi_{i}(\lambda))=a\mathcal{E}_{ii}\varphi_{i}(\lambda)\mathcal{E}_{ij}=a\mathcal{E}_{ij}\mathcal{E}_{ji}\varphi_{i}(\lambda)\mathcal{E}_{ij}=\\
=&a\mathcal{E}_{ii}\mathcal{E}_{ij}\varphi_{j}(\lambda)=\phi(af_{i})\bullet{\lambda}
\end{split},
\end{equation}
por lo que es $\mathbb{K}$-lineal. 

Podemos ver que es un isomorphismo pues la siguiente función es su inversa:
\begin{equation}
\phi^{-1}(bf_{j})=bf_{j}\mathcal{E}_{ji}=b\mathcal{E}_{ji}\in{Af_{i}}.
\end{equation}
En efecto estas funciones son inversas la una de la otra:
\begin{equation}
\begin{gathered}
(\phi^{-1}\circ\phi)(af_{i})=\phi^{-1}(a\mathcal{E}_{ij})=a\mathcal{E}_{ij}\mathcal{E}_{ji}=af_{i} ,\\
(\phi\circ\phi^{-1})(bf_{j})=\phi^{-1}(b\mathcal{E}_{ji})=b\mathcal{E}_{ji}\mathcal{E}_{ij}=bf_{j}.
\end{gathered}
\end{equation}
Además también se cumple la regla de composición entre $L_{i}$ y $L_{j}$:
\begin{equation}
\phi\circ{L_{i}(a)}\circ\phi^{-1}(bf_{j})=\phi(ab\mathcal{E}_{ji})=ab\mathcal{E}_{ji}\mathcal{E}_{ij}=abf_{j}=L_{j}(a)(bf_{j}).
\end{equation}
Luego, estas representaciones son equivalentes y da lo mismo tomar cualquiera de los sumandos directos como espacio de representación.
\end{proof}

Este teorema nos dice que para construir la única representación posible (a menos de representaciones equivalentes) sólo hace falta encontrar un idempotente primitivo y no todo un conjunto completo.

En el caso de álgebras de Clifford se ve que esta álgebra de división siempre es $\mathbb{R}$, $\mathbb{C}$ o $\mathbb{H}$. Más aún, en el caso en que el álgebra de Clifford es simple, se puede encontrar una representación matricial de la misma, donde las entradas de las matrices están en el anillo de división correspondiente $\mathbb{K}=A_{ii}$. Dicha construcción se logra identificando cada elemento $\mathcal{E}_{ij}$ con la matriz $E_{ij}=\delta_{ij}$ y se da una forma explícita en el capítulo 4 de \cite{VazdaRocha}. Siempre existe esta representación matricial, lo que implica que los ideales minimales del álgebra son isomorfos a $\mathbb{K}^{k}$, con $k$ la dimensión de las matrices cuadradas, como se mencionó en el ejemplo \ref{ejemplomatriz}.

\subsubsection{Caso de suma directa de álgebras simples}\label{sumadirectasimples}

Las álgebras de Clifford, o bien son álgebras simples o bien son suma directa de dos álgebras simples. Decimos que una $\mathbb{K}$-álgebra $A$ es suma directa de dos álgebras unitarias $B$ y $C$ y notamos $A=B\oplus{C}$ si se cumple que $A=B\plus{C}$ con $BC=CB=\{0\}$ y por lo tanto $1_{A}=1_{B}+1_{C}$. A continuación probaremos un resultado importante para este tipo de álgebras.

\begin{prop}
Sea $A$ un álgebra unitaria que es suma directa de dos álgebras unitarias $A=B\oplus{C}$, si existe un ideal minimal $\mathcal{I}$, entonces está enteramente contenido en $B$ o enteramente contenido en $C$.
\end{prop}

\begin{proof}
Supongamos que $\mathcal{I}$ es ideal minimal izquierdo de $A$, pero que $\mathcal{I}\nsubset{B}$ e $\mathcal{I}\nsubset{C}$. Como $B\cap\mathcal{I}\neq\{0\}$, existe un ${a}\in{B}\cap\mathcal{I}$, con lo cual $Ba$ es un ideal de $A$ que además está contenido en $\mathcal{I}$, pues $Ba\subset{Aa}\subset{A}\mathcal{I}\subset\mathcal{I}$, por otro lado $a\in{Ba}$, con lo cual $Ba\neq\{0\}$ y además $Ba\neq{\mathcal{I}}$, pues $Ba\subset{B}$ y por hipótesis $\mathcal{I}\nsubset{B}$.

Así, llegamos a un absurdo y por lo tanto $\mathcal{I}$ no es minimal. Luego, todo ideal minimal de $A$ está enteramente contenido en $B$ o en $C$, como queríamos demostrar.
\end{proof}

Una consecuencia de este hecho es que una representación regular irreducible de toda el álgebra $A$ no será fiel, pues por ejemplo, si el espacio de representación es un ideal izquierdo contenido en $B$, $\mathcal{I}=Bf$, entonces $C\mathcal{I}=\{0\}$, con lo cual $C$ está contenido en el núcleo de la representación regular izquierda.

Dada un álgebra $A$ de este estilo, podemos obtener una representación fiel a partir de la suma directa de representaciones irreducibles. La idea es tomar un ideal minimal izquierdo $\mathcal{I}$ contenido en $B$ y otro $\mathcal{J}$ contenido en $C$. A partir de ellos se toma el conjunto $\mathcal{I}\oplus\mathcal{J}$ que es un ideal izquierdo (no minimal) y se utiliza la representación regular izquierda en este ideal. Dicha representación, si bien no es irreducible sí resulta ser fiel.

\begin{defn}
Dada $A=B\oplus{C}$ con $B$ y $C$ álgebras simples, entonces una representación regular izquierda dada por la definición \ref{repreg} se dice que es una \emph{representación regular izquierda fiel} si tiene por espacio de representación a $\mathcal{I}\oplus\mathcal{J}$ con $I$ ideal minimal izquierdo de $A$ contenido en $B$ y $\mathcal{J}$ ideal minimal izquierdo de $A$ contenido en $C$.
\end{defn}

Recordando los cálculos realizados para álgebras simples, supongamos que tenemos en $B$ y $C$ conjuntos completos de idempotentes primitivos ortogonales $\{f_1,...,f_{k}\}$ y $\{g_1,...,g_{l}\}$ respectivamente, entonces el conjunto $\{f_1,...,f_{k},g_{1},...,g_{l}\}$ será un ccipo de $A$ y tendremos la descomposición $A=Af_{1}\oplus...Af_{k}\oplus{A}g_{1}\oplus...\oplus{A}g_{l}=Bf_{1}\oplus...Bf_{k}\oplus{C}g_{1}\oplus...\oplus{C}g_{l}$. Un ideal del estilo mencionado en el párrafo anterior será $Bf_{i}\oplus{C}g_{j}$ para algún $i\in\{1,...,k\}$ y algún $j\in\{1,...,l\}$.

De manera similar a lo desarrollado en lo concerciente a álgebras simples uno podría preguntarse si da lo mismo tomar cualquier par de ideales $Bf_{i}$ y $Cg_{j}$. Nuevamente, la respuesta es que sí.

\begin{teo}
Sea $A=B\oplus{C}$, $B$ y $C$ álgebras simples y una descomposición: \[{A=Af_{1}\oplus...Af_{k}\oplus{A}g_{1}\oplus...\oplus{A}g_{l}}\] asociada a un ccipo de $A$, $\{f_1,...,f_{k},g_{1},...,g_{l}\}$ con $\{f_1,...,f_{k}\}$ ccipo de $B$ y $\{g_{1},...,g_{l}\}$ ccipo de $C$, entonces las representaciones regulares izquierdas (no irreducibles) $(L_{ij},Af_{i}\oplus{A}g_{j})$ y $(L_{rs},Af_{r}\oplus{A}g_{s})$ son equivalentes para todos $i,r\in\{1,...,k\}$, $j,s\in\{1,...,l\}$.
\end{teo}

\begin{proof}
Para demostrar que las representaciones $(L_{ij},Af_{i}\oplus{A}g_{j})$ y $(L_{rs},Af_{r}\oplus{A}g_{s})$ son equivalentes basta con encontrar un isomorfismo lineal $\phi$ entre los espacios vectoriales reales $Af_{i}\oplus{A}g_{j}$ y $Af_{r}\oplus{A}g_{s}$ y ver que $L_{rs}(a)=\phi\circ{L_{ij}(a)}\circ\phi^{-1}$ para todo $a\in{A}$.

Para construir este isomorfismo $\phi$, utilizaremos el teorema \ref{equivalrepsimp} sobre las álgebras simples $B$ y $C$.
Como $B$ y $C$ son álgebras simples, entonces tenemos los conjuntos $B_{pq}=f_{p}Bf_{q}$ para todos los $p,q\in\{1,...,k\}$ y los conjuntos $C_{u,v}=f_{u}Bf_{v}$ para todos los $u,v\in\{1,...,l\}$. Como tanto $B$ como $C$ son álgebras simples, utilizando lo dado por las ecuaciones \ref{algsimp1},\ref{algsimp2},\ref{algsimp3}, si definimos $\mathcal{E}_{pp}=f_{p}$ y $\mathcal{H}_{uu}=g_{u}$, se tiene que existen $\mathcal{E}_{pq}\in{B}_{pq}$ y $\mathcal{H}_{uv}\in{C}_{uv}$, tales que:
\begin{equation}
\begin{gathered}
\mathcal{E}_{pq}\mathcal{E}_{ot}=\delta_{qo}\mathcal{E}_{pt}, \ \ \forall{p,q,o,t}\in\{1,...,k\}\\
\mathcal{H}_{uv}\mathcal{H}_{xy}=\delta_{vx}\mathcal{H}_{uy}, \ \ \forall{u,v,x,y}\in\{1,...,l\}
\end{gathered}
\end{equation}

Definimos entonces $\phi:Bf_{i}\oplus{C}g_{j}\rightarrow{B}f_{r}\oplus{C}g_{s}$, por:
\begin{equation}
\phi(bf_{i}+cg_{j})=bf_{i}\mathcal{E}_{ir}+cg_{j}\mathcal{H}_{js}=b\mathcal{E}_{ir}+c\mathcal{H}_{js}\in{Bf_{r}\oplus{C}g_{s}}.
\end{equation}

Se puede ver fácilmente que es lineal. Además $\phi$ es un isomorfismo, pues su inversa está dada por:
\begin{equation}
\phi^{-1}(bf_{r}+cg_{s})=bf_{r}\mathcal{E}_{ri}+cg_{s}\mathcal{H}_{sj}=b\mathcal{E}_{ri}+c\mathcal{H}_{sj}\in{Bf_{i}\oplus{C}g_{j}}.
\end{equation}

Resta ver que $L_{rs}(a)=\phi\circ{L_{ij}(a)}\circ\phi^{-1}$. Sean $a\in{A}$, $b\in{B}$, $c\in{C}$, entonces:

\begin{equation}
\begin{gathered}
\phi(L_{ij}(a)\phi^{-1}(bf_{r}+cg_{s}))=\phi(L_{ij}(a)(b\mathcal{E}_{ri}+c\mathcal{H}_{sj}))=\\
=\phi(ab\mathcal{E}_{ri}+ac\mathcal{H}_{sj})=ab\mathcal{E}_{ri}\mathcal{E}_{ir}+ac\mathcal{H}_{sj}\mathcal{H}_{js}=\\
=abf_{r}+acg_{s}=a(bf_{r}+cg_{s})=L_{rs}(a)(bf_{r}+cg_{s}),
\end{gathered}
\end{equation}
con lo cual se cumple la igualdad y por lo tanto queda demostrado el teorema.
\end{proof}

En el caso puntual de algebras de Clifford, cuando ésta no es simple, las dos álgebras que forman parte de la suma directa son isomorfas entre sí; por ejemplo: $\text{Cl}_{0,3}\simeq{\mathbb{H}\oplus\mathbb{H}}$ y $\text{Cl}_{1,0}\simeq{\mathbb{R}\oplus\mathbb{R}}$. Más aún, los sumandos simples siempre son álgebras matriciales completas sobre $\mathbb{R}$, $\mathbb{C}$ o $\mathbb{H}$. Notemos entonces que si $A=M_{k}(\mathbb{K})\oplus{M}_{k}(\mathbb{K})$, entonces un elemento en la representación regular izquierda irreducible tiene dimensión $k$, pues está en $\mathbb{K}^{k}$ (ver ejemplo \ref{ejemplomatriz}), mientras que un elemento en la representación regular izquierda fiel tiene dimensión $2k$ pues pertenece a $\mathbb{K}^{k}\oplus\mathbb{K}^{k}$.

\begin{obs}
Notemos que el idempotente que genera al ideal $\mathcal{I}\oplus\mathcal{J}$ es, por ejemplo, $h=f_{1}+g_{1}$, que no es un idempotente primitivo. En este caso el conjunto $hAh$ está dada por la suma directa de álgebras de división, $f_{1}Bf_{1}\oplus{g}_{1}Cg_{1}=B_{11}\oplus{C}_{11}$. En este caso, el ideal $Ah=Bf_{1}\oplus{C}g_{1}$ puede ser dotado de una estructura lineal derecha sobre el \emph{anillo doble} $B_{11}\oplus{C}_{11}$, de la forma:
\begin{equation}
\begin{split}
(Bf_{1}\oplus{Cg_{1}})\times(B_{11}\oplus{C}_{11})&\rightarrow{A}h\\
(\varphi+\chi,\lambda+{\mu})&\mapsto(\varphi\lambda+\chi\mu)
\end{split}
\end{equation}
\end{obs}

\subsubsection{Álgebras de Clifford}

Podría decirse que el caso de las álgebras de Clifford, no es de los casos más complejos, puesto que las mismas son o bien álgebras simples o suma de dos álgebras simples solamente. Este hecho se puede ver a través de la \emph{clasificación de álgebras de Clifford}, que no se dará en este trabajo pero que puede consultarse en \cite{Gallier,Lounesto,VazdaRocha} y que es de central importancia en esta teoría. La clasificación demuestra que existe una periodicidad de orden $8$ en las álgebras; es decir, las álgebras de Clifford esencialmente distintas unas de otras son las que se obtienen para espacios vectoriales de dimensión a lo sumo $8$.

La construcción de las álgebras de Clifford como álgebras matriciales o suma directa de álgebras de ellas (pues son simples o semismiples), utiliza extensivamente las descomposiciones dadas por la ecuación \ref{descomp}. Sin embargo, no hemos probado hasta el momento que tal descomposición exista. El teorema que se da a continuación resulta de especial importancia en teoría de Clifford y trata sobre la existencia de un \emph{conjunto completo de idempotentes primitivos ortogonales} para álgebras de Clifford de dimensión y signatura arbitrarias. Nuevamente, se omitirá la demostración, pero la misma puede encontrarse en la página 111 de \cite{VazdaRocha} 

\begin{teo}\label{Cliffordccipo}
Sea $\text{Cl}_{p,q}$ el álgebra de Clifford asociada al espacio cuadrático real de signatura $(p,q)$ que notaremos por $\mathbb{R}^{p,q}$; y sea $\{e_1,...,e_n\}$ una base ortonormal de $\mathbb{R}^{p,q}$ con $n=p+q$, siempre existen $k$ elementos no escalares $e_{I_{j}}$ que forman parte de la base estándar del álgebra de Clifford asociada a $\{e_1,...,e_n\}$ que cumplen $e_{I_{j}}^2=1$ para todo $j$ y $e_{I_{i}}e_{I_{j}}=e_{I_{j}}e_{I_{i}}$ para todos $i,j$. Dado esto, los elementos de la forma:
\begin{equation}\label{idemprimit}
\frac{1}{2}(1{\pm}e_{I_{1}})...\frac{1}{2}(1{\pm}e_{I_{k}}),
\end{equation}
son idempotentes primitivos ortogonales. Tenemos entonces, combinando independientemente los signos $\pm$ en la ecuación un total de $2^{k}$ idempotentes. Además sumando todos estos idempotentes, se obtiene la identidad, con lo cual el conjunto de los $2^{k}$ idempotentes de la forma \ref{idemprimit} conforman un \emph{ccipo}. El número $k$ depende de el valor de $q-p$ y está dado por $k=q-r_{q-p}$ donde $r_{j}$ es el $j$-ésimo número de \emph{Radon-Hurwitz} de acuerdo a:

\begin{equation}
\begin{tabular}{l*{7}{c}r}
j & 0 & 1 & 2 & 3 & 4 & 5 & 6 & 7 \\
\hline
$r_{j}$ & 0 & 1 & 2 & 2 & 3 & 3 & 3 & 3 
\end{tabular}\ ,
\end{equation}
con la regla $r_{j+8}=r_{j}+4$.

\end{teo}

\subsection{Espinores algebraicos y aplicaciones a la física}\label{espinores}

Para introducir los espinores algebraicos utilizaremos un ejemplo simple como motivación y luego generalizaremos a todas las álgebras de Clifford.

Consideremos el grupo $\text{Spin}(2,0)$ en el álgebra de Clifford $\text{Cl}_{2,0}\neq\text{Cl}_{2}$. Tenemos del ejemplo \ref{Cl20} que $\text{Spin}(2,0)\cong\text{U}(1)=\{\alpha{1}+\beta{k} \ | \ \alpha^2+\beta^2=1\}\subset{\mathbb{H}}$, entonces podemos escribir a $\alpha$ y $\beta$ como $\alpha=\cos{\theta}$, $\beta=\sin{\theta}$. Con lo cual $x=\cos{\theta}1+\sin{\theta}k$, $\alpha(x)=\cos{\theta}\alpha(1)+\sin{\theta}\alpha(k)=\cos{\theta}1+\sin{\theta}k$, donde usamos que $k=e_1e_2$, con lo cual $\alpha(k)=k$ con lo cual tenemos que $x^{-1}=\overline{x}=\cos{\theta}1-\sin{\theta}k$. 

Si consideramos un $v\in\mathbb{V}$, entonces $v=ai+bj$ para algún par $a,b$ de números reales y la acción adjunta de x sobre este elemento toma la forma:
\begin{equation}
\begin{gathered}
\rho_{x}(v)=(\cos{\theta}1+\sin{\theta}k)(ai+bj)(\cos{\theta}1-\sin{\theta}k)=\\
=(\cos{\theta}1+\sin{\theta}k)[(a\cos{\theta}-b\sin{\theta})i+(a\sin{\theta}+b\cos{\theta})j]=\\
=(a(\cos^{2}\theta-\sin^{2}\theta)-2b\cos{\theta}\sin{\theta})i+(a(\cos^{2}\theta-\sin^{2}\theta)+2a\cos{\theta}\sin{\theta})j.
\end{gathered}
\end{equation}

Utilizando las identidades del seno y coseno de la suma de ángulos, tenemos que:
\begin{equation}
\begin{gathered}
\cos^{2}\theta-\sin^{2}\theta=\cos{2\theta},\\
2\cos{\theta}\sin{\theta}=\sin{2\theta},
\end{gathered}
\end{equation}
con lo que se obtiene:
\begin{equation}
\rho_{x}(v)=(a\cos{2\theta}-b\sin{2\theta})i+(b\cos{2\theta}+a\sin{2\theta})j.
\end{equation}
Esto es una rotación de ángulo $2\theta$ sobre el vector $ai+bj=ae_1+be_2$, lo que está en concordancia con el hecho de que $\rho_{x}\in\text{SO}(2)$.

Por otro lado recordemos que el álgebra de Clifford $\text{Cl}_2$ es isomorfa al álgebra de matrices reales de $2\times{2}$, $M_{2}(\mathbb{R})$. A su vez $M_{2}{(\mathbb{R})}$ es isomorfa al álgebra de endomorfismos de un espacio vectorial real $S$ de dimensión dos, $\text{End}{(S)}$, con lo cuál podemos preguntarnos cómo actúa el elemento $x=\cos{\theta}1+\sin{\theta}k$ sobre algún elemento de $S$.

Utilizando la ecuación \ref{Cl2_matriz}, vemos que el elemento $x$ cumple:
\begin{equation}
x= \begin{pmatrix}
		\cos{\theta} \ & -\sin{\theta}\\
		 \sin{\theta} \ & \cos{\theta}\\
	       \end{pmatrix}.
\end{equation}
Visto como las componentes de un endomorfismo de $S$ en cierta base $\{s_1,s_2\}$ esto es una \emph{rotación de ángulo $\theta$}. Esta rotación de ángulo $\theta$ en el espacio vectorial real de dimensión dos, da lugar vía la acción adjunta a una rotación de ángulo $2\theta$ en el espacio vectorial $\mathbb{V}$ de partida.

Esto es equivalente a decir que el elemento $x$ actúa en un elemento $\psi\in{S}$ a través de la representación regular de $\text{Cl}_{2,0}$ explicada en la sección \ref{repreg}. También se dice que $\psi$ es un elemento de la representación regular izquierda de $\text{Cl}_{2,0}$ y que $S$ es el espacio de representación de la acción regular izquierda de $\text{Cl}_{2,0}$. En este caso decimos que un elemento en $S$ es un \emph{espinor} y que $S$ es el espacio de espinores. En dimensiones superiores resulta que hay que tener más cuidado y daremos una definición más precisa para los espinores; sin embargo, la idea detrás siempre es la de trabajar con una representación regular irreducible del álgebra de Clifford.

Las álgebras de Clifford son siempre semisimples, puesto que son isomorfas a álgebras de matrices sobre $\mathbb{R}$, $\mathbb{C}$ o $\mathbb{H}$ o sumas directas de este tipo de álgebras. Al definir los espinores se hace una diferenciación para álgebras simples o no simples. 

\begin{defn}
Dada un álgebra de Clifford $\text{Cl}(\Phi)$ simple llamamos \emph{espacio de espinores} o \emph{espacio espinorial} al espacio de representación de la representación regular izquierda irreducible del álgebra con la estructura adiciónal de $\mathbb{K}$-módulo derecho mencionada en la observación \ref{Kmoduloder}. Equivalentemente podemos decir que el espacio de espinores es un ideal minimal izquierdo del álgebra con dicha estructura adicional de $\mathbb{K}$-módulo derecho. Un elemento en este espacio se denomina \emph{espinor}

Si en cambio el álgebra de Clifford es suma directa de álgebras simples, dicho espacio se llama \emph{espacio de semi-espinores} o \emph{espacio semi-espinorial} y un elemento en el espacio se denomina \emph{semi-espinor}.
\end{defn}

Esta diferenciación en las definiciones se basa en que si un álgebra es semisimple, un ideal minimal está completamente contenido en alguna de las componentes simples del álgebra, y debido al análisis de la sección \ref{sumadirectasimples}, la representación irreducible no resulta ser fiel. En este sentido, se necesitan de dos semi-espinores en el caso semisimple para obtener una representación fiel.

Notemos entonces que un vector en el espacio vectorial original $\mathbb{V}$ tiene $n$ componentes \emph{reales}, mientras que un espinor es un objeto de $k$ componentes (según corresponda por el teorema \ref{Cliffordccipo}), cada una de ellas en el álgebra de división $\mathbb{K}$ que corresponda. Por ejemplo, si consideramos $\text{Cl}_{3,0}$ un vector tiene tres componentes reales, pero como $\text{Cl}_{3,0}\simeq{M_{2}\mathbb{C}}$, entonces un espinor asociado a este espacio cuadrático tiene dos componentes complejas.

\subsubsection{Aplicaciones}

A partir de lo recopilado podemos mencionar algunas aplicaciones a la física de esta teoría. Hemos visto que a partir de un espacio cuadrático de {\bf{vectores}} se construye un espacio de nuevos objetos denominados {\bf{espinores}}. Estos objetos son de naturaleza distinta, pues dada una transformación isométrica ($\in\text{SO}(\Phi)$) en el espacio de vectores, mientras éstos transforman de acuerdo a la representación adjunta del grupo $\text{Spin}$, los espinores lo hacen de acuerdo a la representación regular izquierda del mismo grupo.

Tanto en la teoría de mecánica cuántica con espín (teoría de Pauli), como en la teoría de Dirac del electrón, surgen naturalmente los espinores a través del álgebra de Clifford. Una propiedad importante de estos objetos es que al efectuar una rotación de ángulo $2\pi$ sobre el sistema, el mismo sufre una rotación sólo de $\pi$ en el espacio interno\footnote{Esto se da cuando se definen las rotaciones a partir de la exponencial del álgebra de Lie del grupo SO. Podría evitarse el signo negativo usando la exponencial multiplicada por -1}. Esto es lo que sucede en el ejemplo de la sección \ref{espinores}, donde una rotación de ángulo $2\theta$ en el espacio de vectores de dimensión dos, resulta en una rotación de ángulo $\theta$ para el espinor en cuestión.

Una cuestión a observar es por qué se da esta relación entre la representación adjunta y la representación regular del grupo Spin o de algún grupo en general. Pensemos qué pasa cuando tenemos un espacio vectorial de dimensión finita y el espacio de operadores lineales en él. Consideremos el vector dado por $A{u}$, donde $A$ es un operador lineal y $u$ un vector. Supongamos que aplicamos cierta transformación lineal inversible $S$ a todos los vectores del espacio. En ese caso, se obtienen nuevos vectores y sucede que $\nobreak{A(S{u})\neq{A}{u}}$. Supongamos entonces que a la vez que aplicamos $S$ a todos los vectores, también hacemos una transformación de los operadores $A\mapsto{A'}$, de modo de asegurar que $A'(Su)=A{u}$. Es fácil ver que $A'=SAS^{-1}$. Entonces, mientras los vectores transformaron con la representación regular izquierda del grupo de isomorfismos, los operadores lo hacen a través de la representación adjunta. Lo mismo sucede si se ve a esta transformación desde el enfoque en el cual $S$ es una matriz de cambio de base.

Tanto en teoría de Pauli como en la de Dirac son todos los elementos del álgebra de Clifford los que operan sobre los espinores, no solo los del grupo Spin, con lo cual podría pensarse que hay transformaciones de los espinores que no se corresponden con la acción adjunta de ningún elemento en Spin. Sin embargo estas transformaciones no son invertibles y por lo tanto no representan por ejemplo un cambio de base o ningún cambio de observador.

Analicemos brevemente la llamada ecuación de Dirac. Esta es una ecuación diferencial lineal introducida por Dirac en 1928 \cite{Dirac} cuya solución corresponde al estado cuántico de un electrón libre relativista. La misma está dada por:
\begin{equation}
\Bigg{(}i\gamma_{\mu}\frac{\partial}{\partial{x_{\mu}}}-m\mathbf{1}\Bigg{)}\psi(x^{0},x^{1},x^{2},x^{3})=0.
\end{equation}
Aquí, $m$ es la masa del electrón y cada 4-upla $(x^{0},x^{1},x^{2},x^{3})$ representa un punto del espacio-tiempo con $x^{0}$ una coordenada temporal y $x^{1},\ x^{2}$ y $x^{3}$ tres coordenadas espaciales. Las matrices $\gamma_{\mu}$  con $\mu\in\{0,...,3\}$ y $\mathbf{1}$ (la identidad) son los generadores de la representación regular fiel del álgebra de Clifford del espacio-tiempo $\text{Cl}_{1,3}$ y $\psi(x^{0},...,x^{3})$ es una función que asigna a cada punto del espacio tiempo $(x_{0},...,x_{3})$ un espinor algebraico, es decir, un \emph{campo espinorial}. En este contexto el campo se puede ver como un vector columna de cuatro funciones complejas. Notemos que también aparece una $i$ imaginaria; esto se debe a que el álgebra de Clifford se ha hecho compleja (no se dio en estas notas pero puede consultarse en \cite{Gallier}), hecho que por ahora ignoraremos. Otra cuestión nueva respecto a lo desarrollado en las notas es que el espinor no es un elemento único sino que es un campo espinorial. Supongamos ahora que el campo espinorial $\psi$ es autovalor de todos los operadores derivada $\frac{\partial}{\partial{x_{\mu}}}$, de modo que:
\begin{equation}
\frac{\partial{\psi}}{\partial{x_{\mu}}}=ip^{\mu}\psi \ ,
\end{equation}
donde $p^{\mu}$ es un número real para cada número $\mu$. Entonces la ecuación de Dirac se escribe:
\begin{equation}
(\gamma_{\mu}p^{\mu}-m\mathbf{1})\psi(x^{0},...,x^{3})=0.
\end{equation}
En este caso, para cada punto fijo del espacio-tiempo tenemos una ecuación en un álgebra de Clifford, pues la cantidad entre paréntesis es un elemento del álgebra de Clifford fijo y la solución es un espinor. Una cuestión que vale la pena mencionar es que existe una forma alternativa a la ecuación de Dirac que es la ecuación de Dirac-Hestenes \cite{Hestenes}. Ésta evita la introducción de números complejos y para introducirla es esencial trabajar con idempotentes en el álgebra de Clifford.

\section{Conclusiones}

A lo largo de estas notas hemos introducido el álgebra de Clifford asociada a un espacio cuadrático, analizado su relación con las isometrías asociadas a dicho espacio cuadrático y los grupos Pin y Spin. Por último introducimos los espinores algebraicos y mencionamos aplicaciones a la física.

Quedan cuestiones muy importantes sobre las álgebras de Clifford sin explorar y que pretendo incorporar en un futuro. Algunas de estas son la clasificación de álgebras de Clifford y su 8-periodicidad, aspectos de teoría de Lie sobre los grupos Pin, Spin y el grupo de Clifford-Lipschitz y algunos aspectos topológicos dados por el hecho de que el grupo Spin (Pin) es un cubrimiento doble de SO (O).

\newpage
\appendix{}
\section{Apéndice: Bases ortogonales en espacios cuadráticos regulares}\label{apA}

\textbf{Proposición.} Si b es una métrica no definida, entonces existe al menos un vector de tipo luz distinto de cero.

\begin{proof}
Como $b$ no es definida, existe $v\neq{0}$ tal que $b(v,v)\leq{0}$ y $w\neq{0}$ tal que $b(w,w)\geq{0}$.
Si consideramos $z=\alpha{v}+(1-\alpha)w$ con $\alpha \in [0,1]$, entonces:\\
Si $v$ y $w$ son l.d. $\rightarrow$ $v=\beta{w}$ $\Rightarrow$ $b(v,v)=\beta^{2}b(w,w)$, pero $b(v,v)\leq{0}$ y $b(w,w)\geq{0}$, entonces la única solución es $b(v,v)=b(w,w)=0$ y ambos son vectores nulos.\\
Si $v$ y $w$ son l.i. entonces: $b(z,z)=\alpha^2b(v,v)+2\alpha(1-\alpha)b(v,w)+(1-\alpha)^{2}b(w,w)=f(\alpha)$. f es un polinomio de grado dos en $\alpha$ con lo cual es una función continua en [0,1]. Por otro lado $f(0)=b(w,w)\geq{0}$ y $f(1)=b(v,v)\leq{0}$, con lo cual existe un valor $\alpha_0$ para el cual $f(\alpha_0)=0$. pero $f(\alpha_0)=b(z,z)=0$, con lo cual $z=\alpha_0{v}+(1-\alpha_0)w$ es el vector nulo buscado.
\end{proof}

\textbf{Teorema.} Para toda forma bilineal b en un espacio vectorial V existe una base ortogonal. Además el número de componentes diagonales negativos, positivos y nulos con respecto a cualquier base son los mismos y por lo tanto son invariantes de la forma bilineal b.

\begin{proof}
La demostración se hará por inducción sobre la dimensión del espacio vectorial $V$ sobre el cual la forma actúa. Sea $d=dim(V)$.
\begin{itemize}
\item $d=1$. En este caso puede suceder que $b=0$ y entonces cualquier base es ortonormal. Si $b\neq{0}$ existe $f_1 \in V$ tal que $b_{11}=b(f_1,f_1)\neq{0}$. Para ortonormalizar la base basta tomar $e_1=\frac{f_1}{\sqrt{|b_{11}|}}$, con lo cual $b(e_1,e_1)=sgn(b_{11})=\pm{1}$.
\item Supongamos que toda forma bilineal simétrica en un espacio $(d-1)$-dimensional o menor admite una base ortogonal y tenemos una forma bilineal b en un e.v. $V$ de dimensión $d$.\\
Si $b=0$ entonces cualquier base es ortonormal. Si $b\neq{0}$ existe al menos un vector $v$ tal que $b(v,v)\neq{0}$. Mas aun, hay vectores $v$ y $w$ tales que $b(v,w)\neq{0}$ y si además $b(v,v)=0$ y $b(w,w)=0$ tenemos que:
\begin{equation}
b(v+w,v+w)=2b(v,w)\neq{0},
\end{equation}
con lo cual en general la suma de vectores nulos no es un vetor nulo.\\
Dado un vector $v$ tal que $b(v,v)=\alpha\neq{0}$ definimos $e_d=\frac{1}{\sqrt{|\alpha|}}v$ $\Rightarrow$ $b(e_d,e_d)=sgn(\alpha)=\pm{1}$.
Cosideremos ahora el conjunto $W=v^{\perp}=\{w \in V : b(v,w)=0\}$ de todos los vectores ortogonales a v, llamado ``v-ortogonal'' o ``v-perpendicular''. Veamos que $W$ es un subespacio vectorial. Sean $w_1,w_2 \in W$ y $\alpha in \mathbb{R}$ entonces:
\begin{equation}
\begin{split}
b(v,\alpha{w_1})=\alpha{b(v,w_1)}=&0 \Rightarrow \alpha{w_1} \in W, \\
b(v,{w_1+w_2})=b(v,w_1)+b(v,w_2)=&0+0=0 \Rightarrow {w_1+w_2} \in W. \\
\end{split}
\end{equation}
Luego $W$ es un subespacio de $V$. Más aún, $W\neq{V}$ pues dado que $b(v,v)\neq{0}$, $v\notin{W}$. Entonces $dim{W}=k<d$. Además la resitrcción de $b$ a $W$ es una forma bilineal simétrica en $W$ y por lo tanto, por hipótesis inductiva, $W$ admite una base ortonormal $\{e_1,...,e_k\}$ con $b(e_i,e_j)=a_i\delta_{ij}$ y $a_i \in \{-1,0,1\}$ $\forall i,j \in \{1,..,k\}$
Probaremos que el conjunto $\{e_1,...,e_k,e_d\}$ es una base ortonormal de V. La ortonormalidad para $\{e_1,...,e_k\}$ y $e_d$ por separado ya está probada, resta ver qué sucede con $b(e_d,e_j)$ para un $j \in \{1,...,k\}$ arbitrario. Recordemos que $e_d=\frac{1}{\beta}v$ y por lo tanto $b(e_d,e_j)=b(\frac{1}{\beta}v,e_j)=\frac{1}{\beta}b(v,e_j)=0$ pues $e_j \in W$. Luego el conjunto es ortonormal. Resta ver que $\{e_1,...,e_k,e_d\}$ genera a todo V (y en particular que $k=d-1$).\\
Sea $x \in V$ definimos $\alpha=b(e_d,e_d)b(x,e_d)=\pm{b(x,e_d)}$, entonces:
\begin{equation}
\begin{split}
b(x-\alpha{e_d},v)=&b(x,v)-\alpha{b}(e_d,v)=\\
=\beta{b}(x,e_d)-\alpha\beta{b}(e_d,e_d)&=\beta{b}(x,e_d)(1-b(e_d,d_d)^2)=0,
\end{split}
\end{equation}
con lo cual $x-\alpha{e_d} \in W$ y podemos escribir:
\begin{equation}
x-\alpha{e_d}=\sum_{i=1}^{k}{c^ie_i} \Rightarrow x=\sum_{i=1}^{k}{c^ie_i}+\alpha{e_d}.
\end{equation}
Así, cada $x \in V$ se escribe como combinación lineal de $\{e_1,...,e_k,e_d\}$ de manera única, con lo cual $\{e_1,...,e_k,e_d\}$ es una base de $V$ y por lo tanto $k=d-1$.
\end{itemize}
Resta probar que el número de componentes diagonales negativas, positivas y nulas $b(e_i,e_i)=a_i$ es igual en cualquier base. Para ello daremos caracterizaciones independientes de la base de estos números.
\begin{enumerate}
\item El número de elementos $a_i$ que son nulos es la dimensión del siguiente subespacio:
\begin{equation}
N=\{w \in V : b(v,w)=0 \forall v \in V\}.
\end{equation}
Es simple ver que es un subespacio. Se llama \emph{espacio nulo de b}. Los $e_i$ correspondientes a los $a_i$ conforman una base de $N$. En efecto, sea $w\in{N}\subset{V}$, se puede escribir $w=\sum_{i=1}^d{w^ie_i}$. Como $w \in N$, $b(v,w)=0 \forall v \in V$
\begin{equation}
b(v,w)=w^{i}b(v,e_i)=0 \ \ \forall v\in{V}.
\end{equation}
Si tomamos $v=e_j j\in\{1,...,d\}$, tenemos:
\begin{equation}
b(e_j,w)=w^{i}b(e_j,e_i)=w^{i}a_j\delta_{ij}=w^{i}a_i=0,
\end{equation}
de donde vemos que si $a_i=0$ entonces $w^i$ puede tener algún valor no nulo, pero si $a_i\neq{0}$ entonces $w^{i}$ necesariamente es nulo, con lo cual el vector $w \in {N}$ se escribe:
\begin{equation}
w=\sum_{j:a_j=0}{w^je_j}.
\end{equation}
Luego dado que la dimensión de un subespacio es independiente de la base y que un número finito de elementos de toda base ortogonal generan a $N$, el número de elementos diagonales nulos es el mismo en cualquier base ortogonal de $b$.
\item El número de elementos $a_i$ que valen $1$ es la dimensión de un \emph{subespacio maximal de definición positiva de $b$}.\\
Dichos subespacios no son únicos a menos que $b$ sea definida positiva o semidefinida negativa. Sin embargo, entre todos los espacios de definición positiva de $b$ habrá alguno de dimensión máxima.\\
Sea $W$ uno de estos esoacios de dimensión maximal con $s=\dim{W}$, sea $\{e_i\}$ una base ortonormal de dicho espacio, ordenada de modo tal que $a_1=...=a_k=1$ y $a_i\leq{0}$ $\forall \ i>k$ y sea $X$ el subespacio generado por $\{e_1,...,e_k\}$, todo $v \in X$ se escribe:
\begin{equation}
v=\sum_{i=1}^{s}{v^ie_i} \text{ con } v^i=0 \forall i>k.
\end{equation}
Además:
\begin{equation}
b(v,v)=\sum_{i=1}^{k}(v^i)^2\geq{0},
\end{equation}
entonces el espacio $X$ es un espacio de definición positiva de $b$ y como $W$ es un espacio maximal de definición positiva, tendremos $\dim{W}\geq{\dim{X}}=k$
Por otro lado definimos la transformación lineal $A:W\rightarrow{X}$ que asigna a cada $w=\sum_{i}w^ie_i \in{W}$ le asigna el vector en $X$, $Aw=\sum_{i=1}^k{w^ie_i}$.
Entonces los elementos $w$ en el núcleo de $A$, $Nu(A)$, son los que tienen $w^i=0 \forall i\leq{k}$. Además tenemos:
\begin{equation}
b(w,w)=w^iw^jb(e_i,e_j)=w^iw^ja_i\delta_{ij}=\sum_{i=1}^{s}(w^i)^2a_i=\sum_{i=k+1}^{s}a_i(w^i)^2\leq{0}
\end{equation}
Entonces $Nu(A)$ sólo es un espacio de definición positiva de $b$ si $Nu(A)=\{0\}$. Pero como por hipótesis $W$ es un espacio de definición positiva para $b$, necesariamente $Nu(A)=\{0\}$ y $A$ es monomorfismo. Además $\Im(W)=X$ con lo cual $A$ es un isomorfismo y por lo tanto $\dim{W}\leq{\dim{X}}=k$. Pero previamente habíamos concluído que $\dim{W}\geq{\dim{X}}=k$, entonces tenemos que $\dim{W}=k$ y el número de autovalores positivos en la matriz queda caracterizado por la dimensión de un subespacio.
\item El número de $a_i$ negativos determina la dimensión de un\emph{subespacio maximal de definición negativa de $b$}. La prueba es muy similar a la del caso de $a_i$ postivos
\end{enumerate} 
\end{proof}

\end{document}